\documentclass[11pt,french]{smfart}
\usepackage[latin1]{inputenc}
\usepackage[T1]{fontenc}
\usepackage[english,french]{}
\usepackage{hyperref}
\usepackage{amssymb,amsmath,url,xspace,smfthm,euscript,enumerate}
\usepackage{fancyhdr}
\usepackage{graphicx}
\def\notdiv{\nmid}
\def\div{\,\vert\,}
\def\too{\relbar\lien\rightarrow}
\def\tooo{\relbar\lien\relbar\lien\too}
\let\ds=\displaystyle

  \def\N{\mathbb{N}}
  
  \def\C{\mathbb{C}}
  \def\Q{\mathbb{Q}}
  \def\Z{\mathbb{Z}}
  \def\F{\mathbb{F}}

 \def\zz{\mathbb{Z}}
\def\virg{\raise 2pt \hbox{,}\,\,}
\def\rond{{\scriptstyle\circ}}

\def\No{{\rm N}}
\def\Sauf{\!\setminus\!}

\def\Cl{{\mathcal C}\hskip-2pt{\ell}}
\def\Pl{{\mathcal P}\hskip-2pt{\ell}}

\def\Frac#1#2{\hbox{\footnotesize $\displaystyle \frac{#1}{#2}$}}
\def\plus{\ds\mathop{\raise 2.0pt \hbox{$\bigoplus $}}\limits}
\def\mult{\ds\mathop{\raise 2.0pt \hbox{$\bigotimes$}}\limits}
\def\prd{ \ds\mathop{\raise 2.0pt \hbox{$  \prod   $}}\limits}
\def\Cap{ \ds\mathop{\raise 2.0pt \hbox{$\bigcap   $}}\limits}
\def\Cup{ \ds\mathop{\raise 2.0pt \hbox{$\bigcup   $}}\limits}
\def\sm{  \ds\mathop{\raise 2.0pt \hbox{$  \sum    $}}\limits}
\def\ev{\emptyset}

\def\fin{\vbox{\hrule\hbox to 7.2pt{\vrule height 7pt\hfil\vrule}\hrule}}
\def \tensorZp{\otimes{\raise -0.8pt \hbox{\!\!$_{_{\zz_{\!p}}}$}}}
\def \tensorZ{\otimes{\raise -0.8pt \hbox{\!\!$_{_{\zz}}$}}}
\def\fin{\vbox{\hrule\hbox to 7.2pt{\vrule height 7pt\hfil\vrule}\hrule}}
\def\lien{\mathrel{\mkern-4mu}}
\def\wt{\widetilde}
 
\def\wh{\widehat}

\newenvironment{example}{\begin{enonce}{Exemple}}{\end{enonce}}

\theoremstyle{plain}

\begin{document}
 
 \title{Les $\theta$-r\'egulateurs locaux \\ d'un nombre alg\'ebrique \\ Conjectures $p$-adiques} 

\author[Georges  Gras ]{Georges  {\sc Gras}}

\curraddr{Villa la Gardette, chemin Ch\^ateau Gagni\`ere,  F--38520 Le Bourg d'Oisans.}

\email{g.mn.gras@wanadoo.fr $\bullet$ {\it url : }\url{http://monsite.orange.fr/maths.g.mn.gras/}}

\begin{abstract}
Soit  $K/\Q$ une extension Galoisienne finie de degr\'e $n$, de groupe de Galois $G$, et soit $\eta \in K^\times$.
Pour tout $p$ premier assez grand (\'etranger \`a $n$, ${\rm Disc}(K)$, et $\eta$),
nous d\'efinissons, par utilisation du th\'eor\`eme de Frobenius sur les d\'eterminants de groupes, la famille 
$\big(\Delta_p^\theta (\eta) \in \F_p \big)_\theta$ des $\theta$-r\'egulateurs locaux de~$\eta$,
index\'ee par les caract\`eres $\Q_p\,$-irr\'eductibles $\theta$ de $G$. 

\smallskip
A chaque $\Delta_p^\theta (\eta)$ est associ\'ee une repr\'esentation lin\'eaire ${\mathcal L}^\theta \simeq \delta\, V_\theta$, $0\leq \delta \leq \varphi(1)$
(o\`u $V_\theta$ est la repr\'esentation de caract\`ere  $\theta$ et $\varphi \div \theta$ absolument irr\'eductible)~; la repr\'esentation ${\mathcal L}^\theta$ caract\'erise des propri\'et\'es de $\Delta_p^\theta (\eta)$, dont  sa nullit\'e \'equivalente \`a $\delta \geq 1$ (Th\'eor\`eme \ref{theo24},\,\S\ref{sub15}).

\smallskip
Lorsque $\eta\in\Q^\times$ et $\theta = 1$, $\Delta_p^1 (\eta)$ est le $p$-quotient de Fermat de $\eta$. Lorsque~$\eta$ est une ``unit\'e de Minkowski'' (pour $K$ r\'eel), chaque $\Delta_p^\theta (\eta)$, $\theta \ne 1$, donne le r\'esidu modulo $p$ de la 
$\theta$-composante ${\rm Reg}_p^\theta (\eta)$ dans la factorisation $\prod_{\theta\ne 1}({\rm Reg}_p^\theta (\eta))^{\varphi(1)}$
de $p^{1-n}  {\rm Reg}_p(K)$, o\`u ${\rm Reg}_p(K)$ est  le r\'egulateur $p$-adique classique de $K$.

\smallskip
Nous sugg\'erons, \`a partir d'une propri\'et\'e g\'en\'erale des d\'eterminants de groupes et de l'existence de la repr\'esentation 
${\mathcal L}^\theta$, que la ``proba\-bilit\'e'' de nullit\'e de $\Delta_p^\theta(\eta)$ avec ${\mathcal L}^\theta \simeq \delta\, V_\theta$ est en $\frac{O(1)}{p^{f \delta^2}}$ (cf.\,\S\ref{HP}), o\`u $f$ est le degr\'e r\'esiduel de $p$ dans le corps des valeurs des caract\`eres  $\varphi \div \theta$, et que les $\Delta_p^\theta(\eta)$ sont des variables ind\'ependantes.

\smallskip
Nous conjecturons alors que $p^{1-n} {\rm Reg}_p(K)$, qui mesure l'ordre du $p$-groupe de torsion en $p$-ramification Ab\'elienne au-dessus de $K$, est pour $p$ assez grand une unit\'e $p$-adique sauf peut-\^etre pour un ensemble de nombres premiers $p$ de densit\'e nulle. Pour ces cas dits ``de $p$-divisibilit\'e minimale $p^{\varphi(1)}$''  (i.e.,  un unique $\Delta_p^\theta (\eta)$, tel que $\delta= f=1$, est nul), il reste  possible, $\eta$ \'etant alors  une ``puissance $p$-i\`eme locale partielle'' en $p$, de proposer, en lien avec  la conjecture $ABC$, une conjecture plus forte aboutissant \`a la m\^eme conclusion pour tout $p$ assez grand (Section 7). 
D'autres aspects conjecturaux sur le quotient de Fermat sont discut\'es.

\smallskip
Nous pr\'ecisons et v\'erifions ces propri\'et\'es par des \'etudes num\'eriques portant sur des corps cubiques et quintiques cycliques et sur un corps non Ab\'elien (groupe~$D_6$), et publions les programmes ``PARI'' correspondants.

\vspace{1.0cm}
\end{abstract}
\begin{altabstract} Let $K/\Q$ be a finite Galois extension of degree $n$, of Galois group $G$, and let $\eta\in K^\times$.
For all large enough prime~$p$ (prime to $n$, ${\rm Disc}(K)$, and $\eta$),
we define, by use of the Frobenius theorem on group determinants, the family  
$\big(\Delta_p^\theta(\eta) \in \F_p \big)_\theta$ of local $\theta$-regulators of  $\eta$,  indexed by the $\Q_p\,$-irreducible characters $\theta$ of $G$. 

\smallskip
At each  $\Delta_p^\theta (\eta)$ is associated a linear representation ${\mathcal L}^\theta\simeq \delta \, V_\theta$, $0\leq \delta  \leq \varphi(1)$ 
(where $V_\theta$ is the  representation of character $\theta$ and $\varphi \div \theta$ absolutely irreducible); the representation ${\mathcal L}^\theta$ characterizes some properties of $\Delta_p^\theta (\eta)$, including its nullity equivalent~to~$\delta \geq 1$ 
(Theorem \ref{theo24},\,\S\ref{sub15}).

\smallskip
 When $\eta\in\Q^\times$ and $\theta = 1$, $\Delta_p^1 (\eta)$ is the $p$-Fermat quotient of $\eta$. When~$\eta$ is a ``Minkowski unit'' (for $K$ real), each $\Delta_p^\theta (\eta)$,  $\theta \ne 1$, gives the residue modulo $p$ of the $\theta$-component 
${\rm Reg}_p^\theta (\eta)$ in the factorization $\prod_{\theta\ne 1}({\rm Reg}_p^\theta (\eta))^{\varphi(1)}$ of
$p^{1-n}  {\rm Reg}_p(K)$, where ${\rm Reg}_p(K)$ is the classical $p$-adic regulator  of $K$.

\smallskip
 We suggest, starting from a general property of group determinants and the existence of the representation ${\mathcal L}^\theta$, 
that  the ``probability'' of nullity of $\Delta_p^\theta(\eta)$ with ${\mathcal L}^\theta \simeq \delta \, V_\theta$ is about $\frac{O(1)}{p^{f \delta^2}}$ (cf.\,\S\ref{HP}), where $f$ is the residue degree of $p$ in the field of values of the characters $\varphi \div \theta$, and that the $\Delta_p^\theta(\eta)$ are independent variables.

\smallskip
We conjecture that $p^{1-n} {\rm Reg}_p(K)$, which measures the order of the $p$-torsion group in Abelian $p$-ramification over $K$, is for $p$ large enough a $p$-adic unit except perhaps for a set of prime numbers of zero density. For these cases said 
``of minimal $p$-divisibility $p^{\varphi(1)}$'' (i.e.,  a single $\Delta_p^\theta (\eta)$, such that $\delta = f = 1$, is zero),  it remains  possible, $\eta$ being then a ``partial local $p$th power'' at $p$, to propose, in connection with the $ABC$ conjecture, a stronger conjecture leading to the same conclusion for all large enough $p$ (Section 7). Some other conjectural aspects on the Fermat quotient are discussed.

\smallskip
We precise and verify these properties through numerical studies on particular cyclic cubic and quintic fields and a non-Abelian field
(group $D_6$), and we publish the corresponding ``PARI'' programs.
\end{altabstract}

\date{3 Avril  2014}

\thanks{Je remercie G\'erald Tenenbaum pour un \'echange relatif \`a la th\'eorie probabiliste des nombres et pour d'utiles renseignements.  Je remercie Henri Lombardi pour le signalement d'ambigu\"it\'es et pour ses commentaires. }



\maketitle

\vspace{-1.0cm}
\section{Introduction}
Soit $K/\Q$ une extension Galoisienne de degr\'e fini $n \geq 1$, de groupe de Galois $G$.

\smallskip
On consid\`ere $\eta  \in K^\times$ et on d\'esigne par $F$ le $\Z[G]$-module engendr\'e par $\eta$.  On utilise la notation exponentielle pour la conjuguaison de $\eta$ par $\sigma \in G$ 
(loi de module \`a gauche) qui se traduit par l'\'ecriture $(\eta^\sigma)^\tau =: \eta^{\tau \sigma}$ pour tout 
$\sigma, \tau \in G$.

\smallskip
Nous supposerons toujours que le nombre premier $p$ consid\'er\'e est assez grand, ce qui fait que l'on peut supposer $p$ impair, 
non diviseur de $ n $,
non ramifi\'e dans $K/\Q$, et \'etranger \`a $\eta$. Par exemple, dans le cas o\`u $\eta$ est une unit\'e dont les conjugu\'es engendrent un sous-groupe d'indice fini dans le groupe des unit\'es de $K$ (unit\'e de Minkowski), on peut supposer que $p$ ne divise pas cet indice.

\smallskip
D\'esignons par $Z_K$ (resp.  $Z_{K,(p)}$) l'anneau des entiers (resp. des $p$-entiers) de~$K$~; pour $K=\Q$, on obtient 
$\Z$ (resp. $\Z_{(p)}$).
Pour toute place $v \div p$, on d\'esigne par ${\mathfrak p}_v$ l'id\'eal premier associ\'e \`a $v$. 
Pr\'ecisons une fois pour toutes que dans $Z_{K,(p)}$, toute congruence de la forme
$\beta \equiv 0 \pmod p$ est globale et doit se lire $\beta \equiv 0 \pmod{ \prod_{v\div p}{\mathfrak p}_v }$
($p$ non ramifi\'e)~; elle implique que si $\beta = \sum_{i=1}^n A_i \,e_i$ est \'ecrit sur une $\Z_{(p)}$-base $(e_i)_{i=1,\ldots,n}$
de $Z_{K,(p)}$,  on a $A_i \equiv 0 \pmod p$ pour $i= 1, \ldots, n$.

\smallskip
Si $n_p$ est le degr\'e r\'esiduel commun des places $v \div p$ dans $K/\Q$, les groupes multiplicatifs des corps r\'esiduels sont d'ordres $p^{n_p}-1$ et on a la congruence~:
$$\hbox{$\eta^{p^{n_p}-1} \equiv 1 \pmod{ {\mathfrak p}_v}$ pour tout $v \div p$ ; }$$

par hypoth\`ese  $\prod_{v\div p}{\mathfrak p}_v = (p)$, d'o\`u $\eta^{p^{n_p}-1} \equiv 1 \pmod p$ et~:
$$\eta^{p^{n_p}-1} = 1 + p\,\alpha_p(\eta), \ \ \alpha_p(\eta)\in Z_{K,(p)},  $$
ce qui conduit par Galois \`a la relation~:
$$\alpha_p(\eta^\sigma) \equiv  \alpha_p(\eta)^\sigma  \!\!  \pmod p,  \ \ \forall \sigma \in G, $$
et aux propri\'et\'es ``logarithmiques'' suivantes ($\eta, \eta' \in K^\times$, $\lambda \in \Z$)~:
$$\alpha_p(\eta\,\eta') \equiv \alpha_p(\eta) + \alpha_p(\eta') \!\! \pmod p \ \ \&\ \ 
\alpha_p(\eta^\lambda) \equiv \lambda\, \alpha_p(\eta)  \!\!  \pmod p. $$

\section{R\'egulateurs ${\rm Reg}_p^G (\eta), {\rm Reg}_p^\chi (\eta)$ --  R\'egulateurs locaux $\Delta_p^\theta(\eta)$}

\subsection{Logarithme $p$-adique -- R\'egulateurs $p$-adiques} \label{sub1}
Soit $p$ un nombre premier fix\'e v\'erifiant les hypoth\`eses pos\'ees dans l'Introduction. On suppose que $K$ est consid\'er\'e comme un sous-corps de $\C_p$. Ainsi tout ``plongement'' de $K$ dans $\C_p$ n'est autre qu'un $\Q$-automorphisme 
$\sigma \in G$. 

\smallskip
Soit  ${\mathfrak p}={\mathfrak p}_{v_0}$ un id\'eal premier de $K$ au-dessus de $p$
et soit $D_{\mathfrak p}$ son groupe de d\'ecom\-po\-sition.
Alors les places $v\div p$ conjugu\'ees de $v_0$ correspondent aux $(G : D_{\mathfrak p})$ id\'eaux premiers distincts 
${\mathfrak p}_v := {\mathfrak p}^{\sigma_v}$, o\`u $(\sigma_v)_{v\div p}$ est  un syst\`eme exact 
de plongements repr\'esentant $G/D_{\mathfrak p}$.

\smallskip
On consid\`ere le $\Q_p$-espace vectoriel $\prd_{v \div p} K_v$, de dimension $n$, o\`u $K_v = \sigma_v (K)\,\Q_p$ 
est le compl\'et\'e en $v$ de~$K$~; comme $K/\Q$ est Galoisienne, $K_v/\Q_p$ est ind\'ependante de $v\div p$ mais la notation
$K_v$ rappelle que cette extension locale est {\it munie} du plongement $\sigma_v\, : \,K \too \C_p$, ce qui permet le plongement diagonal d'image dense $i_p : K \too \prd_{v \div p} K_v$ qui \`a $x \in K$ associe $(\sigma_v(x))_{v\div p}$.

\subsubsection{Logarithmes $p$-adiques}\label{log}
Le logarithme $p$-adique ${\rm log}_p ~:  K^\times \too K\Q_p \subset \C_p$ est d\'efini sur l'ensemble des
$1+p\,x$, $x\in Z_{K,(p)}$ par la s\'erie usuelle~:
$${\rm log}_p (1+p\, x)~= \sm_{i \ge 1}\, (-1)^{i+1} \, \Frac{(px)^i}{i} \equiv p\,x \pmod{p^2}. $$

Dans le cas du nombre $\eta \in K^\times$, on utilise la relation fonctionnelle~:
$${\rm log}_p(\eta)~= \Frac{1}{p^{n_p}-1} {\rm log}_p(\eta^{p^{n_p}-1}) = 
\Frac{1}{p^{n_p}-1}{\rm log}_p(1 + p\,\alpha_p(\eta) ) \equiv  - p\,\alpha_p(\eta)  \pmod {p^2} . $$

Cette fonction ${\rm log}_p$, vue modulo $p^N$, $N\geq 2$, est repr\'esent\'ee par des \'el\'ements
de $Z_{K,(p)}$ et est un homomorphisme de $G$-modules pour la loi d\'efinie par~:
$$\sigma \big ({\rm log}_p(\eta)\!\!\!\! \pmod {p^N} \big) := {\rm log}_p(\eta^{\sigma})\!\!\! \pmod {p^N}, $$
 en consid\'erant la congruence~:
\begin{eqnarray*}
\sigma\big ({\rm log}_p(\eta)\!\!\!\! \pmod {p^N} \big)  &\equiv&  \Frac{1}{p^{n_p}-1}  \ 
\sigma \,\Big(\sm_{1 \leq i \leq N}\, (-1)^{i+1} \, \Frac{(p\,\alpha_p(\eta))^i}{i}\Big) \\
&\equiv&  \Frac{1}{p^{n_p}-1}  
\sm_{1 \leq i \leq N} \, (-1)^{i+1} \, \Frac{(p\,\alpha_p(\eta)^\sigma)^i}{i} \!  \pmod {p^N} ,  
\end{eqnarray*}
d\'efinissant un \'el\'ement de $Z_{K,(p)}$ qui approche $\sigma\big ({\rm log}_p(\eta)\big)$ modulo $p^N$.

\smallskip
Soit ${\rm log}_{(p)} ={\rm log}_p \rond \, i_p$ l'application d\'efinie sur le groupe $K_{(p)}^\times$ 
des $\eta$ \'etrangers \`a $p$~par~:
$$\begin{array}{ccc}  K_{(p)}^\times & \tooo  &\ds  \hspace{-0.3cm} \prd_{v \div p} K_v \ .\\
 \eta  & \longmapsto &\ \ \big ( {\rm log}_p (\eta^{\sigma_v}) \big)_{v \div p}^{} 
\end{array}$$
On appelle rang $p$-adique du $\Z[G]$-module $F$ engendr\'e par $\eta$, l'entier~:
$${\rm rg}_p(F)~:= {\rm dim}^{}_{\, \Q_p}(\Q_p {\rm log}_{(p)}(F)) \ \ \hbox{(cf. \cite{Gr1}, III.3.1.2).} $$

\begin{lemm}\label {lemm01}  Soit $p$ premier impair non ramifi\'e dans $K/\Q$ et soit $\lambda \in Z_{K,(p)}$. 
Si $\lambda \notin p\, Z_{K,(p)}$, il existe $u \in K^\times$, \'etranger \`a $p$, tel que 
${\rm Tr}_{K/\Q}(\lambda u) \not\equiv 0 \!\!\pmod p$.
\end{lemm}

\begin{proof} Pour tout $u \in K^\times$, \'etranger \`a $p$, consid\'erons le plongement diagonal de $\lambda u$ dans $\prod_{v \div p} K_v$, et soient ${\rm Tr}_v$ les traces locales  ${\rm Tr}_{K_v/\Q_p}$. On a ${\rm Tr}_{K/\Q}(\lambda u) = \sum_{v \div p} {\rm Tr}_v(\sigma_v(\lambda u))$. 
Par hypoth\`ese, il existe un ensemble non vide $\Sigma$ de places $v \div p$
telles que $\sigma_v(\lambda)$ (donc $\sigma_v(\lambda u)$ pour tout $u$ \'etranger \`a $p$) est une unit\'e de $K_v$. 
Pour $v_0 \in\,\Sigma$, on \'ecrit~:
$${\rm Tr}_{K/\Q}(\lambda u) = \sm_{v \div p,\, v\ne v_0} {\rm Tr}_v(\sigma_v(\lambda u)) +
 {\rm Tr}_{v_0}(\sigma_{v_0}(\lambda u)) =: a +  {\rm Tr}_{v_0}(\sigma_{v_0}(\lambda u)). $$

\vspace{-0.2cm}
Comme $p$ est non ramifi\'e dans $K/\Q$, les traces r\'esiduelles en $p$ sont surjectives et puisque $\sigma_{v_0}(\lambda u)$ est une unit\'e, il suffit de prendre $u \equiv 1 \pmod {\prod_{v ,\,v \ne v_0} {\mathfrak p}_v}$ (auquel cas $a$ ne d\'epend plus de $u$) et 
$u \equiv u_0 \pmod{ {\mathfrak p}_{v_0}}$ convenable
tel que par exemple $ {\rm Tr}_{v_0}(\sigma_{v_0}(\lambda u)) \equiv 1 - a \!\pmod p$ si $a\not\equiv 1\!\! \pmod p$
 (resp. $1\!\! \pmod p$ si $a \equiv 1\pmod p$). 
D'o\`u ${\rm Tr}_{K/\Q}(\lambda u) \equiv 1 \hbox{\ (resp. 2)} \pmod p$.\,\footnote{Pour $p=2$, $K = \Q(\sqrt{17})$,
$\lambda = 1+2 \sqrt {17}$, il n'y a pas de solution $u$ \'etranger \`a 2.}
\end{proof}

Le lemme suivant, valable pour tout $p>2$ \'etranger \`a $\eta$, nous sera particuli\`erement utile (d'apr\`es \cite{Wa},\,\S\,5.5, and  proof of Theorem\,5.31)~: 

\begin{lemm} \label{lemm1} Soit $\eta \in K^\times$ et soient $\lambda(\sigma)$, $\sigma \in G$, des coefficients $p$-entiers
de $K\Q_p$, non tous divisibles par $p$. On suppose que l'on a une relation de d\'ependance modulo $p^N$, $N \geq 2$, 
des vecteurs $\ell_\sigma  = (\ldots,{\rm log}_p (\eta^{\tau\sigma^{\!-1} }), \ldots)_\tau$~:

\smallskip
\centerline{$\sm_{\sigma \in G} \lambda(\sigma){\rm log}_p (\eta^{\tau\sigma^{\!-1}}) \equiv 0\!\pmod {p^N}\ $ pour tout $\tau \in G$. }

 Alors il existe des coefficients $\lambda'(\sigma) \in \Z_{(p)}$, non tous divisibles par $p$, tels que l'on ait la relation 
$\sm_{\sigma \in G} \lambda'(\sigma){\rm log}_p (\eta^{\tau\sigma^{\!-1}}) \equiv 0 \pmod {p^N}$ pour tout $\tau \in G$.

Il en r\'esulte aussi la relation $\sm_{\sigma \in G} \lambda'(\sigma)\alpha_p(\eta)^{\tau\sigma^{\!-1}} \equiv 0 \pmod {p}$ 
pour tout $\tau \in G$.
\end{lemm}

\begin{proof}  On peut toujours, modulo $p^N$, supposer que $\lambda(\sigma) \in Z_{K,(p)}$ pour tout $\sigma \in G$.
On se ram\`ene (par exemple) \`a ${\rm Tr}_{K/\Q}(\lambda(1)) \equiv 1\pmod p$ en multipliant la congruence par $u$
\'etranger \`a $p$ convenable (cf. Lemme \ref{lemm01}). 
Par conjugaison par $\nu \in G$ on obtient
$\sm_{\sigma \in G} \lambda(\sigma)^\nu {\rm log}_p (\eta^{\nu\tau\sigma^{\!-1}}) \equiv 0 \pmod {p^N}$  pour tout $\tau \in G$,
\'equivalent \`a $\sm_{\sigma \in G} \lambda(\sigma)^{\nu} {\rm log}_p (\eta^{s \,\sigma^{\!-1}}) \equiv 0 \pmod {p^N}$ pour tout $s\in G$. 

En prenant la trace dans $K/\Q$ des coefficients (sommation sur $\nu$), on obtient des $\lambda'(\sigma)$ rationnels $p$-entiers
pour tout $\sigma\in G$, avec $\lambda'(1) \equiv 1 \pmod p$.
\end{proof}

On peut donc supposer que de telles relations de  $Z_{K,(p)}$-d\'ependance lin\'eaire modu\-lo $p^N$ 
sont \`a coefficients dans $\Z_{(p)}$, car les deux notions de rang co\"incident. 

\smallskip
En prenant la limite sur $N$, on passe de l'anneau complet $Z_{K}\Z_p$ \`a $\Z_p$.

\subsubsection{R\'egulateurs}\label{reg}
En utilisant l'isomorphisme $\prod_{v \div p} K_v \simeq K\otimes \Q_p \simeq \Q_p[G]$, on en d\'eduit que
le rang $p$-adique de $F$ est \'egal au rang sur $\Z_p$ (au sens pr\'ec\'edent) du d\'eterminant de Frobenius
(ou r\'egulateur $p$-adique de $\eta$)~:
$${\rm Det}^G ( {\rm log}_p(\eta)) := {\rm det} \big( {\rm log}_p(\eta^{\tau\sigma^{\!-1} }) \,\big)_{\sigma, \tau \in G}  $$
(voir \cite{Gr1}, III.3.1.8, o\`u il est aussi fait allusion \`a  la conjecture de Schanuel).

\smallskip
Si ce rang est donn\'e par un plus grand mineur $M$ non nul, $M$ est non nul modulo toute puissance $p^N$ assez grande. 

\smallskip
Le $\Z[G]$-module $F$ est  monog\`ene au sens rappel\'e
dans \cite{J},\,\S\,1, ou \cite{Gr1}, III.3.1.2 (ii), auquel cas la conjecture de Jaulent \cite{J},\,\S\,2, affirme que le rang $p$-adique du groupe $F$ est \'egal \`a son $\Z$-rang ${\rm rg}(F) := {\rm dim}^{}_{\, \Q} (F \otimes \Q)$. 

\smallskip
Afin de proc\'eder \`a des calculs concrets, il convient d'am\'enager les aspects $p$-adiques ci-dessus en se ramenant \`a des congruences modulo $p^N$ dans $Z_{K,(p)}$.
On notera d'abord que tout mineur d'ordre $r$ est de fa\c con canonique divisible par $p^r$, pour tout $p>2$ non ramifi\'e, 
puisque l'on a, dans $Z_{K,(p)}$, la congruence~:
$${\rm log}_p(\eta) \equiv \Frac{1}{p^{n_p}-1}  p\,\alpha_p(\eta) \equiv -p\,\alpha_p(\eta)  \pmod {p^2}. $$
 D'o\`u la d\'efinition suivante~:

\begin{defi} \label{defimodif} (i)
On consid\`ere (pour $p$ assez grand) le d\'eterminant  \`a coef\-ficients entiers~:
$${\rm Det}^G \Big(\Frac{-1}{p}\, {\rm log}_p(\eta)\Big) = \Big(\Frac{-1}{p}\Big)^n\, {\rm Det}^G ( {\rm log}_p(\eta)) = {\rm det}
\Big(\Frac{-1}{p}\, {\rm log}_p(\eta^{\tau\sigma^{\!-1} }) \Big)_{\sigma, \tau \in G}. $$

Ce d\'eterminant de Frobenius, d\'esign\'e dans l'article par ${\rm Reg}_p^G (\eta)$, est appel\'e le r\'egulateur $p$-adique normalis\'e~de~$\eta$.

\smallskip
On a ${\rm Reg}_p^G (\eta) \equiv \Delta_p^G (\eta) \pmod p$, o\`u~:
$$\Delta_p^G (\eta) := {\rm Det}^G(\alpha_p(\eta) ) ={\rm det}\big( \alpha_p(\eta)^{\tau\sigma^{\!-1}}\big)_{\sigma, \tau \in G}$$ 
est  appel\'e le r\'egulateur (normalis\'e) local de $\eta$
 car uniquement d\'efini modulo~$p$.

\smallskip
(ii) Pour $K$ Galoisien r\'eel, le r\'egulateur $p$-adique usuel ${\rm Reg}_p(K)$ est donn\'e par un mineur d'ordre $n-1$ de 
 ${\rm Det}^G ( {\rm log}_p(\varepsilon)) =
 {\rm det}\big( {\rm log}_p(\varepsilon^{\tau\sigma^{\!-1} })\big)_{\sigma, \tau \in G}$, o\`u $\varepsilon$ est une unit\'e de Minkowski, 
et l'entier $p$-adique $p^{1-n} {\rm Reg}_p(K) = {\rm det}\Big(\Frac{-1}{p}\, {\rm log}_p(\varepsilon^{\tau\sigma^{\!-1} })\Big)_{\sigma \ne 1, \tau \ne 1}$ est appel\'e le r\'egulateur $p$-adique normalis\'e de~$K$ ($p$ assez grand).
\end{defi}

Le plongement diagonal de $K$ dans $ \prod_{v \div p} K_v$, au moyen de $i_p = (\sigma_v)_{v\div p}$, \'etant dense, on peut consid\'erer
que, modulo $p^N$, la fonction $\frac{-1}{p}{\rm log}_p$ est \`a valeurs dans~$K$.
D'apr\`es le Lemme \ref{lemm1},\,\S\ref{log}, on se ram\`ene \`a des raisonnement d'alg\`ebre lin\'eaire sur 
$\Z/p^N\Z$~; en particulier, le rang $p$-adique de $F$ est le $\Z/p^N\Z$-rang de la matrice
$\big ( \frac{-1}{p}{\rm log}_p (\eta^{\tau\sigma^{\!-1}}) \!\!\pmod {p^N}\big )_{\sigma, \tau \in G}$, pour tout $N$ assez grand.

Si un mineur $M$ d'ordre ${\rm rg}(F)$ de cette matrice est non nul,
 alors il donne le rang $p$-adique de $F$ et c'est le point de vue pratique choisi que nous limiterons \`a $N=1$,
donc aux $\alpha_p(\eta)$ modulo $p$~; dans ce cas, le rang $p$-adique de $F$ est a priori sup\'erieur ou \'egal
au $\Z/p\Z$-rang de $\big(\alpha_p(\eta)^{\tau\sigma^{\!-1}}\!\!\pmod p\big)_{\sigma, \tau \in G}$.

\smallskip
Dans \cite{Hat}, Hatada a proc\'ed\'e, pour certains corps de nombres, \`a une \'etude statistique de la nullit\'e 
modulo $p$ du r\'egulateur normalis\'e en utilisant le fait que sa $p$-valuation est celle de $p\,\zeta_K(2-p)$,
o\`u $\zeta_K$ est la fonction z\^eta de~$K$.

\begin{rema}\label{global} (a) {\it Analyse locale.}  On ne fait aucune hypoth\`ese sur le $\Z$-rang de $F$.
Si l'on a dans $F$ la relation $\prod_{\sigma \in G} (\eta^{\sigma^{\!-1}})^{\lambda(\sigma)} =1$,
$\lambda(\sigma) \in \Z$, alors, pour tout $p$ \'etranger \`a~$\eta$  on a
$\sum_{\sigma \in G}\lambda(\sigma)\, {\rm log}_p (\eta^{\sigma^{\!-1}}) =0$, ce qui conduira
pour certains caract\`eres rationnels $\chi$ \`a des $\chi$-r\'egulateurs  
(not\'es ${\rm Reg}_p^\chi(\eta)$) identiquements nuls pour tout~$p$.

\smallskip
Ces relations se transmettent pour tout $p$ en les relations plus faibles~:
$$\sm_{\sigma \in G}\lambda(\sigma)\, \alpha_p(\eta)^{\sigma^{\!-1}} \equiv 0  \pmod p$$ 

entre les conjugu\'es de $\alpha_p(\eta)$~; elles sont dites triviales 
(elles ne sont pas d\^ues \`a une circonstance num\'erique pour le $p$ consid\'er\'e, mais \`a l'existence
d'une relation dans $F$ donn\'ee par des $\lambda(\sigma)$ ind\'ependants de $p$, ce qui sera caract\'eris\'e au
Lemme \ref{nul},\,\S\ref{sub5}). Ceci n'a pas lieu si $F \otimes \Q \simeq \Q[G]$.

\smallskip
Inversement, si l'on se donne \`a partir de $\eta$ quelconque la famille de congruences~:

\medskip
\centerline{\hspace{3.5cm} $\big(\sm_{\sigma \in G} \lambda(\sigma)\, \alpha_p(\eta)^{\sigma^{\!-1}} \equiv 0  \pmod p\big)_p$\hspace{3.5cm} (*)}

(ensemble de conditions locales car les $\alpha_p(\eta)$ ne sont connus que modulo $p$), 
la question est de savoir si 
elles sont globalisables sous la forme $\prod_{\sigma \in G} (\eta^{\sigma^{\!-1}})^{\lambda(\sigma)} =1$.
Supposons ne disposer que des congruences
$\sm_{\sigma \in G} \lambda(\sigma)\, \alpha_p(\eta)^{\sigma^{\!-1}} \equiv 0  \pmod p$
 pour presque tout $p$ (\'even\-tuel\-lement au sens du\,\S\ref{prem}), avec des $\lambda(\sigma) \in \Z\,$ ind\'ependants de $p$. Soit
$\Lambda := \prod_{\sigma \in G} (\eta^{\sigma^{\!-1}})^{\lambda(\sigma)} \in F$~; alors ${\rm log}_p(\Lambda)  \equiv 0  \pmod {p^2}$  et $\Lambda$ est, dans $\prod_{v \div p} K_v$, de la forme $\xi\,(1+ \beta\,p)^p$,
o\`u $\xi$ est de torsion d'ordre \'etranger \`a $p$ (pour tout $p$ assez grand)~; 
donc $\Lambda \in \prod_{v \div p} K_v^{\times p}$ pour presque tout $p$.
Conjecturalement $\Lambda$ est une racine de l'unit\'e de $K$ (cf. Conjecture \ref{conj31},\,\S\ref{sub38}). 

\medskip
(b) {\it Analyse globale.} \label{clj} Par comparaison, supposons que l'on ait, dans un cadre limite projective, 
des coefficients  $\wh\lambda(\sigma) \in \wh\Z$ (pour tout $\sigma \in G$), tels que~:
$$\sm_{\sigma \in G} \wh\lambda_p(\sigma) {\rm log}_{(p)} (\eta^{\sigma^{\!-1}}) = 0 , 
\ \ \hbox{o\`u ${\rm log}_{(p)} = {\rm log}_p \rond\, i_p$, }$$ 

\vspace{-0.3cm}
pour tout $p$ \'etranger \`a $\eta$, o\`u  pour tout $\sigma \in G$, $\wh\lambda_p(\sigma)$ est la $p$-composante de $\wh\lambda(\sigma)$.

\smallskip
Soit $i := (i_v)_{v,\,v(\eta)=0}$ le plongement diagonal $F \otimes \wh\Z \too \wh U$, o\`u l'on a pos\'e~:

\centerline{$\wh U =\!\! \prd_{p,\, (p,\eta)=1} \!\Big(\prd_{  v\vert p} U_v^1 \times \prd_{ v\notdiv p,\, v(\eta)=0} \mu_p(K_v) \Big)$, }

 avec des notations \'evidentes ($U_v^1 = \mu_p(K_v) \times U'$, $U'$ $\Z_p$-libre).

\smallskip
 On pose $\wh\Lambda := \prod_{\sigma \in G} (\eta^{\sigma^{\!-1}})^{\wh\lambda(\sigma)} \in F \otimes \wh\Z$ et on d\'esigne par $\wh\Lambda_p = \prod_{\sigma \in G} (\eta^{\sigma^{\!-1}})^{\wh\lambda_p(\sigma)} $ la $p$-composante de $\wh\Lambda$ 
($p$ \'etranger \`a $\eta$).

\smallskip
Comme ${\rm log}_{(p)} (\wh\Lambda_p)=0$ pour tout $p$ \'etranger \`a $\eta$, on a pour toute place $v$ de $K$ \'etrang\`ere \`a $\eta$, $i_v(\wh\Lambda) = \xi_v$, o\`u en g\'en\'eral $\xi_v$ est une racine de l'unit\'e d'ordre diviseur de $\ell^{n_\ell}-1$, o\`u $\ell$ est la caract\'eristique r\'esiduelle de $v$ (les places $v \div p$ telles que $\xi_v$ est d'ordre divisible par $p$ sont en nombre fini).
On peut donc \'ecrire~:
$$i(\wh\Lambda) \in i (F \otimes \wh\Z) \ \Cap \prd_{p,\, (p,\eta)=1} \!\Big( \prd_{ v ,\, v(\eta)=0} \mu_p(K_v) \Big). $$

\vspace{-0.2cm}
En supposant l'analogue pour $F$ de la caract\'erisation locale--globale de la conjecture de 
Leopoldt--Jaulent (cf. \cite{J},\,\S\,2) pour les nombres premiers~$p$ (voir aussi \cite{Gr1}, III.3.6.6 pour le cas du groupe des unit\'es), on peut affirmer (sous cette conjecture pour tout $p$) que
$i (F \otimes \wh\Z)\  \Cap  \prd_{p,\, (p,\eta)=1} \!\Big( \prd_{ v ,\, v(\eta)=0} \mu_p(K_v) \Big) = i (\mu (K))$.

\vspace{-0.1cm}
On en d\'eduit que $\wh\Lambda_p$ est une racine de l'unit\'e $\zeta_p$ de $K$ pour tout $p$ \'etranger \`a $\eta$. 

\smallskip
Si l'on suppose $\wh\lambda_p(\sigma) \equiv  \lambda(\sigma) \pmod p$ pour tout $\sigma \in G$ et tout $p$ \'etranger \`a $\eta$, 
o\`u les $\lambda(\sigma)$ sont des entiers rationnels donn\'es,
alors $\Lambda := \prod_{\sigma \in G} (\eta^{\sigma^{\!-1}})^{\lambda(\sigma)} \in F$ est \'egal \`a $\wh\Lambda_p$ \`a une puissance $p$-i\`eme locale pr\`es, donc $\Lambda = \zeta_p \,\gamma_p^p$, $\gamma_p\in \prod_{v \div p} K_v^{\times p}$ pour tout $p$ \'etranger \`a $\eta$~; on obtient la situation de la Remarque \ref{global} puisque $\zeta_p\ne1$ pour un nombre fini de $p$, et pour $p$ assez grand il y a bien co\"incidence.

\smallskip
On peut donc voir notre d\'emarche comme un affaiblissement de ce contexte $p$-adique classique concernant la conjecture de 
Leopoldt--Jaulent~; mais en contre\-partie, on a d\^u supposer l'existence de la famille d'entiers rationnels 
$(\lambda(\sigma))_{\sigma \in G}$ ind\'ependante de $p$ v\'erifiant la relation (*).
\end{rema}

Notre principal objectif est de voir sous quelles conditions et avec quelle ``probabi\-lit\'e'' le r\'egulateur normalis\'e 
 ${\rm Reg}_p^G (\eta)$ est divisible par $p$ assez grand.
D'apr\`es ce qui pr\'ec\`ede et le Lemme \ref{lemm1},\,\S\ref{log}, on a  ${\rm Reg}_p^G (\eta)   \equiv 0  \pmod p$ 
si et seulement si il existe une relation de $\Z/p\Z$-d\'ependance des vecteurs-lignes 
$\ell_\sigma = (\ldots, \frac{-1}{p} {\rm log}_p (\eta^{\tau{\sigma^{\!-1}}}), \ldots)_\tau \equiv
 (\ldots, \alpha_p(\eta)^{\tau{\sigma^{\!-1}}}, \ldots)_\tau \pmod p$, pour $\sigma \in G$ 
(ceci sera repris au \S\ref{sub8}). 

\smallskip
Comme ${\rm Reg}_p^G (\eta)$ se factorise en produit de puissances de $\chi$-r\'egu\-lateurs 
${\rm Reg}_p^\chi (\eta)$ (pour les caract\`eres rationnels irr\'eductibles $\chi$ de $G$), on \'etudie la divisibilit\'e par $p$ de ces facteurs.
Cette factorisation ne d\'epend pas de $p$, par contre on peut ensuite factoriser ${\rm Reg}_p^\chi (\eta)$ en produit de $\theta$-composantes  ${\rm Reg}_p^\theta (\eta)$ (pour les caract\`eres $p$-adiques irr\'eductibles $\theta \div \chi$), cette factorisation d\'ependant du degr\'e r\'esiduel de $p$ dans le corps des valeurs des caract\`eres absolument irr\'eductibles $\varphi \div \chi$ de $G$. On aura ensuite ${\rm Reg}_p^\theta (\eta) \equiv \Delta_p^\theta (\eta) \pmod p$ en termes de $\theta$-r\'egulateurs locaux de la forme $\Delta_p^\theta(\eta) := {\rm Det}^\theta (\alpha_p(\eta) )$ (d\'efinis seulement modulo $p$
et \'etudi\'es au \S\ref{sub6}).

\smallskip
Une premi\`ere \'etude montrera que la probabilit\'e d'avoir ${\rm Reg}_p^\theta (\eta)$ et ${\rm Reg}_p^{\theta'} (\eta)$ nuls modulo $p$, pour $\theta \ne \theta'$, est le produit des probabilit\'es (ind\'ependance), donc
au plus en $\Frac{O(1)}{p^2}$, et on pourra ``exclure'' cette \'eventualit\'e pour tout $p$ assez grand et se limiter au cas d'une unique $\theta$-composante nulle modulo $p$ (cf.\,\S\S\ref{sub9}, \ref{sub12}).

\subsection{Les d\'eterminants de groupes (ou de Frobenius)}\label{sub2}
 Pour nos objectifs heuristiques 
nous utilisons le cas Ab\'elien tr\`es classique (cf. \cite{Wa}, Lemma\,5.26 ou \cite{C},\,\S\,2)
ainsi que le cas des groupes non commutatifs pour lesquels les aspects probabilistes sont plus d\'elicats au niveau de leurs caract\`eres irr\'eductibles de degr\'es $\geq 2$~; \`a cette fin nous 
commen\c cons par un rappel g\'en\'eral en termes de repr\'esentations  (pour un expos\'e sur les repr\'esentations 
et les caract\`eres, voir~\cite{Se1}).

\subsubsection{Notations g\'en\'erales}\label{subnot}
En tant que repr\'esentation r\'eguli\`ere, on a l'isomor\-phisme $\C[G] \simeq \bigoplus_\rho {\rm deg}(\rho) \,. \,V_\rho$, 
o\`u $(\rho, V_\rho)$ d\'ecrit l'ensemble des repr\'esentations absolument irr\'eductibles de $G$ et o\`u ${\rm deg} (\rho)$
est le degr\'e ($\C$-dimension de $V_\rho$). 

\smallskip
On d\'esigne par $\varphi$ le caract\`ere de $\rho$~; par cons\'equent, ${\rm deg}(\rho) 
= \varphi(1)$.  Nous convenons d'indexer les objets d\'ependant de $\rho$ par la lettre $\varphi$ (e.g. $V_\varphi$)
et de ne conserver $\rho = \rho_\varphi$ qu'en tant qu'homomorphisme de $G$ dans ${\rm End}(V_\varphi)$.

\smallskip 
Comme alg\`ebre d'endomorphismes $E \in \C[G]$ agissant sur la base des $\nu \in G$ par multiplication~: 
$\nu \mapsto  E \,.\, \nu$, 
on a l'isomorphisme $\C[G] \simeq \bigoplus_\varphi {\rm End}(V_\varphi)$, avec 
${\rm End}(V_\varphi) \simeq e_\varphi \C[G]$, o\`u les $e_\varphi  = \Frac{\varphi (1)}{\vert G \vert} 
\sm_{\nu \in G} \varphi (\nu^{\!-1} )\,\nu$ sont les idempotents centraux orthogonaux de  $\C[G]$.
Pour la d\'ecomposition de $e_\varphi \C[G]$ en somme directe de $\varphi (1)$ repr\'esentations irr\'educ\-tibles isomorphes \`a $V_\varphi$,
on  utilise les projecteurs issus d'une repr\'esentation matricielle  $M(\rho_\varphi(\nu)) = \big( a_{ij}^\varphi(\nu) \big)_{i,j}$
 (cf. \cite{Se1},\,\S\,I.2.7)~:
$$\hbox{$\pi_i^\varphi =  \Frac{\varphi(1)}{ n } \sm_{\nu \in G} a_{ii}^\varphi(\nu^{\!-1}) \,\nu$, $\ \,i= 1, \ldots, \varphi(1)$, }$$
qui forment un syst\`eme d'idempotents orthogonaux au-dessus de $e_\varphi = \sum_i \pi_i^\varphi$.

\subsubsection{Rappels sur les d\'eterminants de groupes {\rm (d'apr\`es \cite{C})}}\label{sub3} 
Soit $G$ un groupe fini et soit ${\rm Det}^G(X) = {\rm det}\big( X_{\tau{\sigma^{\!-1}}} \big)_{\sigma, \tau \in G}$
le d\'eterminant du groupe $G$, ou d\'eterminant de Frobenius, en les ind\'etermin\'ees~$X := (X_\nu)_{\nu \in G}$.
 On a alors la formule~:
$${\rm Det}^G(X) = \prd_\varphi {\rm det} 
\Big( \sm_{\nu \in G} X_\nu \rho_\varphi (\nu^{\!-1}) \Big)^{\varphi(1)}. $$

Il en r\'esulte  l'existence de polyn\^omes homog\`enes $P^\varphi(X)$,
de degr\'es $\varphi(1)$, tels que~:
$${\rm Det}^G(X) = \prd_\varphi P^\varphi (X)^{\varphi(1)}.$$

\vspace{-0.2cm}
On regroupe en produits partiels associ\'es aux caract\`eres $\chi$, irr\'eductibles sur~$\Q$, ou aux 
caract\`eres $p$-adiques $\theta$,  irr\'eductibles sur~$\Q_p$ (cf. D\'efinition \ref{defi01},\,\S\ref{calp})~; ceci d\'efinit 
${\rm Det}^\chi(X)$ et ${\rm Det}^\theta(X)$.
La sp\'ecialisation $X_\nu\! \mapsto \frac{-1}{p} {\rm log}_p(\eta^{\nu })\!$ conduit \`a la factorisation 
(cf. D\'efinition \ref{defimodif},\,\S\ref{sub1})~:
$${\rm Det}^G\big( \hbox{ $\frac{-1}{p} $ }   {\rm log}_p(\eta ) \big) =   {\rm Reg}_p^G(\eta) =  \prd_\varphi {\rm det} \Big(\sm_{\nu \in G}\hbox{ $\frac{-1}{p} $ }  {\rm log}_p(\eta^{\nu })\, \rho_\varphi(\nu^{\!-1}) \Big)^{\varphi(1)},$$
et on pose ${\rm Reg}_p^\varphi(\eta) = 
 {\rm det} \big(\!\sm_{\nu \in G}$$\frac{-1}{p} {\rm log}_p(\eta^{\nu })\, \rho_\varphi(\nu^{\!-1}) \!\big ),\,
{\rm Reg}_p^\chi(\eta) = \prd_{\varphi \div \chi}{\rm Reg}_p^\varphi(\eta)$, et
${\rm Reg}_p^\theta(\eta) \!=\! \prd_{\varphi \div \theta}{\rm Reg}_p^\varphi(\eta)$~;
on aura des r\'esidus modulo $p$ de ces r\'egulateurs via les con\-gruences $\frac{-1}{p} {\rm log}_p(\eta^{\nu })
\equiv \alpha_p(\eta)^{\nu } \!\pmod p$, $\nu \in G$, conduisant aux $\Delta_p^\theta(\eta)$.

De fait l'\'etude modulo $p$ de chaque ${\rm Reg}_p^\theta (\eta)$ donne une information beaucoup 
plus pr\'ecise que l'\'etude modulo $p$ du r\'egulateur normalis\'e ${\rm Reg}_p^G(\eta) = \prd_{\theta} {\rm Reg}_p^\theta (\eta)^{\varphi(1)}$.

On \'etudiera plus g\'en\'eralement  la sp\'ecialisation $X_\nu \mapsto \alpha^\nu$, $\alpha \in Z_K$ donn\'e 
ind\'ependamment de tout $\eta$ et $p$, dont nous noterons  le r\'esultat sous la forme~:
$${\rm Det}^G(\alpha) := {\rm det}(\alpha^{\tau{\sigma^{\!-1}}})_{\sigma, \tau \in G} = \prd_\varphi
 {\rm det} \Big(\!\sm_{\nu \in G} \alpha^\nu \, \rho_\varphi(\nu^{\!-1}) \!\Big )^{\varphi(1)}  , $$
o\`u l'on rappelle que $\alpha^{\tau{\sigma^{\!-1}}} = (\alpha^{{\sigma^{\!-1}}})^\tau$.
Par exemple, dans le contexte logarithmique pr\'ec\'edent relatif \`a un nombre alg\'e\-brique $\eta \in K^\times$ 
pour lequel $\alpha \equiv \alpha_p(\eta) \pmod p$, nous utiliserons les r\'egulateurs locaux (car ici $\alpha$ 
n'est d\'efini que modulo $p$)~:
$$\Delta^G_p(\eta) :=  {\rm det}(\alpha^{\tau{\sigma^{\!-1}}}\!\!\!\!\pmod p)_{\sigma, \tau \in G} , $$
via leurs factorisations en $\theta$-composantes (les $\theta$-r\'egulateurs locaux $\Delta^\theta_p(\eta)$).

\subsubsection{Calcul pratique des $P^\varphi(X)$}\label{calp}
Les polyn\^omes $P^\varphi(X)$ sont obtenus de la fa\c con suivante~:
\`a partir de l'espace vectoriel $V = \C[G]$ (muni de la base des $\nu\in G$), on consid\`ere
l'endomorphisme de $V[X] := V[\ldots, X_\nu, \ldots]$
d\'efini par $L(X) = \sm_{\nu \in G} X_\nu \,\nu^{\!-1}$.
 Il est tel que~:

\smallskip
\centerline{$\big(\sm_{\nu \in G} X_\nu \,\nu^{\!-1}\big)\,.\,\tau = \sm_{\nu \in G} X_\nu \,\nu^{\!-1} \tau = 
\sm_{\sigma \in G} X_{\tau \sigma^{\!-1}} \, \sigma, \ \  \,\forall \tau \in G$. }

\medskip
Donc le d\'eterminant de cet endomorphisme dans la base des $\tau \in G$ est le d\'eterminant de Frobenius 
(d\'efini au signe pr\`es).

\smallskip
Soit $(\rho_\varphi, V_\varphi)$ la famille des repr\'esentations absolument irreductibles non isomorphes.
On prendra pour $ {\rm End}(V_\varphi)$ la composante $e_\varphi\C[G]$ associ\'ee au
caract\`ere~$\varphi$.\footnote{Par exemple, pour $G= D_6$, une base de $e_\varphi \C[G]$, pour le caract\`ere irr\'eductible de degr\'e 2,  est donn\'ee
par $\{ E_1=e_\varphi(1-\sigma),  E_\sigma=e_\varphi \sigma (1-\sigma), E_\tau=e_\varphi \tau(1-\sigma), 
 E_{\tau\sigma}=e_\varphi \tau \sigma (1-\sigma)  \}$. }

\smallskip
On utilise l'isomorphisme d'alg\`ebres $\wt \rho : V \too \prd_\varphi {\rm End}(V_\varphi)$ d\'efini par~:

\centerline{$\sm_{\nu \in G} a(\nu)\, \nu^{\!-1} \longmapsto \Big(\sm_{\nu \in G} a(\nu) \, \rho_\varphi (\nu^{\!-1})\Big)_\varphi$. }

\smallskip
o\`u $\rho_\varphi(\nu^{\!-1}) = e_\varphi \,\nu^{\!-1}$ dans l'identification pr\'ec\'edente.
D'apr\`es le th\'eor\`eme de Maschke on a, au niveau de l'endomorphisme $L(X)$~:

\smallskip
\centerline{${\rm det}_V(L(X)) = \prd_\varphi \big( {\rm det}_{V_\varphi}(L^\varphi(X)) \big)^{\varphi(1)}$, } 

o\`u $L^\varphi(X) = \sm_{\nu \in G} X_\nu\,\rho_\varphi(\nu^{\!-1}) \in {\rm End} (V_\varphi [X])$. 
On pose  $P^\varphi(X) :=  {\rm det}_{V_\varphi}(L^\varphi (X))$.  

\smallskip
Avec une r\'ealisation matricielle $M(\rho_\varphi(\nu) )= \big( a^\varphi_{ij}(\nu) \big)_{i,j}$
des $\rho_\varphi(\nu)$,  la matrice asso\-ci\'ee \`a $L^\varphi (X)$ est
$M^\varphi(X) = \big(\sum_{\nu \in G}a^\varphi_{ij}(\nu^{\!-1})\, X_\nu \big)_{i,j}$, de d\'eterminant $P^\varphi(X)$.
On peut \'egalement revenir au proc\'ed\'e direct de Frobenius \`a partir des caract\`eres.
Soit $g$ le plus petit commun multiple des ordres des \'el\'ements de $G$~: on sait que les repr\'esentations sont 
r\'ealisables sur le corps $C = \Q(\mu_g)$ des racines $g$-i\`emes de l'unit\'e (cf.  \cite{Se1},\,\S\,12.3).
On pourra donc toujours supposer que les $a^\varphi_{ij}(\nu)$ sont des nombres alg\'ebriques $p$-entiers 
pour tout $p$ assez grand.

\smallskip
Soit  $\Gamma := {\rm Gal}(C/\Q)$.
Etant donn\'e une repr\'esentation  absolument irr\'eductible
$\rho_\varphi : G \mapsto {\rm End_{C}}(V_\varphi)$, on d\'efinit ses conjugu\'ees de fa\c con Galoisienne, de sorte que 
 pour tout $s \in \Gamma$, $\rho^s_\varphi$ est la repr\'esentation $G \mapsto {\rm End_{C}}(V_{\varphi^s})
\simeq e_{\varphi^s} C[G]$ de caract\`ere $\varphi^s$ d\'efini par $\varphi^s(\nu) = (\varphi(\nu) )^s$, pour tout $s \in \Gamma$.
On a, pour tout $s \in \Gamma$, $\varphi^s(\nu) = 
\varphi(\nu^{\omega(s)} )$, o\`u $\omega$ est le caract\`ere $\Gamma \to(\Z/g\Z)^\times$ de l'action de $\Gamma$ sur~$\mu_g$.
On pose aussi  $\varphi^t(\nu) = \varphi(\nu^t)$ pour tout entier $t$ \'etranger \`a $g$ ($\Gamma$-conjugaison).

\begin{defi}\label{defi01} (i) {\it Caract\`eres rationnels}.
Comme $\Gamma$ op\`ere sur les les caract\`eres irr\'eductibles $\varphi$, on pose~:

\medskip
\centerline{$\chi = \sm_{s\in {\rm Gal}(C_\chi/\Q)}  \varphi^s \ \  {\rm et} \ \  P^\chi (X) := \prd_{s\in {\rm Gal}(C_\chi/\Q)} P^{\varphi^s}(X), $}

\smallskip
 o\`u $C_\chi \subseteq C$ est le corps des valeurs de n'importe quel $\Q$-conjugu\'e de $\varphi$.

\smallskip
(ii) {\it Caract\`eres $p$-adiques}. Si $p\notdiv g$ est un nombre premier on d\'esigne, pour $\chi$ fix\'e, par $L$
et $D$ le corps et le groupe de d\'ecomposition de $p$ dans $C_\chi/\Q$. On d\'esigne par $f  = \vert D \vert$ le degr\'e
r\'esiduel de $p$ dans $C_\chi/\Q$ et par $h = [L : \Q]$ le nombre d'id\'eaux premiers au-dessus de $p$~; ainsi 
$[C_\chi : \Q] = h\,f$.

\smallskip
Soit $\varphi$ un caract\`ere absolument irr\'eductible dont la somme des $\Q$-conjugu\'es est~$\chi$. 

\smallskip
On consid\`ere, pour tout $\nu \in G$, $\Theta(\nu) := \sm_{s\in D} \varphi^s(\nu) \in L$ puis l'image diagonale
$\big( t_{\mathfrak p}(\Theta(\nu))  \big)_{{\mathfrak p} \div p} \in \prd_{{\mathfrak p} \div p} L_{\mathfrak p}= 
\prd_{{\mathfrak p} \div p} \Q_p$ o\`u les $t_{\mathfrak p}$ repr\'esentent ${\rm Gal}(C_\chi/\Q)/D$.
On d\'esigne par
$\theta(\nu)$ la ${\mathfrak p}$-composante $t_{\mathfrak p}(\Theta(\nu))$ qui est dans $\Z_p$~; par ce proc\'ed\'e on d\'efinit les caract\`eres $\theta$ (que l'on devrait indexer par les ${{\mathfrak p} \div p}$), via des congruences modulo ${\mathfrak p}$ de la forme $\theta(\nu) \equiv r_{\mathfrak p} (\nu)\pmod {\mathfrak p}$, $r_{\mathfrak p} (\nu) \in \Z$ pour tout $\nu \in G$.
Comme $\theta$ est  \`a valeurs dans $\Z_p$, on \'ecrira par abus $\theta(\nu) \equiv r_{\mathfrak p} (\nu)\pmod p$.

\smallskip
Si l'on fixe un couple $(\theta, {\mathfrak p})$, les autres sont les $h$ conjugu\'es $(\theta^t, {\mathfrak p}^t)$ ainsi d\'efinis~:
si $\theta = \sum_{s \in D}  \varphi^{s}$, les $h$ conjugu\'es de $\theta$ sont les $\theta^{t} = \sum_{s \in D} 
 (\varphi^{t})^{s}$, $t\in {\rm Gal}(C_\chi/\Q)/ D$, et on aura $\theta^t(\nu) \equiv r_{{\mathfrak p}^t} (\nu)\pmod p$, o\`u les
rationnels $r_{{\mathfrak p}^t} (\nu)$ d\'ependent num\'eriquement des images r\'esiduelles en les ${\mathfrak p} \div p$ des traces relatives des $\varphi(\nu)$, $\varphi \div \chi$.

\smallskip
Pour $p$ fix\'e, l'entier $f$ ne d\'epend que de $\chi$
et est appel\'e par abus le degr\'e r\'esiduel des caract\`eres $\varphi, \theta$ et~$\chi$. 
On a par $\Gamma$-conjugaison, $\varphi^{p^i}(\nu) = \varphi(\nu^{p^i}) = \varphi(\nu)^{s_p^i}$, o\`u $s_p$
est l'automorphisme de Frobenius (d'ordre $f$) dans $C_\chi/\Q$.
On pose alors~:

\medskip
\centerline{ $P^\theta (X) := \prd_{s\in D}  P^{\varphi^s}(X)$. }

\smallskip
(iii) {\it Idempotents.} Ecrivons $\varphi \div \theta$, $\theta \div \chi$ pour indiquer que $\varphi$ (resp. $\theta$) est un terme de $\theta$ (resp.~$\chi$). 
On pose $e_\chi = \sum_{\varphi \div \chi}e_\varphi$ et $e_\theta = \sum_{\varphi \div \theta}e_\varphi$~; d'o\`u
$e_\chi = \sum_{\theta \div \chi} e_\theta$. Les $e_\theta$ (resp. les $e_\chi$) forment un syst\`eme fondamental d'idempotents orthogonaux de $\Q_p[G]$ (resp. $\Q[G]$), \`a savoir  $\Q_p[G] = \bigoplus_\theta e_\theta \Q_p[G]$, et $\ \Q[G] = \bigoplus_\chi e_\chi \Q[G]$. On peut remplacer $\Q_p$ (resp. $\Q$) par $\Z_p$ (resp. $\Z_{(p)}$) car $p \notdiv g$.
\end{defi}

De la formule $P^\varphi(X) =  {\rm det}_{V_\varphi}(L^\varphi(X))$ on d\'eduit que $P^{\varphi^s}(X) =  {\rm det}_{V_{\varphi^s}}(L^{\varphi^s} (X))$ o\`u $L^{\varphi^s}(X) = \sum_{\nu \in G} X_\nu\,\rho_\varphi^s(\nu^{\!-1})$ qui est donn\'e via les $a^\varphi_{ij} (\nu^{\!-1})^s$, ce qui d\'efinit le 
conjugu\'e par $s$ du polyn\^ome $P^\varphi(X)$ (i.e., de ses coefficients).

\begin{lemm}\label{lemm2} Pour tout $p$ assez grand,
les polyn\^omes $P^\chi (X)$ (resp. $P^\theta (X)$) sont \`a coefficients $p$-entiers rationnels 
(resp. entiers $p$-adiques). 

\smallskip
Pour tout caract\`ere irr\'eductible $\varphi$, on a $P^\varphi (\ldots,X_{\pi \nu}, \ldots) = \zeta_\pi\, P^\varphi (\ldots,X_{\nu}, \ldots)$ pour tout $\pi \in G$, o\`u $\zeta_\pi$ est une racine de l'unit\'e d'ordre diviseur de $g$.
\end{lemm}

\begin{proof} Comme $P^\chi (X) = \prod_{s \in {\rm Gal}(C_\chi/\Q)} P^{\varphi^s} (X)$ est invariant par Galois,
 le premier point est clair. De m\^eme pour $P^\theta (X) = \prod_{s \in D} P^{\varphi^s} (X)$.

\smallskip
Pour $\pi \in G$ appelons $[\pi]$ l'op\'erateur de $C[X]$ dans $C[X]$ d\'efini par $[\pi] X_\nu = X_{(\pi\nu)}$
pour tout $\nu \in G$. Alors $[\pi]$ et $\wt \rho : V[X] \too \prod_\varphi {\rm End}(V_\varphi[X])$ commutent~; 
de plus, puisque $\rho_\varphi$ est un homomorphisme, on a la formule suivante~:
$$[\pi]\Big (\sm_{\nu \in G} X_\nu \,\rho_\varphi(\nu^{\!-1}) \Big)= \sm_{\nu \in G} X_{\pi\nu} \,\rho_\varphi(\nu^{\!-1}) =
\Big (\sm_{\nu \in G} X_\nu \,\rho_\varphi(\nu^{\!-1})\Big )\,\rho_\varphi (\pi) . $$
Ensuite, comme le d\'eterminant de $\rho_\varphi(\pi) \in {\rm End} (V_\varphi)$
est celui d'une matrice diago\-nale dont la diagonale est form\'ee de racines de l'unit\'e, on a~:
$$ {\rm det} \Big([\pi]\Big (\sm_{\nu \in G} X_\nu \rho_\varphi(\nu^{\!-1})\Big ) \Big) = 
\zeta_\pi \, {\rm det} \Big(\sm_{\nu \in G} X_\nu \rho_\varphi(\nu^{\!-1})\Big), $$

\vspace{-0.1cm}
o\`u $\zeta_\pi$ est d'ordre diviseur de l'ordre de $\rho_\varphi(\pi)$ lequel est un diviseur de $g$.
\end{proof}

\begin{coro} \label{coro3} 
Pour tout $\pi \in G$ et tout caract\`ere absolument irr\'eductible $\varphi$,  $P^\varphi (\ldots, \alpha^{\pi\nu }, \ldots) = \zeta_\pi\, P^\varphi (\ldots, \alpha^\nu, \ldots)$ par sp\'ecialisation $X_\nu \mapsto \alpha^\nu$, $\alpha \in Z_K$.

\smallskip
Par cons\'equent, $P^\chi (\ldots, \alpha^{\pi\nu}, \ldots) = \pm P^\chi (\ldots, \alpha^\nu, \ldots)$ pour tout $\pi \in G$.\,\footnote{\,Signe $+$ sauf si $\chi = \varphi $ est quadratique et $\varphi (\pi) = -1$
 (cf. Lemme \ref{lemm7}\,(ii),\,\S\ref{sub4}).}

De m\^eme, $P^\theta (\ldots, \alpha^{\pi\nu}, \ldots) = \zeta'_\pi P^\theta (\ldots, \alpha^\nu, \ldots)$  pour tout $\pi \in G$,
o\`u $\zeta'_\pi$ est d'ordre diviseur de p.g.c.d.\,$(g, p-1)$. 
\end{coro}

\subsubsection{D\'eterminants, $\chi$-d\'eterminants num\'eriques}\label{sub111}
 Dans cette partie, il n'y a pas r\'ef\'erence \`a un nombre premier $p$ et les caract\`eres consid\'er\'es sont absolument irr\'eductibles ou rationnels. Ce qui pr\'ec\`ede conduit \`a d\'efinir les $\chi$-d\'eterminants num\'eriques (i.e., ind\'ependants 
de la donn\'ee de $\eta \in K^\times$), \`a partir d'un $\alpha \in Z_K$ quelconque (ce ne sont donc pas des r\'egulateurs au sens classique du terme).
 
\begin{defi} \label{defi4}  Soit $G$ un groupe fini et soit~:
$${\rm Det}^G(X) = {\rm det}(X_{\tau\sigma^{\!-1}})_{\sigma, \tau \in G}  = \prd_\varphi {\rm det} \big( \sm_{\nu \in G} X_\nu\, \rho_\varphi(\nu^{\!-1}) \big)^{\varphi(1)} = \prd_\varphi P^\varphi(X)^{\varphi(1)} $$

\vspace{-0.1cm}
 le d\'eterminant de groupe associ\'e.
Les $\chi$-d\'eterminants (avec ind\'etermin\'ees et num\'eriques) sont par d\'efini\-tion les expressions (sur les conjugu\'es de~$\alpha$)~:
$${\rm Det}^\chi (X)= \prd_{\varphi \div \chi} P^\varphi (X) \ \,  {\rm et } \ \,
{\rm Det}^\chi (\alpha) = \prd_{\varphi \div \chi} P^\varphi (\ldots, \alpha^\nu,\ldots) , \ \, \alpha \in Z_K , $$

\vspace{-0.2cm}
de sorte que ${\rm Det}^G(\alpha) = \prd_\chi \, ({\rm Det}^\chi (\alpha))^{\varphi(1)}$ (o\`u pour chaque $\chi$, $\varphi \div \chi$).
\end{defi}
 
\vspace{-0.45cm}

\begin{example}\label{ex5}  {\rm Dans le cas du groupe $D_6 = \{1, \sigma, \sigma^2, \tau, \tau\sigma, \tau\sigma^2  \}$, on a les 
$\chi$-d\'eterminants num\'eriques suivants~:

\smallskip
${\rm Det}^1(\alpha) = \alpha +\alpha^{\sigma} +\alpha^{\sigma^2} +\alpha^{\tau} +\alpha^{\tau\sigma} +\alpha^{\tau\sigma^2}$,

\smallskip
${\rm Det}^{\chi_1}(\alpha) = \alpha +\alpha^{\sigma} +\alpha^{\sigma^2} -\alpha^{\tau} -\alpha^{\tau\sigma} -\alpha^{\tau\sigma^2}$,

\smallskip
${\rm Det}^{\chi_2}(\alpha) = \alpha^{2} +\alpha^{2}{}^{\sigma} +\alpha^{2}{}^{\sigma^2} -\alpha^{2}{}^{\tau} -\alpha^{2}{}^{\tau\sigma} -\alpha^{2}{}^{\tau\sigma^2}  - \alpha  \alpha^{\sigma} - \alpha^{\sigma}  \alpha^{\sigma^2} -\alpha^{\sigma^2}  \alpha$

\hfill $+ \alpha^{\tau}  \alpha^{\tau\sigma}+ \alpha^{\tau\sigma}  \alpha^{\tau\sigma^2} +
\alpha^{\tau\sigma^2}  \alpha^{\tau}$.

\smallskip
Les deux derniers sont de la forme ${\rm Det}'\,\sqrt m$, ${\rm Det}' \in \Q$, o\`u $k = \Q(\sqrt m)$ est le sous corps quadratique de $K$ et on n\'eglige le facteur $\sqrt m$~; mais ${\rm Det}^{\chi_2}(\alpha)$ figure au carr\'e dans le d\'eterminant ${\rm Det}^G(\alpha)$ et le r\'esultat est rationnel, ce qui n'est pas le cas de ${\rm Det}^{\chi_1}(\alpha)$. Ceci est sp\'ecifique des seuls  caract\`eres quadratiques. }
\end{example}

Pour les calculs, on peut revenir aux r\'ealisations matricielles (ici $C = \Q$)~:
$$\rho_\varphi(1) = \left(\begin{matrix} 
1&0\\
0&1\end{matrix}\right), \ \ \rho_\varphi(\sigma) = \left(\begin{matrix} 
-1&-1\\
1&0\end{matrix}\right), \ \ \rho_\varphi(\sigma^2) = \left(\begin{matrix} 
0&1\\
-1&-1\end{matrix}\right), $$
$$\rho_\varphi(\tau) = \left(\begin{matrix} 
1&0\\
-1&-1\end{matrix}\right), \ \ \rho_\varphi(\tau\sigma) = \left(\begin{matrix} 
-1&-1\\
0&1\end{matrix}\right),\ \  \rho_\varphi(\tau\sigma^2) = \left(\begin{matrix} 
0&1\\
1&0\end{matrix}\right), $$

qui conduisent \`a~:

\vspace{-0.4cm}
$$\sm_{\nu \in G}X_\nu \rho_\varphi(\nu^{\!-1}) = \left(\begin{matrix} 
X_1-X_{\sigma^2}  +X_{\tau}-X_{\tau\sigma}  & X_{\sigma} - X_{\sigma^2}-X_{\tau\sigma} +X_{\tau\sigma^2}     \\
-X_{\sigma} + X_{\sigma^2}-X_{\tau} + X_{\tau\sigma^2}  & X_1-X_{\sigma}  -X_{\tau} + X_{\tau\sigma}      \end{matrix}\right), $$
ou encore, par sp\'ecialisation et en prenant le d\'eterminant~:
$${\rm Det}^{\chi_2}(\alpha) = \left \vert\begin{matrix} 
\alpha-\alpha^{\sigma^2}  +\alpha^{\tau}-\alpha^{\tau\sigma}  & \alpha^{\sigma} - \alpha^{\sigma^2}-\alpha^{\tau\sigma} +\alpha^{\tau\sigma^2}     \\
-\alpha^{\sigma} + \alpha^{\sigma^2}-\alpha^{\tau} + \alpha^{\tau\sigma^2}  & \alpha-\alpha^{\sigma}  -\alpha^{\tau} + \alpha^{\tau\sigma}      \end{matrix}\right \vert. $$

Toujours pour le caract\`ere $\chi_2$ (de degr\'e 2) et la repr\'esentation $e_{\chi_2}\,\Q[G] \simeq2\,V_\varphi$, il existe deux projecteurs orthogonaux $\pi_1, \pi_2$, de somme $e_{\chi_2} = \frac{1}{3}(2 - \sigma - \sigma^2)$ (cf.\,\S\ref{subnot}),
ce qui donne~ici~:
$$\pi_1 = \Frac{1}{3}(1-\sigma^2 + \tau - \tau\sigma)\ \ \& \ \  \pi_2 = \Frac{1}{3}(1-\sigma - \tau + \tau\sigma). $$

\subsubsection{Cas des groupes Ab\'eliens}\label{sub4}
Ce cas se d\'eduit des \'etudes pr\'ec\'edentes, mais il est possible d'\^etre plus pr\'ecis
et d'en rester au point de vue classique des caract\`eres de degr\'e 1.

\smallskip
Soit $G$ un groupe Ab\'elien fini d'ordre $n$.
On d\'esigne par $\wh G$ le groupe des caract\`eres $\varphi : G \to \C^\times$~; l'image de $\varphi$ est le groupe $\mu_d$ des racines de l'unit\'e d'ordre \'egal \`a l'ordre $d$ de~$\varphi$.
On a alors
${\rm Det}^G(X) = {\rm det} (X_{\tau\sigma^{\!-1}})_{\sigma, \tau \in G} = \prd_{\varphi \in \wh G}\, \Big(
\sm_{\nu\in G} \varphi (\nu) X_{\nu^{\!-1}} \Big)$.

Si $G$ est produit direct de groupes cycliques $G_i$, alors en \'ecrivant que 
$\wh G = \prd_i \wh G_i$, il vient
${\rm Det}^G(X) = \prd_i  \prd_{\varphi \in \wh G_i}\Big(
\sm_{\nu\in G_i} \varphi (\nu) X_{\nu^{\!-1}} \Big)$.
Il en r\'esulte que l'on peut se ramener au cas o\`u $G$ est cyclique d'ordre $n$. 
On a alors une d\'ecomposition par classes de caract\`eres de m\^eme ordre $d\div n$ (i.e., $\Q$-conjugu\'es dans 
$\Q(\mu_d)$)~:
on obtient les caract\`eres rationnels $\chi$ qui regroupent les $\varphi^t$, o\`u $\varphi$ est fix\'e d'ordre $d$ et
o\`u $t$ parcourt $(\Z/d\Z)^\times$. On peut donc \'ecrire~:
$${\rm Det}^G(X)= \prd_{\chi } \prd_{\varphi \div \chi} \Big(\sm_{\nu\in G} \varphi (\nu) X_{\nu^{\!-1}} \Big). $$

 Soit $G$ un groupe cyclique d'ordre $n \geq 1$ et soit $\chi$ un caract\`ere rationnel irr\'eductible de $G$ d'ordre $d \div n$
(par abus, l'ordre commun des $\varphi \div \chi$).
Dans ce cas les $\chi$-d\'eterminants (avec ind\'etermin\'ees et num\'eriques) sont les expressions~:
$${\rm Det}^\chi (X) := \prd_{\varphi \div \chi} \!\Big(\sm_{\nu\in G} \varphi (\nu) \, X_{\nu^{\!-1}} \Big) \, \ {\rm et} \, \ 
{\rm Det}^\chi(\alpha)  := \prd_{\varphi \div  \chi}\!\Big(\sm_{\nu\in G} \varphi (\nu) \, \alpha^{\nu^{\!-1}} \Big), \ \alpha \in Z_K. $$

\begin{lemm}\label{lemm7}
Soit $C_\chi = \Q(\mu_d)$ le corps  des valeurs des caract\`eres irr\'eductibles $\varphi \div \chi$ d'ordre $d$. 

\smallskip
(i) On a ${\rm Det}^\chi (X) = \No_{C_\chi /\Q} \Big( \sm_{\nu\in G} \varphi (\nu)\,X_{\nu^{\!-1}} \Big)
\in \Z[X]$.

\smallskip
(ii) Par sp\'ecialisation $X_\nu \mapsto \alpha^\nu$, $\alpha \in Z_K$, on obtient  ${\rm Det}^\chi(\alpha) \in \Z$, sauf si $\chi = \varphi$ 
est un caract\`ere quadratique correspondant \`a un sous-corps quadratique $\Q(\sqrt m) \subseteq K$, auquel cas 
${\rm Det}^\chi(\alpha)  = {\rm Det}'{}^\chi (\alpha)\, .\,  \sqrt m$, avec ${\rm Det}'{}^\chi (\alpha)\in \Z$.
\end{lemm}

\begin{proof} Le premier point est \'evident puisque le produit d\'efinissant ${\rm Det}^\chi(\alpha)$ porte sur tous les 
$\Q$-conjugu\'es de $\varphi$.\,\footnote{\,A ce sujet, signalons 
que les deux op\'erations ``conjuguaison dans $C_\chi /\Q$'' et ``\'evaluation par les $G$-conjugu\'es de $\alpha$'' 
ne commutent pas, sauf s'il y a disjonction lin\'eaire de $C_\chi $ et $K$ sur $\Q$. }

\smallskip
Les sommes $\sm_{\nu \in G} \varphi (\nu)\, X_{\nu^{\!-1}}$ sont des {\it r\'esolvantes  de Hilbert} pour lesquelles 
on a les relations
$[\tau] \Big ( \sm_{\nu \in G} \varphi (\nu)\, X_{\nu^{\!-1}} \Big ) = 
\varphi (\tau)\, \sm_{\nu \in G} \varphi (\nu)\, X_{\nu^{\!-1}},\ \ \hbox{pour tout $\tau \in G$}$ 
(cf. Lem\-me~\ref{lemm2},\,\S\ref{calp})
et $\No_{C_\chi /\Q} \Big( [\tau] \Big(\sm_{\nu\in G} \varphi (\nu)\,X_{\nu^{\!-1}} \Big) \Big)
= \No_{C_\chi /\Q} \Big( \sm_{\nu\in G} \varphi (\nu)\,X_{\nu^{\!-1}} \Big)$ si $C_\chi  \ne \Q$,
 puisque la norme d'une racine de l'unit\'e est 1. D'o\`u dans ce cas par sp\'ecialisation avec $\alpha \in Z_K$~:
$$\No_{C_\chi /\Q} \Big(\sm_{\nu\in G} \varphi (\nu)\,\alpha^{\tau{\nu^{\!-1}}} \Big) = 
\No_{C_\chi /\Q}\Big(\sm_{\nu\in G} \varphi (\nu)\,\alpha^{\nu^{\!-1}} \Big) = {\rm Det}^\chi(\alpha). $$

Donc, ${\rm Det}^\chi(\alpha)$, invariant par tout $s\in {\rm Gal}(C_\chi /\Q)$ et tout 
$\tau \in {\rm Gal}(K/\Q)$, est un entier rationnel.

\smallskip
Si la norme dans $C_\chi /\Q$ est triviale (auquel cas $\chi = \varphi$ est d'ordre 1 ou 2), on a les deux 
$\chi$-d\'eterminants num\'eriques possibles suivants~:

\smallskip
(i) $\ \,{\rm Det}^1(\alpha) = \sm_{\sigma \in G} \alpha^{\sigma^{\!-1}} \in \Z$, pour $\chi=1$.

(ii) ${\rm Det}^{\chi}(\alpha) = \sm_{\sigma \in G} \chi(\sigma)\, \alpha^{\sigma{\!-1}}$, pour $\chi$ quadratique~;
alors il existe un corps quadratique $k  = \Q(\sqrt m\,)\subseteq K$ (correspondant au noyau de $\chi$) et $\tau \in G$ 
tels que $\tau {\rm Det}^{\chi}(\alpha) = -{\rm Det}^{\chi}(\alpha)$~; d'o\`u 
${\rm Det}^\chi(\alpha) = {\rm Det}'{}^\chi (\alpha)\, .\,  \sqrt m$, o\`u ${\rm Det}'{}^\chi(\alpha) \in \Z$. 
\end{proof}

Le cas d'un caract\`ere quadratique \'etant la seule exception, on convient de n\'egliger le facteur $\sqrt m$
 et d'en rester \`a la notation ${\rm Det}^{\chi}(\alpha)$ pour le rationnel correspondant.

\begin{example}\label{ex8}
{\rm  Dans le cas $n=2$ ($K = \Q(\sqrt m)$), si $\alpha, \alpha'$ sont les conjugu\'es de 
$\alpha = u + v\, \sqrt m \in Z_K$ donn\'e, alors on a le d\'eterminant num\'erique~:

\footnotesize
\smallskip
${\rm Det}^G(\alpha) = \left \vert\begin{matrix} 
\alpha& \alpha'  \\
\alpha'& \alpha 
\end{matrix}\right \vert = (\alpha +\alpha' ) (\alpha -\alpha' )$,
\normalsize

\smallskip
et par d\'efinition les deux $\chi$-d\'eterminants~:
$${\rm Det}^1(\alpha) = \alpha +\alpha'  = 2u , \ \ {\rm Det}^\chi(\alpha)  = \alpha -\alpha' = 2v\, \sqrt m . $$

Si $\alpha \notin \Q$ est par exemple tel que $\alpha+\alpha' = 0$, le $1$-d\'eterminant ${\rm Det}^1(\alpha)$ 
est  nul. Par contre,  le $\chi$-d\'eterminant
${\rm Det}^\chi(\alpha) = 2 \alpha = 2v\, \sqrt m$ n'est pas nul (puisque $\alpha-\alpha' \ne 0$) 
et on pourra \'etudier ses \'eventuelles nullit\'es modulo $p$ pour tout $p$ assez grand. 

\smallskip
De m\^eme ${\rm Det}^\chi(\alpha)$ est nul si $\alpha \in \Z$, sans que ${\rm Det}^1(\alpha)$ ne soit nul si $\alpha \ne 0$. }
\end{example}

\begin{example}\label{ex9}
{\rm  Consid\'erons le cas $n=3$ ($K$ corps cubique cyclique) et soient $\alpha, \alpha', \alpha''$ les conjugu\'es de 
$\alpha\in Z_K$~; alors on a le d\'eterminant num\'erique~:

\footnotesize
\smallskip
$\left \vert\begin{matrix} 
\alpha& \alpha'& \alpha'' \\
\alpha'& \alpha''& \alpha \\
\alpha''& \alpha& \alpha' 
\end{matrix}\right \vert \  = (\alpha +\alpha' + \alpha'') \,\No_{\Q(j)/\Q}(\alpha +j^{-1}\alpha' + j^{-2}\alpha'')= 3\,\alpha \alpha'  \alpha'' - (\alpha^3 +\alpha'{}^3 + \alpha''{}^3). $

\normalsize
\medskip
On a donc par d\'efinition les deux $\chi$-d\'eterminants~:

\smallskip
${\rm Det}^1(\alpha) \ = \alpha +\alpha' + \alpha'',   $

\smallskip
${\rm Det}^\chi(\alpha) = \No_{\Q(j)/\Q}(\alpha +j^{-1}\alpha' + j^{-2}\alpha'') =
\alpha^2 +\alpha'{}^2 + \alpha''{}^2 -\alpha\alpha' - \alpha'\alpha''  -  \alpha''\alpha . $

\smallskip
Si par exemple ${\rm Det}^1(\alpha) = 0$ avec $\alpha \notin \Z$, alors~:
$${\rm Det}^\chi(\alpha) = \No_{\Q(j)/\Q}(\alpha +j^{-1}\alpha' + j^{-2}\alpha'') = 
3\,(\alpha^2+\alpha \alpha'+ \alpha'{}^2)\ne 0. $$ }
\end{example}

\subsubsection{Crit\`ere de nullit\'e triviale des $\chi$-r\'egulateurs locaux}\label{sub5}
 Soit $\eta \in K^\times$ et soit $F$ le $\Z[G]$-module engendr\'e par $\eta$.

\begin{rema}\label{rema51}
Dans la d\'ecomposition ${\rm Det}^G(\alpha) = \prod_\chi ({\rm Det}^\chi(\alpha))^{\varphi(1)}$ et
lorsque la sp\'eciali\-sation est donn\'ee par $\alpha \equiv \alpha_p(\eta) \equiv \frac{-1}{p}{\rm log}_p(\eta)\pmod p$, 
certains $\chi$-r\'egulateurs locaux $\Delta_p^\chi(\eta) := {\rm Det}^\chi(\alpha)$ sont nuls modulo $p$ 
pour tout $p$ d\`es qu'il existe une relation  multiplicative de la forme 
$\ \prod_{\nu \in G}(\eta^{\nu^{\!-1}})^{\lambda(\nu)} = 1$, $\lambda(\nu) \in \Z$,  qui conduit \`a
$\sum_{\nu \in G} \lambda (\nu) \,\alpha^{\nu^{\!-1}} \equiv 0 \pmod p$ pour tout $p$ \'etranger \`a $\eta$.

\smallskip
Cette relation globale $U = \sum_{\nu \in G} \lambda (\nu) \,\nu^{\!-1} \ne 0$ de $\Z[G]$
 (ind\'ependante de $p$) est \`a distinguer
des relations locales $\sum_{\nu \in G} u (\nu) \,{\nu^{\!-1}} \in \Z_{(p)}[G]$ (o\`u les $u (\nu)$ d\'ependent de $p$) lorsque
$\sum_{\nu \in G}u (\nu) \,\alpha^{\nu^{\!-1}} \equiv 0 \pmod p$ pour certains $p$.
On rappelle l'aspect  ``r\'eciproque''  (cf.  Remarque \ref{global},\,\S\ref{reg}), purement conjectural.
\end{rema}

\begin{lemm} \label{nul} Si l'on a ${\rm dim}_\Q ((F \otimes \Q)^{e_\chi} )< {\rm dim}_\Q (e_\chi \Q[G]) = [C_\chi : \Q]\, \varphi(1)^2$ 
(i.e., il existe $U\in \Q[G]$ telle que $\eta^{U_\chi} = 1$, avec $U_\chi \ne 0$), alors
les $\chi$-r\'egulateurs locaux $\Delta_p^\chi(\eta) :={\rm Det}^\chi(\alpha)$ sont nuls modulo~$p$ pour tout $p$ 
 assez grand.\,\footnote{Il suffit de remarquer que toute relation globale 
$\eta^{U_\chi} = 1$ dans $F\otimes \Q$, $U_\chi \ne 0$, permet d'en d\'eduire que, pour tout $p$ assez grand, on est dans les conditions d'application des Lemmes du\,\S\ref{sub9} (crit\`ere de nullit\'e modulo $p$ des $\Delta_p^\theta(\eta)$) pour $\theta \div \chi$, \`a partir chaque fois de $\alpha \equiv \alpha_p(\eta) \pmod p$. }
\end{lemm}

Dans ce cas, on dira que les $\Delta_p^\chi(\eta)$ sont {\it trivialement nuls modulo $p$}. Ceci entra\^ine la 
nullit\'e triviale de certains $\Delta_p^\theta(\eta)$, $\theta \div \chi$, \`a savoir ceux pour lesquels 
$U_\theta := e_\theta U \not\equiv 0 \pmod {\mathfrak p}$, pour ${\mathfrak p} \div p$ associ\'e \`a $\theta$.

\begin{rema}\label{rema511}
(i)  Si $\varphi(1) = 1$, $\Delta_p^\chi(\eta)$ est trivialement nul modulo $p$ si $\eta^{e_\chi} = 1$ (i.e.,  $U_\chi = e_\chi$),
auquel cas $\Delta_p^\theta(\eta) \equiv 0 \!\!\pmod p$ trivialement pour tout~$\theta \div \chi$.

\smallskip
(ii) Si $\varphi(1) > 1$, $\Delta_p^\chi(\eta)$ est  nul modulo $p$ s'il existe $i$, $1 \leq i \leq \varphi(1)$, 
tel que, par extension des scalaires, on ait, dans $F \otimes C_\chi$, $\eta^{\pi_i^\varphi} = 1$ pour $\varphi \div \chi$ (cf.\,\S\ref{subnot}).

\smallskip
Par exemple, pour $G = D_6$ et $\varphi =\chi= \chi_2$, les \'el\'ements $\pi_1^\varphi = \frac{1}{3}(1-\sigma^2 + \tau -\tau\sigma)$ et $\pi_2^\varphi = \frac{1}{3}(1-\sigma - \tau +\tau\sigma)$ sont
tels que $e_\chi \pi_i^\varphi = \pi_i^\varphi$, pour $i=1, 2$, $\pi_1^\varphi+\pi_2^\varphi = e_\chi$, et $\pi_1^\varphi\pi_2^\varphi = 0$ (cf. Exemple \ref{ex5},\,\S\ref{sub111}). 

\smallskip
 Ainsi on pourrait avoir la $\varphi$-relation non triviale $\eta^{U_1} :=  \eta^{1-\sigma^2 + \tau -\tau\sigma} = 1$
tandis que $\eta^{U_2} := \eta^{1-\sigma - \tau +\tau\sigma} \ne 1$ 
(i.e., ${\rm dim}_\Q (F \otimes \Q)^{e_\chi}=2$ pour ${\rm dim}_\Q (e_\chi \Q[G])=4$))~; on aurait donc
$\eta^{e_\chi .\,(U_1+U_2)}  = \eta^{3 e_\chi}  =\eta^{3 U_2} \ne 1$, or on v\'erifie que le $\chi$-r\'egulateur
$\Delta_p^\chi(\eta)$ est nul modulo $p$ pour tout $p$ assez grand en raison de la premi\`ere relation.

\smallskip
 Le fait de supposer $F$ de $\Z$-rang $n$ \'evite cet inconv\'enient. On peut toujours s'y ramener 
en modifiant convenablement $\eta$.

\smallskip
(iii) Pour $U \in \Z_{(p)}[G]$, on a $U_\chi = \sum_{\varphi \div \chi} U_\varphi$  et $U_\varphi = e_\varphi U_\chi$.
On a $U_\chi  \equiv 0 \pmod p$ si et seulement si $U_\varphi  \equiv 0 \pmod p$ pour au moins un (donc tout)
$\varphi \div \chi$ (car les $\varphi \div \chi$ sont conjugu\'es par ${\rm Gal}(C_\chi/\Q)$).
Ces congruences (mod $p$) dans les alg\`ebres de groupes signifient selon les cas $\!\!\!\pmod{p \,\Z_{(p)}[G]}$ ou
$\!\!\!\pmod{p \,Z_{\chi,(p)} [G]}$ (o\`u $Z_\chi$ est l'anneau des entiers du corps des valeurs $C_\chi$).

\smallskip
Ceci n'a pas lieu pour $U_\chi = \sum_{\theta\div \chi} U_\theta$ et $U_\theta = e_\theta U_\chi$ car 
$U_\theta \equiv 0 \pmod p$ (dans $\Z_p[G]$) signifie $U_\theta \equiv 0 \pmod {\mathfrak p}$
dans $L[G]$, seulement \'equivalent \`a $U_\varphi \equiv 0 \pmod  {\mathfrak p}$ pour tout $\varphi \div \theta$
(cf. D\'efinition \ref{defi01}\,(ii), \S\ref{calp}, et \ref{defi11}\,(iii), \S\ref{sub155}).
\end{rema}

\begin{example}\label{ex1} {\rm
a) Soit $G$ cyclique d'ordre $n$ et soit $\chi$ d'ordre $d \div n$~; alors
les \'el\'ements $\eta \in K^\times$ tels que $\eta^{e_\chi} = 1$ correspondent \`a la nullit\'e triviale modulo $p$ de
$\Delta_p^\chi(\eta) = \No_{C_\chi /\Q} \big( \sum_{\nu\in G} \varphi (\nu)\,\alpha^{\nu^{\!-1}} \big)$
pour $\alpha\equiv \alpha_p (\eta) \pmod p$.

\smallskip
Par exemple, pour $n = 3$, on a les deux idempotents rationnels~:
$$e_1= \Frac{1}{3} (1+ \sigma + \sigma^2), \ \ 
e_\chi= \Frac{1}{3} (2-\sigma- \sigma^2) . $$
Soit $\alpha \equiv \alpha_p (\eta)\pmod p$ et consid\'erons les nullit\'es triviales des $\Delta_p^\chi(\eta)$ modulo~$p$.

\smallskip
(i) Les \'el\'ements $\eta \in K^\times$ tels que $\eta^{e_1} = 1$
(i.e., de norme 1 dans $F\otimes\Q$), correspondent \`a la nullit\'e triviale de
$\Delta_p^1 (\eta) = \alpha + \alpha^\sigma+\alpha^{\sigma^2}$.

\smallskip
(ii) Les \'el\'ements $\eta \in K^\times$ tels que $\eta^{e_\chi} = 1$ ou $\No_{K/\Q}(\eta) = \eta^3$
(dans $F\otimes\Q$), donc tels que $\eta \in \Q^\times$, correspondent \`a la nullit\'e de
$\Delta_p^\chi(\eta) = \No_{\Q(j)/\Q}(\alpha +j^2\,\alpha^{\sigma^{1}} + j\,\alpha^{\sigma^{2}})$.

\medskip
b) Les trois idempotents relatifs au groupe $D_6$ sont~:
\footnotesize
\begin{eqnarray*}
e_1 &=& \Frac{1}{6} (1+ \sigma + \sigma^2 + \tau + \tau \sigma +\tau \sigma^2), \\
e_{\chi_1} &=&  \Frac{1}{6} (1+ \sigma + \sigma^2 - (\tau + \tau \sigma +\tau \sigma^2)), \\
e_{\chi_2} &=&  \Frac{2}{6} (2 - \sigma - \sigma^2).
\end{eqnarray*}

\normalsize
(i) Les  $\eta$ tels que $\eta^{e_1} = 1$ correspondent \`a la nullit\'e triviale de
$\Delta_p^1 (\eta) = {\rm Tr}_{K/\Q} (\alpha)$.

\smallskip
(ii)  Les \'el\'ements $\eta$ tels que $\eta^{e_{\chi_1} }= 1$ sont tels que $\No_{K/k}(\eta) \in\Q^\times$, o\`u $k$ est le sous-corps quadratique de $K$, et correspondent \`a la nullit\'e triviale de~:

\smallskip
$\Delta_p^{\chi_1} (\eta) =  \alpha +\alpha^{\sigma} +\alpha^{\sigma^2} -\alpha^{\tau} -\alpha^{\tau\sigma} -\alpha^{\tau\sigma^2} = (1-\tau)\, {\rm Tr}_{K/k} (\alpha)$.

\smallskip
(iii) Les \'el\'ements $\eta$ tels que $\eta^{U_{\chi_2}}=1$ pour $U_{\chi_2} \in e_{\chi_2} \Q[G]\,\Sauf \{0\}$
 (Lemme \ref{nul} pr\'ec\'edent) conduisent \`a la nullit\'e triviale de~: 

\smallskip
$\Delta_p^{\chi_2} (\eta) = \alpha^{2} +\alpha^{2}{}^{\sigma} +\alpha^{2}{}^{\sigma^2} -\alpha^{2}{}^{\tau} -\alpha^{2}{}^{\tau\sigma} -\alpha^{2}{}^{\tau\sigma^2}  - \alpha  \alpha^{\sigma} - \alpha^{\sigma}  \alpha^{\sigma^2} -\alpha^{\sigma^2}\alpha$

\hfill$ + \alpha^{\tau}  \alpha^{\tau\sigma} + \alpha^{\tau\sigma}  \alpha^{\tau\sigma^2} +\alpha^{\tau\sigma^2}  \alpha^{\tau}$. }
\end{example}

\subsection{D\'efinition et \'etude des $\theta$-r\'egulateurs locaux}\label{sub6}
Soit $\eta \in K^\times$ et soit $p$ assez grand. On a les congruences suivantes (cf.\,\S\ref{sub1})~:
$$\Frac{-1}{p}\, {\rm log}_p(\eta) \equiv  \alpha_p(\eta)\!\! \!\pmod{p} \ \ {\rm et}\ \  \alpha_p(\eta^\nu) \equiv \alpha_p(\eta)^\nu \!\!\!\pmod p \ \  {\rm pour \ tout}\   \nu \in G. $$

\subsubsection{G\'en\'eralit\'es}\label{gene1}
On fixe un nombre alg\'ebrique $\alpha \in Z_K$ d\'efini par $\alpha \equiv \alpha_p(\eta) \pmod p$.
 On obtient le d\'eterminant  \`a coefficients dans $Z_K$, uniquement d\'efini modulo $p$,
mais repr\'esent\'e dans $Z_K$~:
$$\Delta^G_p(\eta) =  {\rm Det}^G(\alpha) = {\rm det} \big( \alpha^{\tau\sigma^{\!-1}}\big)_{\sigma, \tau \in G}. $$

 Le groupe $G$ op\`ere sur $\Delta^G_p(\eta)$ par permutations des lignes ; donc $\Delta^G_p(\eta)^2$
est invariant par Galois et est donc rationnel. Si $\Delta^G_p(\eta) \notin \Q$, on retrouve l'existence d'un facteur $\sqrt m$
que l'on sait provenir de la r\'esolvante d'un caract\`ere quadratique de $G$ (cf. Lemme \ref {lemm7},\,\S\ref{sub4}) et que l'on n\'eglige.
On consid\`ere donc~:

\vspace{-0.2cm}
$$\Delta^G_p(\eta) =  \prd_\chi \ \prd_{\theta \div \chi} \ \prd_{\varphi \div \theta} P^{\varphi} 
(\ldots, \alpha^\nu,\ldots )^{\varphi(1)}
=: \prd_\chi  \Delta^\chi_p(\eta)^{\varphi(1)}  =: \prd_\chi \ \prd_{\theta \div \chi} \Delta^\theta_p(\eta)^{\varphi(1)} , $$

o\`u pour chaque $\chi$ (resp. $\theta$), $\varphi$ est un caract\`ere  irr\'eductible divisant $\chi$  (resp. $\theta$).

\begin{defi}\label{defi10}
 Pour tout $p$ assez grand et 
 pour chaque caract\`ere $\Q_p$-irr\'eductible $\theta$ de $G$, on appelle $\theta$-r\'egulateur local de $\eta$, l'entier $p$-adique d\'efini par $\Delta_p^\theta(\eta) := \prd_{\varphi \div \theta} P^{\varphi} (\ldots, \alpha^\nu,\ldots )$ pour $\alpha \equiv \alpha_p(\eta) := \frac{1}{p} (\eta^{p^{n_p}-1} -1)\pmod p$.

Pour $\theta \div \chi$ ($\chi$ fix\'e), les $\theta$-r\'egulateurs locaux corres\-pondants d\'ependent de 
la d\'ecomposition de $p$ dans $C_\chi/\Q$ et sont au nombre de $h = \frac{[C_\chi : \Q]}{f}$, o\`u $f$ est 
leur degr\'e r\'esiduel (cf. D\'efinition \ref{defi01} (ii),\,\S\ref{calp}). Ils ne sont d\'efinis que modulo $p$.
\end{defi}

\begin{rema}\label{rema110} On peut de la m\^eme fa\c con \'ecrire (pour $p$ assez grand) que le r\'egulateur normalis\'e
${\rm Reg}_p^G(\eta)$ est \'egal \`a $\prd_\chi {\rm Reg}_p^\chi(\eta)^{\varphi(1)}    = \prd_\theta {\rm Reg}_p^\theta(\eta)^{\varphi(1)}$ o\`u~:
\vspace{-0.1cm}
$${\rm Reg}_p^\theta(\eta) = \prd_{\varphi \div \theta} P^{\varphi}\big  (\ldots, \hbox{$\frac{-1}{p}$} {\rm log}_p(\eta^\nu) ,\ldots\big ). $$

\vspace{-0.1cm}
On a alors les congruences ${\rm Reg}_p^\theta(\eta) \equiv \Delta_p^\theta(\eta) \pmod p$~; ainsi on a 
$p$ divise ${\rm Reg}_p^\theta(\eta)$ si et seulement si $\Delta_p^\theta(\eta) 
\equiv 0 \pmod p$~; mais dans ce cas $p^{\varphi(1)}$ divise ${\rm Reg}_p^G(\eta)$
puisque ${\rm Reg}_p^G(\eta) = \prd_\theta {\rm Reg}_p^\theta(\eta)^{\varphi(1)}$, o\`u chaque fois $\varphi\div \theta$ 
(cf.\,\S\ref{sub3}).
\end{rema}

\begin{rema}\label{rema11} (i) Dans tous les cas, 
$\Delta_p^\chi (\eta) :=\No_{C_\chi/\Q}\big(  P^{\varphi} (\ldots, \alpha^\nu ,\ldots ) \big ) \in \Z$
($\varphi \div \chi$ fix\'e), avec la convention sur la notation $\No_{C_\chi/\Q}$, 
notamment lorsque $K$ et $C_\chi$ ne sont pas lin\'eairement disjoints. On rappelle l'exception $\chi$ quadratique.

\smallskip
De m\^eme $\Delta_p^\theta (\eta) :=\No_{\mathfrak p}\big  ( P^{\varphi}  (\ldots, \alpha^\nu ,\ldots ) \big )$,
o\`u pour ${\mathfrak p} \div p$ dans $L$ (ou $C_\chi$), ${\mathfrak p}$ associ\'e \`a $\theta$,
$\No_{\mathfrak p} $ d\'esigne la norme locale absolue
dans le compl\'et\'e de $C_\chi$ en ${\mathfrak p}$~; on retrouve $\Delta_p^\chi (\eta)$ comme produit des normes locales correspondantes.

\smallskip
 M\^emes relations en rempla\c cant $\Delta$ par ${\rm Reg}$ et $\alpha$ par $\frac{-1}{p} {\rm log}_p(\eta)$.

\medskip
(ii) Si $H = \{\nu \in G,\, \varphi(\nu) = \varphi(1) \}$ est le noyau de $\varphi \div \theta \div \chi$ (qui ne d\'epend que de~$\chi$) 
et si $K'$ est le sous-corps de $K$
fixe par $H$, on a $\Delta_p^\theta (\eta)= \Delta_p^{\theta'} (\No_{K/K'}(\eta) )$ o\`u $\theta'$ est le caract\`ere fid\`ele
issu de $\theta$. 
Quitte \`a remplacer $\eta$ par $\eta' := \No_{K/K'}(\eta)$ on peut toujours supposer que $\theta$ 
est un caract\`ere fid\`ele.
\end{rema}

\subsubsection{Cas des $\chi$-r\'egulateurs locaux, $\chi$ de degr\'e 1,  d'ordre 1 ou 2}\label{sub7}
Soit $\eta \in K^\times$ et soit $\alpha \equiv  \alpha_p(\eta) \pmod p$, $\alpha \in Z_K$.

\medskip
 (i) Si $\chi =\theta =1$, le $\theta$-r\'egulateur  correspond \`a  $\No_{K/\Q}(\eta) = a \in \Q^\times$ et est donn\'e par 
${\rm Tr}_{K/\Q}(\alpha)$, autrement dit~:
$$\Delta^1_p (\eta) \equiv \Frac{-1}{p} {\rm log}_p(a) \equiv \Frac{1}{p}(a^{p-1} - 1) \equiv q_p \pmod p \ \  
\hbox{($p$-quotient de Fermat  de $a$)}.$$

Le programme PARI correspondant est le suivant (pour $a = 659$)~:\,\footnote{\,Dans tous les programmes PARI \cite{P} propos\'es, la compatibilit\'e avec TeX oblige \`a \'ecrire les symboles {\it par}, $\&$, {\it hfill},  avec un antislash, \`a placer 
des {\it \$} et des {\it  \{  \}} pour les exposants\ldots}
\footnotesize 

\medskip
$\{$$e=659; print("e = ",e);  for(n=1, 5*10^8, p=2*n+1; $\par
$if(isprime(p)==1 \& Mod(e,p)!=0,  $\par
$qp= Mod(e,p^2)^{(p-1)}-1;  if(qp==0, print(p)) ))$$\}$
\normalsize

\medskip
Pour $p< 10^9$, on ne trouve que les solutions $p = 23, 131,2221, 9161, 65983$.

\medskip
(ii)  Si $\chi=\theta$ est quadratique et si $k= \Q(\sqrt m\,)$ est le sous-corps quadratique de $K$ fixe par le noyau de $\chi$,
 on obtient un $\theta$-r\'egulateur qui correspond au cas o\`u  $\No_{K/k}(\eta) \in k^\times\,\Sauf\Q^\times$~; 
si ${\rm Tr}_{K/k}(\alpha) =: u + v\,\sqrt m \in k$, il est donn\'e par~:
$$\Delta^\theta_p (\eta) \equiv (1-\tau) ( u + v\,\sqrt m)\equiv 2v\,\sqrt m \pmod p . $$

 Si $K$ est un corps quadratique r\'eel d'unit\'e fondamentale $\varepsilon$, en raison de la relation de
d\'ependance multiplicative $\varepsilon^{1+\sigma} =\pm 1$, les  $1$-r\'egulateurs $\Delta_p^1(\varepsilon)$
sont tous nuls.  Le $\theta$-r\'egulateur pour le caract\`ere quadratique est
$\Delta^\theta_p(\varepsilon) \equiv 2\,v\,\sqrt m \pmod p$ (calcul\'e via $\varepsilon^{p^{n_p}-1}  \equiv 1 +p\, v\,\sqrt m \pmod{p^2}$). 

\smallskip
Au moyen du programme PARI ci-dessous on calcule le $\theta$-r\'egulateur $\Delta^\theta_p(\varepsilon)$
de l'unit\'e fondamentale $\ \varepsilon = 5+2\,\sqrt 6$, pour tout $p< 10^9$ ($p \ne 2, 3$), mais le programme 
est valable pour un corps quadratique arbitraire et $\eta \in K^\times \,\Sauf \,\Q^\times$ 
de norme quelconque (en modifiant $m, aa, bb$))~:

\medskip
\footnotesize 
$\{$$m=6; aa=5; bb=2; N=aa^2-m*bb^2 ; print("m = ",m,"   ","norme = ",N);  $\par
$for (n=1,  5*10^8, p=2*n+1; if ( isprime (p)==1  \&  Mod(m*N,p)!=0, p2= p^2;$\par
$a=Mod(aa,p2); b=Mod(bb,p2); P2=x^2- Mod(m,p2); $\par
$y = Mod(a+b*x,P2); z=y^{(p-1)}; U=z^{(p+1)}-1;  $\par
$RP=  component(U,2); rp=component(RP,2)\,\footnote{\,Attention, exceptionnellement on doit prendre 
le coefficient de $x = \sqrt m$ et non le terme constant.};$\par
$if (rp==0, print (p)) ))$$\}$.
\normalsize

\smallskip
On ne trouve un $\theta$-r\'egulateur $\Delta^\theta_p(\varepsilon)$ nul modulo $p$, avec $p<10^9$, 
que pour $p= 7, 523$, ce qui constitue une seconde observation sur la raret\'e du ph\'enom\`ene.

\medskip
Soit $\eta  = 7 + 2\,\sqrt 6$ de norme $25$. Le $\Z$-rang du groupe $F$ engendr\'e par $\eta$ est ici \'egal \`a 2 (pas de nullit\'es triviales). On v\'erifie que les $p$-quotients de Fermat $\Delta^1_p(\eta)$ de $25$ 
sont tous non nuls modulo $p$ dans l'intervalle test\'e.
Les solutions $p$ obtenues pour $\Delta^\theta_p(\eta) \equiv 0 \pmod p$, $\theta \ne 1$, sont $p=11, 37,  163, 4219$.

\smallskip
Pour $\eta  = 16 + \,\sqrt {123}$, de norme $133$, on trouve pour $\theta\ne 1$  les deux solutions $p = 5, 751$.

\smallskip
Pour  $\eta=1+5 \sqrt{-1}$ de norme $26$, on trouve pour $\theta\ne 1$ les deux solutions $p = 73, 12021953$.

Pour le nombre d'or $\Frac{1+ \sqrt 5}{2}$ on ne trouve aucune solution dans l'intervalle test\'e.

\vspace{-0.3cm}
\section{Existence de relations $\F_p$-lin\'eaires entre les conjugu\'es de $\alpha$}

Soit $\eta \in K^\times$ fix\'e et soit $p$ assez grand. Soit $\alpha_p(\eta) := \frac{1}{p}\big( \eta^{p^{n_p} - 1}-1 \big)\in Z_{K,(p)}$~; modulo $p$, on pourra toujours repr\'esenter $\alpha_p(\eta)$ dans $Z_K$. 
On s'int\'eresse \`a \'etablir la
relation entre la nullit\'e modulo $p$ de certains $\Delta_p^\theta(\eta)$ et l'existence de relations 
$\F_p$-lin\'eaires entre les conjugu\'es de $\alpha_p(\eta)$.

\smallskip
 On suppose implicitement que le $\Z$-rang de $F$ est \'egal \`a $n$.

\subsection{$\F_p$-ind\'ependance}\label{sub8}
Soit $\alpha \in K$ quelconque (donc $\alpha \in Z_{K,(p)}$ pour tout $p$ assez grand).
Nous dirons que les $\alpha^\nu$ sont $\F_p$-ind\'ependants si,  pour toute famille de coefficients $u(\nu) \in \Z_{(p)}$, la relation $\sm_{\nu\in G} u(\nu) \, \alpha^{\nu^{\!-1}} \equiv 0 \pmod p$, dans l'anneau $Z_{K,(p)}$, implique $u(\nu) \equiv 0 \pmod p$ pour tout $\nu \in G$. On dira aussi (cf. D\'efinition \ref{defi11},\,\S\ref{sub155}) que toute relation $U = \sm_{\nu\in G} u(\nu) {\nu^{\!-1}}$, associ\'ee \`a $\alpha$, est nulle modulo $p\,\Z_{(p)}[G]$. 
Ceci est ind\'ependant de la classe de $\alpha$ modulo $p$.

\smallskip
On a le r\'esultat \'el\'ementaire suivant o\`u l'on rappelle que ${\rm Det}^G (\alpha)= 
{\rm det}\big(\alpha^{\tau \sigma^{\!-1}} \big)_{\sigma, \tau}$~:

\begin{prop}\label{prop13} On suppose $p$ assez grand afin que $\alpha \in Z_{K,(p)}$.

\smallskip
  (i)  Les $\alpha^\nu$ sont $\F_p$-ind\'ependants si et seulement si $\alpha$ est une $\Z_{(p)}$-base normale de $Z_{K,(p)}$.

\smallskip
(ii) Les $\alpha^\nu$ sont $\F_p$-ind\'ependants si et seulement si ${\rm Det}^G(\alpha) $ est \'etranger \`a $p$.
\end{prop}

\begin{proof}
Si $\alpha$ est une $\Z_{(p)}$-base normale de $Z_{K,(p)}$, toute relation de la forme $\sum_{\nu\in G} u(\nu) \, \alpha^{\nu^{\!-1}} \equiv 0 \pmod p$, $u(\nu) \in \Z_{(p)}$, conduit \`a $u(\nu) \equiv 0 \pmod p$ pour tout $\nu \in G$.

\smallskip
 Supposons alors que les $\alpha^\nu$ soient $\F_p$-ind\'ependants et  qu'il existe une relation non triviale de d\'ependance
$\Q$-lin\'eaire entre les conjugu\'es de $\alpha$~; il en r\'esulte une relation de la forme 
$\sum_{\nu \in G}{r(\nu)}\alpha^{\nu^{\!-1}}\!\! = 0$ avec des entiers $r(\nu)$ non tous nuls, tels~que 

${\rm p.g.c.d. }(r(\nu)) = 1$~;
d'o\`u $r(\nu) \equiv 0 \pmod p$ pour tout $\nu \in G$ (absurde). 
Par cons\'equent $\alpha$ est d\'ej\`a une $\Q$-base normale de $K$. 

Si $\beta \in Z_{K,(p)}$, il existe des $r(\nu) \in \Z$ et un entier rationnel $d$ \'etranger \`a  ${\rm p.g.c.d. }(r(\nu))$ 
tels que $d\, \beta = \sum_{\nu \in G}{r(\nu)}\alpha^{\nu^{\!-1}}$. Il est clair que $p\notdiv d$ sinon tous les $r(\nu)$ 
sont divisibles par $p$. Ainsi $\alpha$ est une $\Z_{(p)}$-base normale de $Z_{K,(p)}$.

(ii) Supposons que les $\alpha^\nu$ soient $\F_p$-ind\'ependants~; comme $\alpha = \frac{1}{d} \beta$,
$\beta \in Z_K$, $d\in \Z$, $p\notdiv d$, on peut se ramener au cas d'un $\alpha$ entier.
Comme $p$ est assez grand, il ne divise pas le discriminant de $K/\Q$, et le discriminant de la $\Z_{(p)}$-base normale
$\alpha$, de  $Z_{K,(p)}$, est \'etranger \`a $p$ (en effet, si $Z_K$ est l'anneau des entiers de $K$, le conducteur 
${\mathfrak f} \in \Z$ tel que ${\mathfrak f} \, Z_K \subseteq \bigoplus_\nu \Z\,\alpha^\nu$ est non divisible par $p$ et les deux discriminants co\"incident \`a une unit\'e $p$-adique pr\`es).
Or le discriminant de la base normale $\alpha$ est le carr\'e du d\'eterminant de Frobenius ${\rm Det}^G(\alpha)  =
 {\rm det}\big(\alpha^{\tau \sigma^{\!-1}} \big)_{\sigma, \tau \in G}$.

Supposons ${\rm Det}^G(\alpha)$ \'etranger \`a $p$, et supposons qu'il existe des $\lambda(\sigma) \in \Z_{(p)}$, non tous divisibles par $p$, tels que $sum_{\sigma \in G} \lambda(\sigma)\, \alpha^{\sigma^{\!-1}}  \equiv 0 \pmod p$. Alors, par toutes les conjugaisons par $\tau \in G$, on obtient une relation $\Z_{(p)}$-lin\'eaire sur les lignes de la forme
$\sum_{\sigma \in G} \lambda(\sigma) (\ldots, \alpha^{\tau \sigma^{\!-1}} , \ldots)_\tau \equiv (\ldots, 0 , \ldots)_\tau \pmod p$,
d'o\`u ${\rm Det}^G(\alpha) \equiv 0 \pmod p$ (absurde).
\end{proof}

\begin{coro}\label{coro14} Soit $\eta \in K^\times$ et soit $\alpha = \alpha_p(\eta)$.
Si pour $p$ assez grand l'un au moins des $\theta$-r\'egulateurs locaux $\Delta_p^\theta(\eta)$
est nul modulo $p$, alors  les $\alpha^\nu$ ne sont pas $\F_p$-ind\'ependants et il existe au moins une relation $\F_p$-lin\'eaire de la forme $\sm_{\nu\in G} u(\nu) \, \alpha^{\nu^{\!-1}} \equiv 0 \pmod p$,  $u(\nu)  \in \Z_{(p)}$ non tous divisibles par $p$.
\end{coro}

Nous allons pr\'eciser les propri\'et\'es des $\theta$-composantes de $U = \sum_{\nu\in G} u(\nu) \,{\nu^{\!-1}}$, dite relation associ\'ee, et faire le lien avec la nullit\'e ou non des $\Delta_p^\theta(\eta)$ modulo $p$.

\subsection{Crit\`ere de nullit\'e modulo $p$ des $\Delta_p^\theta(\eta)$}\label{sub9}
Soit $C_\chi$ le corps des valeurs des caract\`eres absolument irr\'eductibles $\varphi \div \chi$
de $G$, pour un caract\`ere rationnel irr\'eductible $\chi$. On suppose pour simplifier que $K \cap C_\chi = \Q$. 
On notera $Z_{\chi,(p)} := Z_{C_\chi,(p)}$ l'anneau des valeurs, sur $\Z_{(p)}$, des $\varphi\div \theta \div \chi$.

\subsubsection{Principaux lemmes}\label{sub155}
On d\'esigne par $e_\varphi = \frac{\varphi(1)}{ n } \sum_{\sigma \in G} \varphi(\sigma^{\!-1})\,\sigma$ 
les idempotents centraux orthogonaux de $C_\chi [G]$, et par ${\rm End}_{C_\chi}(V_\varphi) \simeq e_\varphi \,C_\chi [G]$ les anneaux d'endomorphismes des repr\'esentations irr\'eductibles $V_\varphi$ correspondantes.

\smallskip
On a, par sommes de conjugu\'es convenables de $\varphi$,  les idempotents $p$-adiques et rationnels
$e_\theta = \frac{\varphi(1)}{ n } \sum_{\sigma \in G} \theta(\sigma^{\!-1})\,\sigma\in \Z_{(p)}[G]$ et $e_\chi = \frac{\varphi(1)}{ n } \sum_{\sigma \in G} \chi (\sigma^{\!-1})\,\sigma\in \Z_{(p)}[G]$ (cf. D\'efinition \ref{defi01},\,\S\ref{calp} utilisant le 
corps de d\'ecomposition $L$ de $p$ dans $C_\chi/\Q$ et $D = {\rm Gal}(C_\chi/L)$). 

\begin{defi}\label{defi11} (i) Si $\sum_{\nu \in G} u(\nu)\,\alpha^{\nu^{\!-1}} \equiv 0 \pmod p$, $u(\nu) \in \Z_{(p)}$
pour tout $\nu \in G$, on appelle {\it relation associ\'ee} l'\'el\'ement $U = \sum_{\nu \in G} u(\nu)\,{\nu^{\!-1}} \in \Z_{(p)}[G]$, et on d\'efinit
 les $\varphi$-relations $U_\varphi := e_\varphi .\,U\in Z_{\chi,(p)}[G]$, les $\theta$-relations $U_\theta := e_\theta .\,U\in \Z_p[G]$, 
ou encore les $\chi$-relations $U_\chi := e_\chi .\,U\in \Z_{(p)}[G]$.

\smallskip
Quel que soit l'anneau $\mathcal Z$ des coefficients, les \'ecritures ci-dessus ne sont d\'efinies que modulo~$p\,\mathcal Z$
(en raison de la d\'efinition de $\alpha \equiv \alpha_p(\eta) \pmod p$).

\smallskip
(ii) On d\'esigne par $\mathcal L$ le $G$-module des relations associ\'ees $U \in \Z_{(p)} [G]$ relatives \`a~$\alpha$ et $p$.
Par abus on peut aussi voir $\mathcal L$ dans $\F_p [G]$.

\smallskip
 Vu dans $\F_p [G]$, on a $\mathcal L = \{0\}$ si et seulement si
les $\alpha^\nu$ sont $\F_p$-ind\'ependants (cf.\,\S\ref{sub8}) et on a ${\mathcal L} = \F_p [G]$ si et seulement si 
$\alpha \equiv 0 \pmod p$.

\smallskip
(iii) Pour tout caract\`ere $p$-adique $\theta$, on d\'esigne par ${\mathcal L}^\theta \simeq \delta \, V_\theta$ la $\theta$-composante $e_\theta\,{\mathcal L}$, o\`u $V_\theta$ (de $\F_p$-dimension $f \varphi(1)$) est la repr\'esentation irr\'eductible de 
caract\`ere~$\theta$~; on a $0\leq \delta \leq \varphi(1)$. Soit ${\mathfrak p}\div p$ l'id\'eal premier de $C_\chi$ 
(ou $L$) associ\'e \`a $\theta$. On rappelle que $\theta(\nu)=\sum_{s\in D} \varphi^s (\nu) \in  Z_{L,(p)}$ 
(pour $\varphi \div \theta$) est d\'efini via $\theta(\nu) \equiv r_{\mathfrak p}  (\nu) \pmod {\mathfrak p}$, 
$r_{\mathfrak p}  (\nu) \in \Z$~; donc  si $U\in \Z_{(p)}[G]$, $U_\theta\in Z_{L,(p)}[G]$ est congrue modulo ${\mathfrak p}$
\`a un \'el\'ement de $\Z_{(p)}[G]$.
On verra donc $U_\theta \!\!\pmod p$ dans $\Z_{(p)}[G]$ ou dans $Z_{L,(p)}[G] \!\!\pmod {\mathfrak p}$ selon le contexte.
\end{defi}

Soit $U \in \Z_{(p)}[G]$ une relation~; alors on obtient $U_\varphi  = \sum_{\nu \in G} u_\varphi(\nu)\,\nu^{\!-1} \in  Z_{\chi,(p)}[G]$, avec $u_\varphi(\nu) = \frac{\varphi(1)}{ n }  \sum_{\tau\in G} \varphi(\tau^{\!-1}) u(\nu \tau)$.
On rappelle que l'on a $U_\chi := e_\chi U \equiv 0 \pmod p$ si et seulement si $U_\varphi  \equiv 0 \pmod p$ 
pour au moins un (donc tout) $\varphi \div \chi$. 
Par contre, $U_\theta \equiv 0 \pmod p$ dans $\Z_p$ (de fait $U_\theta \in {\mathfrak p}   Z_{L,(p)}$) si et seulement si $U_\varphi  \equiv 0 \pmod {\mathfrak p} $ pour un (donc tout) $\varphi \div \theta$. Donc $U_\theta\equiv 0 \pmod {\mathfrak p}$ n'entra\^ine pas $U_\chi \equiv 0 \pmod p$, car on ne peut conjuguer que par $D$.

\begin{lemm}\label{lemm15} Si $\sm_{\nu\in G} u(\nu)  \alpha^{\nu^{\!-1}} \!\!\! \equiv 0 \!\!\pmod p$, $u(\nu)\!\in \! \Z_{(p)}$,
et si $U = \!\sm_{\nu\in G} u(\nu)\,\nu^{\!-1}$, alors $U_\varphi\,.\, \alpha := \sm_{\nu\in G} u_\varphi(\nu)\,\alpha^{\nu^{\!-1}} 
\equiv 0 \pmod p$ pour tout caract\`ere irr\'eductible $\varphi$. 
\end{lemm}

\begin{proof} On a 
$U_\varphi\,.\, \alpha  = \frac{\varphi(1)}{ n } \sum_{\tau\in G} \,\varphi(\tau^{\!-1}) \big( \sum_{\sigma\in G} u(\sigma)  
\alpha^{\tau\sigma^{\!-1}} \big) \!\! \equiv\!\! 0 \pmod p$, par conjugaisons par $\tau$ de $\sum_{\sigma\in G} u(\sigma)  \alpha^{\sigma^{\!-1}} \equiv 0 \pmod p$.
\end{proof}

\begin{lemm}\label{lemm16} Soient  $U \in \mathcal L$ et $\varphi$ tels que $U_\varphi \not\equiv 0 \!\pmod {\mathfrak p}$.
Alors l'endomor\-phisme $E_\varphi :=  e_\varphi  \sm_{\nu\in G}  \alpha^\nu \nu^{\!-1}\!
\in {\rm End}_{KC_\chi} \!(V_\varphi)$ est non inversible modulo ${\mathfrak p}$.
\end{lemm}

\begin{proof} Raisonnons par transposition des endomorphismes (ce qui ne change pas les d\'eterminants). On a~:
\begin{eqnarray*}
 U_{\varphi} \,.\, E_\varphi &=& e_\varphi  \sm_{\nu\in G}  U_{\varphi}\, \alpha^\nu\, \nu^{\!-1}
    = e_\varphi   \sm_{\nu\in G}  \alpha^\nu \, \sm_{\sigma\in G} u_{\varphi} (\sigma)\sigma^{\!-1}  \nu^{\!-1} \\
  &=&e_\varphi   \sm_{\tau\in G}  \Big( \sm_{\nu\in G} u_{\varphi}(\nu^{\!-1}\tau) \,\alpha^\nu\Big)\tau^{\!-1} \\
&=&  e_\varphi   \sm_{\tau\in G}  \big( U_{\varphi} \,.\,\alpha \big)^{\tau} \tau^{\!-1}\equiv 0 \pmod p,
\end{eqnarray*}
d'apr\`es le Lemme \ref{lemm15} ci-dessus.
\end{proof}

Comme $E_\varphi$ est un endomorphisme de $V_\varphi$ sur $KC_\chi$, il existe un id\'eal premier 
${\mathfrak P}\div {\mathfrak p}$ 
de $KC_\chi $ tel que ${\rm det}(E_\varphi)\equiv 0\! \pmod {\mathfrak P}$. 
Mais toute conjugaison par $\tau \in G$ donne $E_\varphi^\tau =  e_\varphi \sm_{\nu\in G}  \alpha^{\tau\nu} \, \nu^{\!-1} = 
e_\varphi \sm_{\nu\in G} \alpha^{\nu}\, \nu^{\!-1}  \, .\, (e_\varphi  \tau)=  E_\varphi\, \rond\, e_\varphi \tau$, et on obtient
${\rm det}(E_\varphi^\tau) = {\rm det}(E_\varphi)\,  {\rm det}(e_\varphi \tau) \equiv 0 \pmod {\mathfrak P^\tau}$,
d'o\`u ${\rm det}(E_\varphi)  \equiv 0 \pmod {\prod_{\tau \in G} \mathfrak P^\tau}$ puisque les  
$ {\rm det}(e_\varphi \tau)$ sont~inversibles.

\smallskip
Donc pour l'id\'eal premier ${\mathfrak p} \div p$ de $Z_\chi$ tel que $U_\varphi \not\equiv 0 \pmod {\mathfrak p}$ et
 ${\rm det}(E_\varphi)
\equiv 0 \pmod {\mathfrak p}$ (\'etendu \`a $KC_\chi$), on est conduit \`a 
$P^{\varphi} (\ldots, \alpha^\nu, \ldots) \equiv 0 \pmod {\mathfrak p}$, donc \`a 
$\Delta_p^\varphi(\eta) \equiv 0 \pmod {\mathfrak p}$.
Comme $\Delta_p^\theta(\eta)$ est la norme locale en ${\mathfrak p}$  de $\Delta_p^\varphi(\eta)$,
on obtient~:

\begin{coro}\label{coro161} Si $U_\varphi \not\equiv 0 \!\pmod {\mathfrak p}$, on a
$\Delta_p^\theta(\eta) \equiv 0 \pmod p$ pour le caract\`ere $p$-adique $\theta$ (au-dessus de $\varphi$) associ\'e \`a 
${\mathfrak p}$ (cf. D\'efinition~\ref{defi01},\,\S\ref{calp})
\end{coro}

\begin{lemm}\label{lemm161} Si l'endomorphisme $E_\varphi := e_\varphi  \sm_{\nu\in G}  \alpha^{\nu} \, \nu^{\!-1} \in
{\rm End}_{KC_\chi}(V_\varphi)$ est non inversible modulo ${\mathfrak p}$, alors il existe une $\varphi$-relation
non nulle modulo ${\mathfrak p}$, de la forme $W = \sm_{\sigma \in G} w(\sigma) \sigma^{\!-1} \in e_\varphi \, Z_{\chi,(p)}[G]$, 
telle que $W.\, \alpha \equiv  0 \pmod  {\mathfrak p}$.
\end{lemm}

\begin{proof} Le Lemme  \ref{lemm1},\,\S\ref{sub1} ramenant \`a des raisonnements $Z_{\chi,(p)}$-lin\'eaires, il existe $W \in  e_\varphi  Z_{\chi,(p)} [G]$ tel que $W \not\equiv 0 \pmod {\mathfrak p}$ est dans le 
noyau de la transpos\'ee de $E_\varphi$, et on a la congruence $ W \,.\, E_\varphi \,.\,\equiv  0 \pmod {\mathfrak P}$ 
pour ${\mathfrak P} \div {\mathfrak p}$.

\smallskip
La relation  $E_\varphi^\tau = E_\varphi\,\rond\,  e_\varphi \tau$ et le fait que  $W$ est \`a coefficients dans $Z_{\chi,(p)}$
 montrent, par conjugaisons, que la congruence a lieu modulo  ${\mathfrak p}$ (\'etendu).

\smallskip
Posons $W = \sm_{\sigma \in G} w (\sigma) \sigma^{\!-1}$, $w (\sigma) \in Z_{\chi,(p)}$ pour tout $\sigma \in G$~;
la congruence $W \, .\, E_\varphi  \equiv  0 \pmod {\mathfrak p}$ s'\'ecrit (puisque $ e_\varphi W = W$)~:
$$\sm_{\nu \in G} \sm_{\sigma \in G}  w (\sigma) \alpha^\nu \sigma^{\!-1} \,\nu^{\!-1} \equiv 
 \sm_{\sigma \in G}  w (\sigma)  \sm_{t \in G} \alpha^{t^{\!-1} \sigma^{\!-1}}  \,t \equiv  0 \pmod {\mathfrak p}, $$
d'o\`u~:
$$\sm_{t \in G}\big(\sm_{\sigma \in G} w(\sigma) \alpha^{t^{\!-1}\sigma^{\!-1}} \big)\,t \equiv 0 \pmod {\mathfrak p}~; $$

il en r\'esulte $\sm_{\sigma \in G}  w(\sigma) \alpha^{t^{\!-1}\sigma^{\!-1}} \!\!\equiv  0\pmod {\mathfrak p}$, pour tout $t\in G$, 
puis $\sm_{\sigma \in G}  w(\sigma) \alpha^{\sigma^{\!-1}} \!\!\equiv  0 \! \pmod {\mathfrak p}$, 
ce qui donne la $\varphi$-relation associ\'ee non triviale modulo ${\mathfrak p}$~:
$$W = \sm_{\sigma \in G}  w(\sigma)\,\sigma^{\!-1} \in e_\varphi Z_{\chi,(p)}[G]$$

(ceci n'entra\^inant pas $W \in e_\varphi \Z_{(p)}[G]$), telle que $W.\, \alpha \equiv  0 \pmod  {\mathfrak p}$).
\end{proof}

\subsubsection{Un cas particulier}\label{sub150}  Soit $\eta \in K^\times$ et soit $\alpha = \alpha_p(\eta)$.
Dans le $G$-module~${\mathcal L}$,  consid\'erons $U = \sm_{\nu \in G} u(\nu) \nu^{\!-1}$. 
On dispose par conjugaison des $n$ congruences~:
$$\sm_{\nu \in G} u(\nu)  \,\alpha^{\tau \nu^{\!-1}} = \sm_{\nu \in G} u(\nu \tau)
 \,\alpha^{\nu^{\!-1}}\equiv 0 \pmod p, \ \  \forall \tau \in G, $$

mais les relations $\sm_{\nu \in G} u(\nu \tau) \,\nu^{\!-1}$ (vues dans $\F_p[G]$) ne sont pas
n\'ecessairement $\F_p$-ind\'ependantes.

\begin{lemm}\label{lemm20} Le syst\`eme des  $n$ relations 
$\big(\sm_{\nu \in G} u(\nu \tau) \,\nu^{\!-1}\big)_{\tau \in G}$, $u(\sigma) \in \F_p$ pour tout 
$\sigma \in G$,  est de $\F_p$-rang 1 si et seulement si $u : G \too \F_p^\times$ est un homomorphisme de groupes.
\end{lemm}

\begin{proof}
Si  pour tout $\tau \in G$, il existe $\lambda(\tau)\in \F_p$ tel que $u(\nu \tau) = \lambda(\tau) u(\nu) $ pour tout $\nu \in G$, 
on a (en se ramenant par exemple \`a $u(1)=1$ puisque $U \ne 0$), 
$u(\tau) = \lambda(\tau)$ pour tout $\tau$, d'o\`u $u(\nu \tau) =u(\nu)  u(\tau)$ pour tout $\tau, \nu \in G$~; 
donc $u : G \too \F_p^\times$ est un homomorphisme de groupes. 
\end{proof}

Par cons\'equent, l'image de $u$ est un groupe cyclique d'ordre $d$ diviseur de $p-1$ isomorphe \`a $G/H$ o\`u $H$ 
(le noyau de $u$) fixe un sous-corps $K'$ de $K$ cyclique sur $\Q$ de degr\'e $d$.
 Si l'on pose $\eta' = \No_{K/K'} (\eta)$,
$r := u(s)$ (d'ordre $d$ modulo $p$), o\`u  $s$ est g\'en\'erateur dans ${\rm Gal}(K'/\Q)$, 
on est ramen\'e \`a une relation de la forme $\sm_{i=1}^d r^{i} \,\alpha'{}^{s^{-i}} 
\equiv 0 \pmod p$ avec $\alpha' = {\rm Tr}_{K/K'} (\alpha)$. 

Ceci traduit la nullit\'e modulo $p$ de
$\Delta_p^{\theta'}(\eta')$ pour le caract\`ere $\theta'$ de $K'$ d\'efini par $\theta'(s) \equiv r \pmod p$.

\begin{coro}\label{coro21} Si le $G$-module  ${\mathcal L}$ des 
 relations $\sm_{\nu \in G} u(\nu) \,\nu^{\!-1} \in \Z_{(p)} [G]$ relatives \`a $\alpha$ est de $\F_p$-rang~1, alors
il existe un unique caract\`ere $p$-adique $\theta$ de degr\'e 1 de~$G$, dont le noyau fixe un sous-corps $K'$ cyclique de $K$ 
de degr\'e $d$, et qui est tel que $\Delta_p^{\theta'}(\eta')\equiv \sm_{i=1}^d r^{i} \alpha'{}^{s^{-i}} \equiv 0 \pmod p$
o\`u $\theta'$ est le caract\`ere fid\`ele  issu de $\theta$.
\end{coro}

\begin{rema}\label{rema21} Comme $d\div p-1$, $p$ est totalement d\'ecompos\'e dans le corps $C_\chi = \Q(\mu_d)$. 
Donc $\theta$ est issu d'un caract\`ere absolument irr\'eductible $\varphi$ unique.
Posons $\varphi (s) = \zeta$~; alors l'id\'eal ${\mathfrak p} = (p, \zeta- r)$ est un id\'eal premier au-dessus de $p$ dans $C_\chi$,
ce qui identifie le couple $(\theta, {\mathfrak p})$ (cf. D\'efinition \ref{defi01} (ii),\,\S\ref{calp}).
\end{rema}

\subsubsection{R\'esultats principaux}\label{sub15}
Les r\'esultats techniques des\,\S\S\ref{sub155}, \ref{sub150}  conduisent \`a l'\'enonc\'e sui\-vant (o\`u $p$ est suppos\'e assez grand)~:

\begin{theo}\label{theo24} Soit $K/\Q$ une extension Galoisienne de groupe de Galois~$G$.
Soit $\eta \in K^\times$ tel que le $\Z[G]$-module engendr\'e par $\eta$
soit de $\Z$-rang $n=\vert G \vert$. 

\smallskip
Pour tout $p$ on pose $\eta_1 := \eta^{p^{n_p} - 1} = 
1 + p\,\alpha_p(\eta)$, $\alpha_p(\eta) \in Z_{K,(p)}$, o\`u  $n_p$ est le degr\'e r\'esiduel de $p$ dans $K/\Q$. 

\smallskip
Soit ${\mathcal L}$ le $G$-module des relations $U \in \Z_{(p)} [G]$ relatif \`a $\alpha_p(\eta)$, i.e.,  telles que~:

\smallskip
\centerline{$\sm_{\nu \in G} u(\nu)\, \alpha_p(\eta)^{\nu^{\!-1}} \equiv 0 \pmod p, \ \,u(\nu) \in \Z_{(p)} \ 
\hbox{(cf. D\'efinition \ref{defi11},\,\S\ref{sub155})}. $}

\smallskip
On a alors les r\'esultats suivants~:

\smallskip
(i) Il existe une relation $U \not\equiv 0 \pmod {p\,\Z_{(p)} [G]$} (i.e., ${\mathcal L}$ est non trivial) si et seule\-ment si le r\'egulateur normalis\'e ${\rm Reg}_p^G(\eta)$ est divisible par $p$ (cf. Remarque~\ref{rema110},\,\S\ref{gene1}).

\smallskip
(ii) Soit $\theta$ un caract\`ere $p$-adique irr\'eductible de $G$ et soit $f$ son degr\'e r\'esiduel (i.e., celui de $p$ dans le corps des valeurs des caract\`eres absolument irr\'eductibles~$\varphi \div \theta$). 

\smallskip
Alors, vu dans $\F_p[G]$, le $G$-module ${\mathcal L}^\theta := e_\theta {\mathcal L}$ est de $\F_p$-dimension non nulle si et seulement si le $\theta$-r\'egulateur local $\Delta_p^\theta(\eta)$ (cf.\,\S\ref{sub6}) est divisible par $p$. Dans ce cas, la 
$\F_p$-dimension de ${\mathcal L}^\theta$ est $\delta f \varphi(1)$, $1 \leq \delta \leq \varphi(1)$.\,\footnote{\,On verra que l'on doit attribuer \`a cette \'eventualit\'e la probabilit\'e $\frac{1}{p^{f\delta^2}}$ (cf.\,\S\ref{HP}). }

\smallskip
(iii) Si ${\rm dim}_{\F_p} ({\mathcal L} )= 1$, alors $\theta$ est  un caract\`ere de degr\'e 1 relatif \`a un 
sous-corps cyclique $K'$ de $K$, et ${\mathcal L}$ est engendr\'e par la $\theta$-relation $\sum_{i=0}^{d-1} \, r^{i}\, \alpha_p(\eta')^{s^{-i}} \equiv 0 \pmod p$, o\`u $\eta' := \No_{K/K'}(\eta)$, $d=[K' : \Q]\, \big \vert \, p-1$,
o\`u $r \in \Z$ et $s$ g\'en\'erateur de ${\rm Gal}(K'/\Q)$ sont tels que 
$\theta(s) \equiv r \pmod p$ (cf. Lemme \ref{lemm20}, Corollaire~\ref{coro21} et Remarque \ref{rema21},\,\S\ref{sub150}).
\end{theo}

\begin{proof} Il reste donc \`a prouver le point (ii) de l'\'enonc\'e.

\smallskip
 (a)  Si ${\mathcal L}^\theta \ne \{0\}$, il existe $U=\sum_{\nu\in G} u(\nu) \,\nu^{\!-1} \in \mathcal L$
 tel que $U_\theta  \not\equiv 0 \pmod {\mathfrak p}$~; donc $U_\varphi  \not\equiv 0 \pmod {\mathfrak p}$ pour tout 
 $\varphi \div \theta$. 
D'apr\`es le Lemme  \ref{lemm16} et le Corollaire \ref{coro161},\,\S\ref{sub155}, on a  $\Delta_p^\theta(\eta) \equiv 0 \pmod p$.

\smallskip
(b)  Supposons que $\Delta_p^\theta(\eta) \equiv 0 \pmod p$~;
par la nullit\'e modulo $p$ de ${\rm Det}^G(\alpha)$ qui en r\'esulte,
 il existe une relation de $\F_p$-d\'ependance de la forme $\sum_{\nu \in G} u(\nu)\, \alpha^{\nu^{\!-1}} \equiv 0 \pmod p$, 
$u(\nu) \in \Z_{(p)}$ non tous divisibles par~$p$,  et on a $U = \sum_{\nu \in G} u(\nu)\, \nu^{\!-1}  \in  {\mathcal L}$
(cf. Corollaire \ref{coro14},\,\S\ref{sub8}), mais il convient d'en d\'eduire ${\mathcal L}^\theta \ne \{0\}$.

\smallskip
Le cas Ab\'elien (cf.\,\S\ref{sub27}) sera r\'esolu de fa\c con plus pr\'ecise.
D'apr\`es le Lemme~\ref{lemm161}, \S\ref{sub155}, il existe un caract\`ere irr\'eductible $\varphi \div \theta$, et une 
$\varphi$-relation non triviale modu\-lo ${\mathfrak p}$ de la forme $W := \sum_{\nu \in G}  w(\nu)\nu^{\!-1}$, $w(\nu) \in Z_{\chi,(p)}$, telle que $W.\,\alpha \equiv 0\!\! \pmod {\mathfrak p}$.  

\smallskip
Soit $L$ le corps de d\'ecomposition de $p$ dans $C_\chi/\Q\,$ et soit $D = {\rm Gal}(C_\chi/L)$.
 Si $\{z, \ldots, z^{f}\}$ est une $Z_{L,(p)}$-base de $Z_{\chi,(p)}$, on a $w(\nu) = \sum_{i=1}^{f} a_i (\nu) z^i$
 pour tout $\nu \in G$, $a_i(\nu) \in Z_{L,(p)}$, d'o\`u $\sum_{\nu \in G}  \sum_{i=1}^{f} a_i(\nu) z^i\, \alpha^{\nu^{\!-1}} \equiv 0 \pmod  {\mathfrak p}$~; donc par identification sur la base des $z^i$ on obtient le syst\`eme de relations~:

\medskip
\centerline {$\sm_{\nu \in G} a_i(\nu) \alpha^{\nu^{\!-1}} \equiv 0 \pmod  {\mathfrak p}, \ i = 1, \ldots, f. $}

\smallskip
Pour tout $i= 1, \ldots, f $, et tout $\nu \in G$, 
il existe des $r_{\mathfrak p}^i (\nu) \in \Z$ tels que $ a_i(\nu)  \equiv r_{\mathfrak p}^i (\nu) \pmod  {\mathfrak p}$, d'o\`u
$\sum_{\nu \in G} a_i(\nu) \alpha^{\nu^{\!-1}}\equiv  \sum_{\nu \in G}r_{\mathfrak p}^i (\nu) \alpha^{\nu^{\!-1}}\equiv  0 \pmod {\mathfrak p}$,
et comme $\sum_{\nu \in G}r_{\mathfrak p}^i (\nu) \alpha^{\nu^{\!-1}} \in K$, il vient 
$\sum_{\nu \in G}r_{\mathfrak p}^i (\nu) \alpha^{\nu^{\!-1}}\equiv  0 \pmod p$.

\smallskip
Puisque $W$ est non triviale modulo ${\mathfrak p}$, les $r_{\mathfrak p}^i (\nu)$ ne sont pas tous nuls modulo $p$ et il existe 
une relation non triviale $\sum_{\nu \in G}r_{\mathfrak p}^i (\nu)\,\alpha^{\nu^{\!-1}}$, $i \in \{1,\ldots,f\}$.
Comme $W$ est une $\varphi$-relation, ceci se transmet \`a $\sum_{\nu \in G} a_i(\nu)\,\alpha^{\nu^{\!-1}}$ et par cons\'equent,
$\sum_{\nu \in G} r_{\mathfrak p}^i (\nu) \, {\nu^{\!-1}}$ est une $\theta$-relation non triviale de $\mathcal L$.
De fait, la matrice $\big ( r_{\mathfrak p}^i (\nu) \big)_{i,\nu}$ est de rang $f$.
\end{proof}

\begin{rema}\label{remadim}
La $\F_p$-repr\'esentation irr\'eductible $V_\theta$ de caract\`ere $\theta$ est de 
dimension $f \varphi(1)$ et par extension des scalaires, on a $V_\theta \simeq f V_\varphi$~; on a ${\mathcal L}^\theta \simeq \delta  V_\theta$, $0 \leq \delta \leq \varphi(1)$, et ${\mathcal L}^\theta$ est de $\F_p$-dimension $\delta  f \varphi(1)$.

\smallskip
Il est clair que $\delta =\varphi(1)$ \'equivaut \`a ${\mathcal L}^\theta= e_\theta \,\F_p[G]$ qui correspond \`a la relation
$e_\theta \,\alpha \equiv 0 \pmod p$, autrement dit \`a un maximum de $\theta$-relations.
\end{rema}

\begin{coro}\label{coro25} (i) Lorsque ${\mathcal L}^\theta \ne \{0\}$, on obtient des rel\`evements de la forme 
$\eta^{U_\theta} \in \prod_{v \div p} K_v^{\times p}$ pour toute $\theta$-relation $U_\theta \in {\mathcal L}^\theta$.

(ii) Sous la condition (iii) du th\'eor\`eme (${\mathcal L} = {\mathcal L}^\theta$ de $\F_p$-dimension 1) et lorsque $\theta$ est un carac\-t\`ere fid\`ele, on a $\eta^{U_\theta} \in \prod_{v \div p} K_v^{\times p}$, o\`u $U_\theta = \sum_{k=0}^{n-1} r^{k}\,s^{-k}$ et  o\`u $n = \vert G\vert$ divise $p-1$ (cas totalement d\'ecompos\'e).
\end{coro}

\begin{proof}
Soit $\alpha := \alpha_p(\eta)$.
On a $\eta_1^{U_\theta} \!= (1+p\,\alpha)^{U_\theta} \equiv 1 + p\, U_\theta \cdot \alpha\! \pmod {p^2}$, et
comme $U_\theta\, \cdot\,\alpha \equiv  0 \pmod p$ par d\'efinition, il vient
$\eta_1^{U_\theta} = 1 +p^2 \beta$, $\beta \in Z_{K,(p)}$.
Donc $\eta_1^{U_\theta} = (1+p\,\gamma)^p$, $\gamma \in \prod_{v \div p} K_v$, et
$\eta = \eta^{p^{n_p}} \eta_1^{-1}$ implique $\eta^{U_\theta} \in \prod_{v \div p} K_v^{\times p}$.
Le cas (ii) est imm\'ediat.
\end{proof}

\begin{rema}\label{rema26} D'apr\`es la Remarque \ref{rema11}\,(ii),\,\S\ref{gene1}, le fait de remplacer $\eta$ par 
$\eta' = \No_{K/K'}(\eta)$ ne modifie pas la question du nombre (fini ou non)
de premiers $p$ tels que $\Delta^G_p(\eta) \equiv 0 \pmod p$, car le nombre de sous-corps $K'$ de $K$ est fini
auquel cas si $\Delta^G_p(\eta) \equiv 0 \pmod p$ pour une infinit\'e de $p$, il existe un caract\`ere rationnel fid\`ele $\chi'$
pour lequel $\Delta^{\chi'}_p(\eta') \equiv 0 \pmod p$ une infinit\'e de fois.
\end{rema}

\subsection{Cons\'equences pratiques}\label{sub10}
Donnons  quelques exemples concrets.
Il r\'esulte de ce qui pr\'ec\`ede, lorsque $\Delta^{\theta}_p(\eta) \equiv 0 \pmod p$, que l'on peut trouver indif\'eremment une $\theta$-relation 
$\sum_{\sigma \in G} u(\sigma)\,\sigma^{\!-1}$ \`a coefficients $p$-entiers rationnels,
ou, pour $\varphi \div \theta$, une $\varphi$-relation $\sum_{\sigma \in G} u_\varphi(\sigma)\,\sigma^{\!-1}$
 \`a coef\-ficients dans $Z_{\chi,(p)}$.

\begin{example}\label{ex19}  {\rm  Pour un corps cubique cyclique $K$, on a $e_\chi = \frac{1}{3}(2-\sigma-\sigma^2)$
pour $\chi \ne 1$. On a $C_\chi = \Q(j)$.

\smallskip
(i) Si $p$ est inerte dans $C_\chi/\Q$,  $\chi = \theta = \varphi+\varphi^s$, $f=2$, et on montrera (cf.\,\S\ref{sub19}) 
que l'on obtient les $\theta$-relations ind\'ependantes $\alpha-\alpha^\sigma \equiv 0 \pmod p$ et 
$\alpha^\sigma-\alpha^{\sigma^2} \equiv 0 \pmod p$, conjugu\'ees ($f=2$). 
Donc $U= 1-\sigma$ est telle que (avec $\varphi(\sigma) = j$)~:
\begin{eqnarray*}
 U_\varphi := e_\varphi \,U&\!\!=\!\!& \Frac{1}{3}(1+ j^{-1}\sigma +  j^{-2}\sigma^2)(1-\sigma) 
= \Frac{1}{3}(1-j + (j^2-1)\, \sigma + (j-j^2)\,\sigma^2 ) \\
&\!\!=\!\!&  \Frac{1}{3}(1-j)\,(1+ j^{-1}\sigma +  j^{-2}\sigma^2),
\end{eqnarray*}
qui donne, \`a une unit\'e $p$-adique pr\`es, $U_\varphi = 1+ j^{-1}\sigma +  j^{-2}\sigma^2$
qui engendre la $\varphi$-composante ${\mathcal L}^\varphi$  associ\'ee \`a la relation 
$\Delta_p^\varphi (\eta) = \alpha+ j^{-1}\alpha^\sigma +  j^{-2}\alpha^{\sigma^2}\equiv 0 \pmod p$
et on a $\Delta_p^\theta (\eta) = \No_{C_\chi/\Q} (\Delta_p^\varphi (\eta) ) \equiv 0 \pmod p$.
Par identification sur la base $\{1, j\}$, on retrouve les $\theta$-relations \`a coefficients rationnels 
$\alpha-\alpha^\sigma \equiv 0 \pmod p$ et $\alpha^\sigma - \alpha^{\sigma^2} \equiv 0 \pmod p$.

\smallskip
(ii) Dans le cas $p$ d\'ecompos\'e, il y a deux caract\`eres $p$-adiques, dont $\theta = \varphi$, et
on obtiendra une unique $\theta$-relation de la forme $\Delta_p^\theta (\eta) =\alpha+r^{-1}\alpha^\sigma  +r^{-2}\alpha^{\sigma^2}\equiv 0 \pmod p$ dont les conjugu\'ees lui sont proportionnelles car $r^3 \equiv 1 \pmod p$. 

\smallskip
D'o\`u $U = 1+r^{-1}\sigma  +r^{-2}{\sigma^2}$, ce qui conduit (avec
$j \equiv r \pmod {\mathfrak p}$ pour l'id\'eal premier $ {\mathfrak p} \div p$ associ\'e \`a $\theta = \varphi$) \`a~:
\begin{eqnarray*} 
U_\varphi = e_\varphi \,U & \!\!=\!\!  & \Frac{1}{3}(1+ j^{-1}\sigma +  j^{-2}\sigma^2)(1+r^{-1}\sigma  +r^{-2}{\sigma^2}) \\
&=& \Frac{1}{3} (1+ r^{-1}j^{-2} +r^{-2}j^{-1} )\,(1+ j^{-1}\sigma +  j^{-2}\sigma^2) \\
\!\!&\equiv& \!\!  1+ j^{-1}\sigma +  j^{-2}\sigma^2 \pmod {\mathfrak p},
\end{eqnarray*}
soit $U_\varphi = 1+ j^{-1}\sigma +  j^{-2}\sigma^2$ engendrant ${\mathcal L}^\varphi$. 
On retrouve la $\theta$-relation de d\'epart
\`a partir de $\alpha+ j^{-1}\alpha^\sigma +  j^{-2}\alpha^{\sigma^2} \equiv 0 \pmod {\mathfrak p}$ (\'etendu \`a $KC_\chi$), 
 qui donne ici $\alpha+ r^{-1}\alpha^\sigma +  r^{-2}\alpha^{\sigma^2} \equiv 0 \pmod {\mathfrak p}$ donc 
$\equiv 0\pmod p$ (premier membre dans $K$). 

\smallskip
Le conjug\'e $\theta^s$ de $\theta$ correspondrait \`a la congruence $j \equiv r \pmod { {\mathfrak p}^s}$
(\'equivalente \`a $j \equiv r^2 \pmod  {\mathfrak p}$), mais on n'a pas n\'ecessairement et simultan\'ement  
la relation $\Delta_p^{\theta^s} (\eta) =\alpha+ r^{-2}\alpha^\sigma +  r^{-1}\alpha^{\sigma^2} \equiv 0 \pmod p$, 
ce qui conduirait \`a
$\alpha \equiv \alpha^\sigma \equiv \alpha^{\sigma^2} \pmod p$ et la nullit\'e modulo $p$ de 
$\Delta^{\theta}_p(\eta)$ et de $\Delta^{\theta^s}_p(\eta)$ (probabilit\'e en $\frac{O(1)}{p^2}$).}
\end{example}

\begin{example}\label{ex18} {\rm Prenons un exemple num\'erique (cf.\,\S\ref{p=61}) avec le 
corps $K$ d\'efini par le polyn\^ome $Q=x^6+9x^4-4x^3+27x^2+36x+31$ (groupe $D_6$),
pour $\chi =\theta= \chi_2$ de degr\'e 2. Soit $\eta = x^5-2x^4+ 4x^3 -3x^2+x-1$.

\smallskip
Pour $p=61$,  on obtient le $G$-module ${\mathcal L}$ \`a partir des trois relations 
$\F_p$-lin\'eaires ind\'ependantes (par programme)~:
\begin{eqnarray*}
 19\alpha+56\alpha^\sigma+46\alpha^{\sigma^2}+\alpha^\tau  & \equiv& 0 \pmod p\\
46\alpha+19\alpha^\sigma+56\alpha^{\sigma^2}+\alpha^{\tau \sigma^2} &\equiv& 0 \pmod p\\
56\alpha+46\alpha^\sigma+19\alpha^{\sigma^2}+\alpha^{\tau \sigma} &\equiv& 0 \pmod p\,.
\end{eqnarray*}

L'idempotent ``trace'' $e_{1}$ donne la relation triviale, donc le $p$-quotient de Fermat 
$\Delta_p^{1} (\eta)$ est non nul modulo $p$.
On obtient la ${\chi_1}$-relation correspondante \`a l'idempotent $e_{\chi_1}$ ($\chi_1$ non trivial de degr\'e~1) en faisant
la somme des trois relations, ce qui donne
$\alpha+\alpha^\sigma+\alpha^{\sigma^2}-\alpha^\tau-\alpha^{\tau\sigma}-\alpha^{\tau\sigma^2}\equiv 0 \pmod p$
(d'o\`u pour $\theta=\chi_1$ la nullit\'e du $\theta$-r\'egulateur $\Delta_p^{\theta} (\eta)$). De fait, il s'agit d'une nullit\'e 
triviale car le programme trouve que tout $p$ est solution pour $\Delta_p^{\theta} (\eta)\equiv 0 \pmod p$
 (les conjugu\'es de $\eta$ v\'erifient $\eta^{1+\sigma+\sigma^2-\tau-\tau\sigma-\tau\sigma^2}=1$).
Les choix de $\eta$ \'etant faits au hasard, ce fait est une pure co\"incidence !

\smallskip
Pour $\theta=\chi_2$ (de degr\'e 2), le $\theta$-r\'egulateur 
$\Delta_p^{\theta} (\eta)$ est nul modulo $p$ (non trivialement) et cela cor\-respond \`a
la $\theta$-relation suivante en appliquant $e_{\theta} = \frac{1}{3}(2-\sigma-\sigma^2)$~:
$$-\alpha + 36\alpha^\sigma+ 26 \alpha^{\sigma^2}+21\alpha^\tau + 20 \alpha^{\tau\sigma} + 20\alpha^{\tau\sigma^2}
\equiv 0 \pmod p. $$

Notons que par conjugaison, cette derni\`ere relation engendre un $\F_p$-espace de dimension 2 
(autrement dit ${\mathcal L}^\theta \simeq V_\theta$).
En effet,  pour la $\theta$-relation correspondante~:

\smallskip
\centerline{$U=-1 + 36\sigma+ 26{\sigma^2}+21\tau + 20{\tau\sigma} + 20{\tau\sigma^2}$, }

\smallskip
on a par d\'efinition $\sigma^2 U= -U - \sigma U$ et on trouve les relations suivantes~:

\smallskip
\centerline{$\tau U = 24U+51\sigma U$,\ \ \  $\tau\sigma  U = 27U+37\sigma U$, \ \ \ $\tau\sigma ^2 U = 10U+34\sigma U$.} }
\end{example}

Par la suite, nous supposerons que le $\Z[G]$-module $F$ engendr\'e par $\eta$ est de $\Z$-rang $n = \vert G\vert$
(afin qu'il n'y ait pas de nullit\'es triviales).

\smallskip
Ensuite, nous n\'egli\-gerons les premiers $p$ pour lesquels au moins deux $\theta$-r\'egulateurs sont divisible par $p$, une telle probabilit\'e \'etant au plus en $\frac{O(1)}{p^2}$, vu l'ind\'ependance des $\theta$-r\'egulateurs associ\'es \`a plusieurs caract\`eres $p$-adiques (cf.~\S\ref{sub12} pour des statistiques num\'eriques).
Il restera alors le cas $\Delta^\theta_p(\eta) \equiv 0 \pmod p$ pour un unique caract\`ere $p$-adiquel $\theta$ de $G$ 
sous r\'eserve d'avoir $f=1$ et la repr\'esentation ${\mathcal L}^\theta$ minimale (tout cas contraire donnant une probabilit\'e au plus en $\frac{O(1)}{p^2}$).

\smallskip
 L'obstruction essentielle viendrait alors des $p$ satisfaisant \`a la d\'efinition suivante~:

\begin{defi}\label{defidec}  Un nombre premier $p$ constitue un cas de {\it $p$-divisibilit\'e mini\-male $p^{\varphi(1)}$}
(pour le r\'egulateur normalis\'e ${\rm Reg}_p^G (\eta)$) s'il existe un unique $\theta$ tel que~: 
$\Delta^\theta_p(\eta) \equiv 0 \pmod p$, $p$ est totalement d\'ecompos\'es dans le corps des valeurs de 
$\varphi \div \theta$ (i.e., $f =1$),  ${\mathcal L}^\theta \simeq \delta  V_\theta$ est de $\F_p$-dimension minimale (i.e., $\delta =1$),
et ${\rm Reg}_p^\theta (\eta) \sim p$~; dans ce cas, ${\rm Reg}_p^G (\eta) \sim p^{\varphi(1)}$ 
(cf. Remarque \ref{rema110}, \S\,\ref{gene1}).

\smallskip
Si $G$ est Ab\'elien, il s'agit des  $p \equiv 1 \pmod n$, o\`u $n$ est l'ordre de $\varphi$.
\end{defi}

\vspace{-0.2cm}
\section{Exp\'erimentations et consid\'erations heuristiques}

\subsection{M\'ethodes probabilistes en th\'eorie des nombres}\label{sub11} 
Si des \'ev\'ene\-ments $E_p$, index\'es par les nombres premiers, sont ind\'ependants et de probabilit\'es ${\rm Pr} (E_p)$, on peut appliquer le principe heuristique de Borel--Cantelli qui consiste \`a dire que si la s\'erie $\sum_p {\rm Pr} (E_p)$ converge, alors
la conjecture naturelle est que les \'ev\'enements $E_p$ sont r\'ealis\'es un nombre fini de fois et que si elle diverge 
ils sont r\'ealis\'es une infinit\'e de fois avec une densit\'e en rapport (cf. \cite{T}, Chap.\,III.1). 

\smallskip
Dans notre cas, $E_p$ est pour $\eta  \in K^\times$ fix\'e l'\'ev\'enement ``$\Delta_p^G(\eta) \equiv 0 \pmod p$'' ou  ``$\Delta_p^\chi(\eta) \equiv 0 \pmod p$'' ($\chi$ ration\-nel) ou encore  ``$\Delta_p^\theta (\eta) \equiv 0 \pmod p$'' qui n'est d\'efini que si l'on choisit un $\theta \div \chi$ pour chaque~$p$ (cf.\,\S\ref{sub6}).

\smallskip
Tout se joue alors sur la valeur ${\rm Pr} (E_p) = \frac{1}{p^{1+\epsilon(p,\eta)}}$, $\epsilon(p,\eta)> 0$ ou $\epsilon(p,\eta)= 0$, pour les cas de $p$-divisibilit\'e minimale $p^{\varphi(1)}$ (D\'efinition ci-dessus).

\smallskip
En l'absence d'arguments th\'eoriques pr\'ecis, il est d'usage de supposer que le cas $\Delta_p^\theta(\eta) \equiv 0 \pmod p$
se produit avec une probabilit\'e en $\frac{O(1)}{p}$.\,\footnote{Nous utilisons par commodit\'e un num\'erateur $O(1)$,  toujours effectif et souvent \'egal \`a~1.}  Il en r\'esulte alors que le nombre de solutions $p<x\,$ \`a $\,\Delta_p^G(\eta) = \prod_\theta \Delta_p^\theta (\eta)^{\varphi(1)} \equiv 0 \pmod p$ serait de l'ordre de $O(1)\sum_{p<x} \frac{1}{p}$, donc de 
$O(1){\rm log}{\rm log}(x)$.
 Voir la Remarque \ref{rema110},\,\S\ref{gene1} pour l'interpr\'etation en termes de $p$-divisibilit\'es des r\'egulateurs correspondants.

\smallskip
Or il est \'evident, puisque $\Delta_p^\theta(\eta)$ est une norme locale dans l'extension $C_\chi/\Q$,  qu'un tel r\'egulateur local est soit \'etranger \`a $p$ soit divisible par $p^f$ o\`u $f$ est le degr\'e r\'esiduel de $p$ dans cette extension~; de m\^eme, si le caract\`ere irr\'eductible $\varphi \div \theta$ est de degr\'e $\varphi (1) \geq 2$, la divisibilit\'e par $p^{f\varphi (1)}$ intervient puisque $\Delta_p^\theta(\eta) = \prod_{\varphi \div \theta} \Delta_p^\varphi(\eta)$ est \'elev\'e \`a la puissance $\varphi (1)$.
Cependant il n'est pas clair pour autant que la ``probabilit\'e'' d'\^etre divisible par~$p$ (ou  la ``densit\'e des cas observ\'es'') 
soit de fa\c con pr\'ecise fonction de $f\,\varphi (1)$. 
On verra que le degr\'e $\varphi (1)$ n'intervient pas directement mais que, par contre, le nombre $\delta $ tel que ${\mathcal L}^\theta \simeq \delta  V_\theta$ intervient, ainsi que $f$, sous la forme $\frac{O(1)}{p^{f \delta^2}}$ qui est la probabilit\'e d'avoir\ 
``$\ \Delta_p^\theta (\eta) \equiv 0 \pmod p\ $ \& $\ {\mathcal L}^\theta \simeq \delta  V_\theta\ $'' (cf.\,\S\ref{HP}). 
On admettra donc que tous les cas de probabilit\'e au plus en $\frac{O(1)}{p^2}$ sont en nombre fini et peuvent \^etre \'ecart\'es
pour $p$ assez grand.
En outre il faut s'assurer que les $\Delta_p^\theta(\eta)$ modulo $p$ sont ind\'ependants.

\subsection{Ind\'ependance probabiliste (sur $\theta$) des variables $\Delta^\theta_p(\eta)$}\label{sub12}
Nous traitons d'abord le cas du groupe $D_6$, par utilisation de la fonction {\it random}, pour d\'eterminer deux aspect~: 

\smallskip
(i) l'ind\'ependance des $\theta$-r\'egulateurs (probabilit\'e au plus en $\frac{O(1)}{p^2}$ d'avoir deux $\theta$-r\'egula\-teurs $\Delta^\theta_p(\eta)$ et $\Delta^{\theta'}_p(\eta)$ nuls modulo $p$, pour $\theta \ne \theta'$)~;

\smallskip
(ii) la probabilit\'e en $\frac{O(1)}{p}$ d'avoir la nullit\'e modulo $p$ de $\Delta_p^{\theta}(\eta)$ pour le caract\`ere $\theta = \chi_2$ de degr\'e 2, le cas des caract\`eres de degr\'e 1 \'etant analogue.

\smallskip
On consid\`ere le corps $K$ (compos\'e de $\Q(\sqrt[3] 2)$ et de $\Q(j)$) d\'efini par le polyn\^ome $$Q=x^6+9x^4-4x^3+27x^2+36x+31 .$$
On prend au hasard $\eta$ modulo $p^2$, \'etranger \`a $p$, ce qui donne des $\alpha$ r\'epartis
al\'eatoirement modulo $p$. On peut choisir $p$ \`a volont\'e (ici $p=37$).

\smallskip
Le programme calcule les conjugu\'es de $\alpha = \alpha_p(\eta)$ sur la base $\{x^5, x^4, x^3, x^2, x, 1\}$.
La variable $N_0$ compte le nombre de $\eta$  \'etrangers \`a $p$.
Ensuite, chacun des trois r\'egulateurs est calcul\'e, et dans les variables $N_1, N_2, N_3,   N_{12}, N_{13} , N_{23} , N_{123}$, 
on donne le nombre de cas de nullit\'es simultan\'ees de 1, 2 ou 3 r\'egulateurs.

\smallskip
La proc\'edure {\it nfgaloisconj} pour $D_6$ donne les automorphismes dans l'ordre~:
$$1, \tau, \sigma\tau = \tau \sigma^2,\sigma^2 \tau = \tau\sigma, \sigma, \sigma^2,$$
que nous avons permut\'e pour avoir (de $e_1$ \`a~$e_6$) la liste
$\ 1, \sigma, \sigma^2, \tau, \tau\sigma, \tau \sigma^2. $
\footnotesize

\medskip
$\{$$Q=x^6+9*x^4-4*x^3+27*x^2+36*x+31; p=37; p2=p^2;$\par
$e1= x ;  $\par
 $e2= 11/180*x^5 + 1/180*x^4 + 11/18*x^3 - 1/45*x^2 + 403/180*x + 419/180 ; $\par
 $e3= 13/180*x^5 - 7/180*x^4 + 13/18*x^3 - 38/45*x^2 + 509/180*x + 127/180 ; $\par
 $e4= -4/45*x^5 + 1/45*x^4 - 8/9*x^3 + 26/45*x^2 - 137/45*x - 91/45 ;  $\par
 $e5= -1/60*x^5 - 1/60*x^4 - 1/6*x^3 - 4/15*x^2 - 73/60*x - 79/60 ; $\par
 $e6= -1/36*x^5 + 1/36*x^4 - 5/18*x^3 + 5/9*x^2 - 65/36*x + 11/36 ; $\par
 $N0=0; N1=0; N2=0; N3=0; N12=0; N13=0; N23=0; N123=0;$\par
 $ for(i=1, 1000000,  $\par
 $a=random(p2); b=random(p2);  c=random(p2); $\par
 $aa=random(p2); bb=random(p2); cc=random(p2); $\par
 $Eta= Mod(a*e1^5+b*e1^4+c*e1^3+aa*e1^2+bb*e1+cc,Q); $\par
 $N=norm(Eta); if (Mod(11*N,p)!=0, N0=N0+1; $\par
 $u=Mod(a,p2); v=Mod(b,p2); w=Mod(c,p2);  $\par
 $uu=Mod(aa,p2); vv=Mod(bb,p2); ww=Mod(cc,p2);  $\par
 $P=x^6+9*x^4-4*x^3+27*x^2+36*x+Mod(31,p); $\par
 $P2=x^6+9*x^4-4*x^3+27*x^2+36*x+Mod(31,p2);$\par
 $y=Mod(u*x^5+v*x^4+w*x^3+uu*x^2+vv*x+ww ,P2); $\par
 $z=y^{(p-1)} ; t=z; for(i=1, 5, z=z^p*t); U=component(z-1,2);  $\par
 $u1=component(U,1); if(u1==0, u1=Mod(0,p2)); $\par
 $u2=component(U,2); if(u2==0, u2=Mod(0,p2)); $\par
 $u3=component(U,3); if(u3==0, u3=Mod(0,p2)); $\par
 $u4=component(U,4); if(u4==0, u4=Mod(0,p2)); $\par
 $u5=component(U,5); if(u5==0, u5=Mod(0,p2)); $\par
 $u6=component(U,6); if(u6==0, u6=Mod(0,p2)); $\par
 $x1=component(u1,2)/p; x2=component(u2,2)/p; x3=component(u3,2)/p;  $\par
 $x4=component(u4,2)/p; x5=component(u5,2)/p; x6=component(u6,2)/p;  $\par
 $X1=Mod(x1,p); X2=Mod(x2,p); X3=Mod(x3,p);  $\par
 $X4=Mod(x4,p); X5=Mod(x5,p); X6=Mod(x6,p);  $\par
 $E1= Mod(X1+X2*e1+X3*e1^2+X4*e1^3+X5*e1^4+X6*e1^5, P); $\par
 $E2=Mod(X1+X2*e2+X3*e2^2+X4*e2^3+X5*e2^4+X6*e2^5, P); $\par
 $E3=Mod(X1+X2*e3+X3*e3^2+X4*e3^3+X5*e3^4+X6*e3^5, P); $\par
 $E4=Mod(X1+X2*e4+X3*e4^2+X4*e4^3+X5*e4^4+X6*e4^5, P); $\par
 $E5=Mod(X1+X2*e5+X3*e5^2+X4*e5^3+X5*e5^4+X6*e5^5, P); $\par
 $E6=Mod(X1+X2*e6+X3*e6^2+X4*e6^3+X5*e6^4+X6*e6^5, P); $\par
 $F1=component(E1,2); F2=component(E2,2); F3=component(E3,2); $\par 
 $F4=component(E4,2); F5=component(E5,2); F6=component(E6,2);  $\par
 $SP=F1+ F2+F3; SM=F4+ F5+ F6;  FP= SP+SM; FM=SP-SM;   $\par
 $ RP=E1^2+E2^2+E3^2 - E4^2-E5^2-E6^2-E1*E2-E2*E3 -E3*E1 $\par
 \hfill  $+E4*E5+E5*E6+E6*E4; $\par
  $if(FP==0, N1=N1+1); $\par
   $if(FM==0,N2=N2+1); $\par
  $if(RP==0,N3=N3+1); $\par
  $if(FP==0 \& FM==0, N12=N12+1); $\par
  $if(FP==0 \& RP==0, N13=N13+1); $\par
   $if(FM==0 \& RP==0, N23=N23+1); $\par
  $if(FP==0 \& FM==0 \& RP==0, N123=N123+1) ) ); $\par
   $print("p = ",p,"     ","N1 = ",N1,"   ","N2= ",N2, "     ","N3 = ",N3,"    ", $\par
     $(N1+0.0)/N0,"    ",(N2+0.0)/N0,"    ",(N3+0.0)/N0 );   $\par
   $print("N12 = ",N12,"   ","N13= ",N13, "     ","N23 = ",N23,"    ","N123 = ",N123,"    ", $\par
   $  (N12+0.0)/N0,"    ",(N13+0.0)/N0,"    ",(N23+0.0)/N0,"         ", (N123+0.0)/N0);   $\par
   $print("      ","1/p = ",1./p,"        ","1/p^2 = ",1./p^2,"    ","1/p^3 = ",1./p^3)$$\}$
\normalsize

\medskip
Pour $p = 13$ on obtient les valeurs suivantes~: 

\smallskip
$N_0= 999115$ ; $N_1= 76820$ ; $N_2= 77009$ ; $N_3= 82239$ ; 

\smallskip
$N_{12 }= 5898$ ; $N_{13 }= 6301$  ; $N_{23 }= 6453$ ; $N_{123 }= 442$, et les densit\'es respectives~:

\smallskip
$\frac{N_1}{N_0} =  0.076888$ ; $\frac{N_2}{N_0} =   0.07707 $ ; $\frac{N_3}{N_0} = 0.0823$ ; 

\smallskip
$\frac{N_{12}}{N_0} = 0.00590$ ; $\frac{N_{13}}{N_0} =   0.006306$ ; $\frac{N_{23}}{N_0} =  0.006458$ ; 
$\frac{N_{123}}{N_0} =  0.0004424$ ; 

\smallskip
avec $\frac{1}{p} = 0.07692$,  $\frac{1}{p^2} = 0.005917$,  $\frac{1}{p^3} = 0.000455$, 
d'o\`u les probabilit\'es attendues.

\medskip
Pour $p = 37$ on obtient les valeurs suivantes~: 

\smallskip
$N_0= 999952$ ; $N_1= 27153$ ; $N_2= 27054$ ; $N_3= 27747$ ; 

\smallskip
$N_{12 }= 718$ ; $N_{13 }= 761$  ; $N_{23 }= 755$ ; $N_{123 }= 16$,  et les densit\'es respectives~:

\smallskip
$\frac{N_1}{N_0} =  0.0271543$ ; $\frac{N_2}{N_0} =  0.027055$ ; $\frac{N_3}{N_0} = 0.0277483$ ; 

\smallskip
$\frac{N_{12}}{N_0} =  0.000718$ ; $\frac{N_{13}}{N_0} =   0.000761$ ; $\frac{N_{23}}{N_0} =  0.000755$ ; 
$\frac{N_{123}}{N_0} =   1.600 \times 10^{-5}$,

\smallskip
avec $\frac{1}{p} = 0.027027$,  $\frac{1}{p^2} = 0.00073046$,  $\frac{1}{p^3} =  1.97 \times 10^{-5}$.

\medskip
On donnera au\,\S\ref{sub112} des programmes effectuant ces statistiques \`a partir du calcul du rang de la matrice des coefficients des conjugu\'es de $\alpha$ (groupes $C_3$, $C_5$ et $D_6$). Pour les groupes cycliques, le degr\'e
r\'esiduel intervient contrairement au cas de $D_6$, o\`u $f=1$ quel que soit $p$, mais pour $D_6$ l'entier $\delta$ vaut 0, 1 ou~2.

\subsection{Principes d'analyse -- Lin\'earisation du probl\`eme}\label{sub13}
Soit $\eta \in K^\times$ donn\'e~; on suppose pour simplifier que le $G$-module engendr\'e par $\eta$ est de $\Z$-rang 
$n = \vert G\vert$ (sinon on doit tenir compte des nullit\'es triviales).

\smallskip
On consid\`ere le $G$-module ${\mathcal L}$ engendr\'e par les relations \`a coefficients $p$-entiers rationnels 
(cf. D\'efinition \ref{defi11},\,\S\ref{sub155}). 
Comme il a \'et\'e vu dans le\,\S\ref{sub9}, la condition $\Delta_p^\theta(\eta) \equiv 0 \pmod p$ \'equivaut \`a
la non trivialit\'e de ${\mathcal L}^\theta$
(cf. Th\'eor\`eme \ref{theo24},\,\S\ref{sub15}) et c'est cet aspect ``lin\'earis\'e'' qui va permettre l'analyse probabiliste.

\smallskip
Le cas $\chi =\theta= 1$ a lieu pour tout groupe $G$ et $\Delta^1_p(\eta)$ est congru modulo $p$ au $p$-quotient de Fermat 
$q_p(a)$ de $a = \No_{K/\Q}(\eta)$. On peut admettre a priori que ce quotient de Fermat a la probabilit\'e 
$\frac{1}{p}$ d'\^etre \'egal \`a une valeur $t\pmod p$ donn\'ee, mais nous reviendrons sur cette question au\,\S\ref{subalt}
pour sugg\'erer que ceci semble arbitraire (voir cependant dans \cite{Hat},\,\S\,2, l'\'etude statistique de la distribution de $q_p(2)$ modulo $p \leq 1010783$, mais $p$ n'est ici pas tr\`es grand).

\smallskip
La relation $\No_{K/\Q}(\eta)=a$ se transmet \`a $\alpha = \alpha_p(\eta)$ via la formule $\sum_{\nu \in G} \alpha^\nu \equiv q_p(a)  \pmod p$ qui constitue, pour $\theta=1$, la $\theta$-relation lin\'eaire non triviale associ\'ee \`a  la nullit\'e modulo $p$ de $\Delta_p^1(\eta)$ lorsque $q_p(a) \equiv 0 \pmod p$.

\smallskip
 L'existence d'une $\theta$-relation lin\'eaire non triviale entre les conjugu\'es de $\alpha$ peut toujours, pour la commodit\'e num\'erique des programmes, s'exprimer  sur $n$ variables rationnelles de la forme $A_1, \ldots, A_n$, o\`u les $A_i$ sont par exemple les composantes de $\alpha$ sur une $\Z_{(p)}$-base de $Z_{K,(p)}$. 

\smallskip
Pour introduire cet aspect statistique par le calcul, nous  reprenons en d\'etail  le cas cubique cyclique,\,\S\ref{sub19}, 
puis le cas Ab\'elien en toute g\'en\'eralit\'e,\,\S\ref{sub27}. Dans le \S\ref{sub29}, nous traitons un cas non Ab\'elien via le groupe $D_6$.

\subsubsection{Principes heuristiques et probabilistes}\label{sub14}
On peut  poser les principes heuristiques suivants ($p$ fix\'e assez grand et 
utilisation de la fonction {\it random} de PARI pour d\'efinir $\eta$ \'etranger \`a $p$)~:

\smallskip
(i) On peut se limiter \`a conna\^itre le nombre $\eta$ modulo $p^2$ puisque $\alpha$ est d\'efini par
$\alpha \equiv \frac{-1}{p} {\rm log}_p (\eta) \pmod p$.

Bien que la correspondance entre $\eta$ modulo $p^2$ ($\eta$ \'etranger \`a $p$) et $\alpha \equiv \frac{-1}{p} {\rm log}_p(\eta)$
modulo $p$ soit {\it alg\'ebriquement} canonique, nous avons pr\'ef\'er\'e partir de $\eta$ afin de respecter l'aspect diophantien, d'autant que par utilisation de {\it random}, tous les $\eta$ ne sont pas atteints. Par exemple, n\'egliger cet aspect pour le cas 
$K = \Q$ reviendrait \`a utiliser la fonction  {\it random} sur l'intervalle $[0, p-1]$ en notant le nombre de fois o\`u 0 est pris, ce qui n'a plus rien \`a voir avec le sujet !

\smallskip
Si d'un point de vue $p$-adique, $\alpha$ parcourt de fa\c con \'equiprobable l'anneau
quotient $Z_{K,(p)}/(p) \simeq \F_p^{\,n}$, l'exp\'erience montre que les r\'esultats statistiques restent excel\-lents si pour $p$ assez grand on limite la fonction  {\it random} (pour d\'efinir $\eta$) \`a un petit domaine de $(\Z/p^2\Z)^n$, ce qui pr\'eserve l'aspect diophantien et d\'emontre une uniformit\'e (limitation obligatoire lorsque $p$ est tr\`es grand).

\smallskip
(ii) Soit $(e_i)_{i=1,\ldots,n}$ une $\Z_{(p)}$-base de $Z_{K,(p)}$ et posons $\alpha = \sum_{i=1}^n A_i e_i$, $A_i \in \Z_{(p)}$~;
alors, modulo $p$,  les variables $A_i$ sont ind\'ependantes et \'equiprobables dans $\F_p$, et ceci ne d\'epend pas du groupe 
$G$ ni du choix de la base.

\smallskip
(iii) Toute relation non triviale de la forme $\sum_{\nu \in G} u(\nu)\,\alpha^{\nu^{\!-1}}
\equiv 0 \pmod p$ se traduit par une relation non triviale analogue sur les $A_i$
(ceci r\'esulte du fait que les conjugu\'es des $e_j$ sont des formes lin\'eaires en les $e_i$ ind\'ependantes de $p$).

\subsubsection{Heuristique principale.} \label{HP} La probabilit\'e de $\Delta_p^\theta (\eta)\equiv 0 \pmod p$ est celle de ${\mathcal L}^\theta \ne \{0\}$ (cf. Th\'eor\`eme \ref{theo24},\,\S\ref{sub15}). Elle est induite par le nombre de relations $\F_p[G]$-ind\'e\-pen\-dantes. 
Compte tenu de la Remarque \ref{remadim},\,\S\ref{sub15}, si ${\mathcal L}^\theta \simeq \delta  V_\theta$, $1 \leq \delta  \leq \varphi(1)$,  
on peut justifier, de la fa\c con suivante, que l'on doit affecter \`a ce cas une proba\-bilit\'e en $\frac{O(1)}{p^{f \delta^2}}$, o\`u $f$ est le degr\'e r\'esiduel de $\theta$ (i.e., celui de $p$ dans le corps des valeurs $C_\chi$ de $\varphi \div \theta$), o\`u l'on consid\`ere $V_\theta$ comme $\F_p$-repr\'esentation et ensuite, par extension des scalaires, $V_\theta$ et $V_\varphi $ 
comme $\F_{p^f}$-repr\'esentations~:

\smallskip
On a ${\mathcal L}^\theta = \bigoplus_{\varphi \div \theta} {\mathcal L}^\varphi$ o\`u ${\mathcal L}^\varphi \simeq \delta V_\varphi$. L'id\'ee est que si ${\mathcal L}^\varphi \simeq \varphi(1) V_\varphi \simeq e_\varphi \F_{p^f}[G]$, (i.e., $e_\varphi \alpha \equiv 0 \pmod p$) alors la probabilit\'e correspondante est minimale en $\frac{O(1)}{p^{f\varphi(1)^2}}$ puisque $e_\varphi \alpha$ est d\'efini par $f \varphi(1)^2$ composantes ind\'ependantes dans  $\F_p$ (cf. (ii)).

\smallskip
 Or $e_\varphi \F_{p^f}[G] \simeq {\rm End}(V_\varphi)$ comme alg\`ebre d'endomorphismes d'un $\F_{p^f}$-espace vecto\-riel de dimension $\varphi(1)$.
Par cons\'equent,  ${\mathcal L}^\varphi \simeq \delta V_\varphi$ est alors vue comme sous-alg\`ebre d'endomorphismes d'un 
$\F_{p^f}$-espace vectoriel de dimension $\delta$, d'o\`u une probabilit\'e  en $\frac{O(1)}{p^{f\delta^2}}$
d'avoir ${\mathcal L}^\varphi \simeq \delta V_\varphi$. 
On notera que  la probabilit\'e d'avoir tous les $\Delta_p^\theta (\eta)\equiv 0 \!\!\pmod p$ avec chaque fois $\delta=\varphi(1)$ 
 (i.e., $\alpha\equiv 0 \pmod p$, ou encore aux $n$ composantes de $\alpha$ nulles modulo $p$) est $\frac{O(1)}{p^n}$ puisque $\sum_\theta f \varphi(1)^2 = \vert G \vert=n$.

\smallskip
 Le cas non trivial le plus fr\'equent est $\delta =1$ (le degr\'e r\'esiduel $f$ d\'epend canoni\-quement de $p$ contrairement \`a $\delta $ qui est ``num\'erique''). Par exemple, le passage de $\delta =1$ \`a $\delta =2$ (pour $f=1$) fait passer les probabilit\'es de $\frac{O(1)}{p}$ \`a $\frac{O(1)}{p^{4}}$, quasi nulle pour $p$ grand (tr\`es bien confirm\'e par les statistiques num\'eriques,
cf. \S\,\ref{subD6}).

\begin{example}\label{ex11}{\rm 
Cas de $G=D_6$ ($f=1$, $1 \leq \delta  \leq 2$). Soit $\theta$ le 
caract\`ere irr\'eductible de degr\'e 2~; la repr\'esentation $e_\theta \F_p[G]$ est isomorphe \`a $2\,V_\theta$ 
o\`u $V_\theta$ est de $\F_p$-dimension 2. 
On peut par exemple engendrer $e_\theta \F_p[G]$ de la fa\c con suivante (cf. Remarque \ref{rema511},\,\S\ref{sub5})~:

\smallskip
On consid\`ere les $\theta$-relations $U_1 = 1-\sigma^2 + \tau -\tau\sigma$ et $U_2 = 1-\sigma - \tau +\tau\sigma$, 
et on v\'erifie que l'on a les expressions~:
\begin{eqnarray*}
             U_1 &=& 1-\sigma^2 + \tau -\tau\sigma, \ \ \ \sigma U_1 = \sigma-1 + \tau\sigma^2 -\tau ,  \ \ \ \ \,
\sigma^2 U_1 = -U_1- \sigma U_1, \\
           U_2 &=& 1-\sigma - \tau +\tau\sigma, \ \ \   \sigma U_2 = -\sigma^2+\sigma + \tau -\tau\sigma^2  \ \ \ 
 \sigma^2 U_2 =  -U_2- \sigma U_2,
\end{eqnarray*}
\centerline{ $\tau U_1 = -\sigma U_1$,  $\ \tau\sigma U_1 = - U_1$, $\ \tau\sigma^2 U_1 = -\sigma^2 U_1$, }

\smallskip
\centerline{ $\tau U_2 = - U_2$, $\ \tau\sigma U_2 = - \sigma^2 U_2$, $\ \tau\sigma^2 U_2 = -\sigma U_2$. }

\medskip
Les quatre \'el\'ements $U_1, \sigma U_1, U_2, \sigma U_2$ forment une $\F_p$-base de l'espace des $\theta$-relations 
possibles, ce qui  justifie la probabilit\'e en $\frac{O(1)}{p}$ seulement pour le cas $\delta \geq 1$ mais en $\frac{O(1)}{p^4}$
pour $\delta = 2$. }
\end{example}

\subsection{Statistiques sur le rang de la matrice des coefficients}\label{sub112}

Une premi\`ere statistique consiste \`a d\'eterminer la probabilit\'e d'avoir au moins une relation non triviale entre les conjugu\'es
de $\alpha$~; si $\alpha^\nu = \sum_{i=1}^n A_i(\nu) \,e_i$, alors la matrice $\big( A_i(\nu) \big)_{i, \nu}$ doit \^etre de 
$\F_p$-rang strictement inf\'erieur \`a $n$. 

\smallskip
Nous avons vu que pour $\theta \div \chi$ la probabilit\'e de nullit\'e modulo $p$ de $\Delta_p^\theta(\eta)$ 
seul est en $\Frac{1}{p^f}$, o\`u $f$ est le degr\'e r\'esiduel de $\theta$ (i.e., celui de
$p$ dans $C_\chi/\Q$)~; comme il y a $h = \frac{[C_\chi : \Q]}{f}$ tels caract\`eres
$p$-adiques $\theta \div \chi$, la probabilit\'e d'avoir au moins un $\Delta_p^\theta(\eta)$ nul modulo $p$ 
pour $\theta \div \chi$ est en $\Frac{h}{p^f}$. 

\smallskip
Par cons\'equent si l'on d\'esigne par $\theta_i$, $h_i$, $f_i$ les param\`etres ci-dessus pour la totalit\'e des 
caract\`eres $p$-adiques de $G$ (regroup\'es par caract\`eres rationnels $\chi_i$), la probabilit\'e th\'eorique d'obtenir une matrice de $\F_p$-rang $< n$ est donn\'ee par~:
$$\sm_{i} \Frac{h_i}{p^{f_i}} - \sm_{i< j} \Frac{h_i}{p^{f_i}}\Frac{h_j}{p^{f_j}} + 
\sm_{i< j< k} \Frac{h_i}{p^{f_i}}\Frac{h_j}{p^{f_j}} \Frac{h_k}{p^{f_k}} - \cdots \, ,$$

ce que l'on peut v\'erifier au moyen des programmes suivants calculant, pour des $\eta$  al\'eatoires,
le nombre de cas de  $\F_p$-rang $< n$ ($G \simeq C_3, C_5, D_6$ respectivement)~:

\smallskip
\subsubsection {Cas $G$ cyclique d'ordre 3\ \rm (deux caract\`eres rationnels)}\label{subC3}
Dans le cas $p\equiv 1 \pmod 3$, on a trois caract\`eres
$p$-adiques de degr\'e r\'esiduel $f=1$, dans  le cas $p\equiv 2\pmod 3$, on a un caract\`ere
$p$-adique de degr\'e r\'esiduel $f=2$ et le caract\`ere unit\'e.
\footnotesize

\medskip
  $\{$$t=11;  Q=x^3-t*x^2-(t+3)*x-1;  p=43;  p2=p^2; N0=0;  N3=0; $\par$ 
   for(i=1, 500000,  
   aa=random(p2);  bb=random(p2);  cc=random(p2); $\par$  
  e=Mod(aa*x^2+bb*x+cc, Q); N=norm(e);    if(Mod(N,p)!=0, N0=N0+1; $\par$ 
  a=Mod(aa,p2);   b=Mod(bb,p2);   c=Mod(cc,p2); $\par$   
  P=x^3-t*x^2-(t+3)*x-Mod(1,p);  P2=x^3-t*x^2-(t+3)*x-Mod(1,p2) ; $\par$  
  y=Mod(a*x^2+b*x+c,P2);   u= y^{(p+1)} ;   v=u^p*y ;   z=v^{(p-1)}  -1; $\par$  
  U=component(z,2); $\par$   
  u1=component(U,1);   if(u1==0, u1=Mod(0,p2)); $\par$  
  u2=component(U,2);   if(u2==0, u2=Mod(0,p2)) ; $\par$  
  u3=component(U,3);   if(u3==0, u3=Mod(0,p2)); $\par$  
  x1=component(u1,2)/p;   x2=component(u2,2)/p;   x3=component(u3,2)/p; $\par$  
  X1=Mod(x1,p);   X2=Mod(x2,p);   X3=Mod(x3,p); $\par$   
  SX1=X1-2*X2-X3*(t-2); $\par$   
  SX2=-X3*(t^2+t+1)-X2*(t+1); $\par$   
  SX3=X3*t+X2; $\par$  
  TX1=X3*(t^2+3*t+4)+X2*(t+2)+X1; $\par$   
  TX2=X3*(t^2+t+1)+X2*t; $\par$  
  TX3=-X3*(t+1)-X2; $\par$  
  E1=Mod(X1+X2*x+X3*x^2, P); $\par$   
  E2=Mod(SX1+SX2*x+SX3*x^2, P); $\par$  
  E3=Mod(TX1+TX2*x+TX3*x^2, P); $\par$ 
 F1=component(E1,2);   F2=component(E2,2);   F3=component(E3,2); GG=[F1, F2, F3]; $\par$ 
  M=matrix(3, 3, i, j, component ( component (GG, i), j) );   r=matrank(M); $\par$ 
 if(r<3, N3=N3+1 ))) ; print("p = ",p,"     ",N0,"   ","N3 = ",N3, "    ",(N3+0.0)/N0); $\par$ 
  res=Mod(p,3);  if(res==Mod(1,3), print(3./p -3./p^2 +1./p^3)); $\par$
  if(res!=Mod(1,3), print(1./p +1./p^2-1./p^3))$$\}$
\normalsize

\medskip
On obtient les exemples suivants~: 

$p = 43$,    $N_0=  4999952$,   $N3 = 341000 $,    $\frac{N_3}{N_0} = 0.068200$,
probabilit\'e $0.068685$.

$p = 41$,    $N_0= 4999931 $,   $N3 = 124889$,    $\frac{N_3}{N_0} = 0.024978$,
probabilit\'e $0.024970$.

\medskip
\subsubsection {Cas $G$ cyclique d'ordre 5\ \rm (deux caract\`eres rationnels)}\label{subC5}
 Ce cas est le seul des cas \'etudi\'es pour lequel 
il y a (pour $p \equiv -1 \pmod 5$) deux caract\`eres $p$-adiques de degr\'e r\'esiduel $f=2$.
\footnotesize 

\medskip
  $\{$$Q=x^5+x^4-4*x^3-3*x^2+3*x+1;     p=13;   p2=p^2; $\par$
  s1=x;   s2=x^2-2;   s3=x^4 - 4*x^2 + 2;   s4= x^3-3*x;   s5=-x^4 - x^3 + 3*x^2 + 2*x - 1; $\par$ 
  N0=0;   N5=0; $\par$ 
  for(i=1, 500000, $\par$ 
  aa=random(p2);   bb=random(p2);   cc=random(p2);   dd=random(p2);  ee=random(p2); $\par$
  Eta=Mod(aa*x^4+bb*x^3+cc*x^2+dd*x+ee,Q); $\par$  
  N=norm(Eta);    if(Mod(N,p)!=0, N0=N0+1; $\par$ 
  a=Mod(aa,p2);   b=Mod(bb,p2);   c=Mod(cc,p2);   d=Mod(dd,p2);   e=Mod(ee,p2); $\par$   
  P=x^5+x^4-4*x^3-3*x^2+3*x+Mod(1,p); P2=x^5+x^4-4*x^3-3*x^2+3*x+Mod(1,p2); $\par$  
  y=Mod(a*x^4+b*x^3+c*x^2+d*x+e,P2); $\par$   
  u=y^{(p+1)}  ;   v=u^p*y;   w=v^p*y;   s=w^p*y;   t=s^{(p-1)}  -1; $\par$  
  U=component(t,2); $\par$  
  u1=component(U,1);   if(u1==0, u1=Mod(0,p2)); $\par$  
  u2=component(U,2);   if(u2==0, u2=Mod(0,p2)); $\par$  
  u3=component(U,3);   if(u3==0, u3=Mod(0,p2)); $\par$  
  u4=component(U,4);   if(u4==0, u4=Mod(0,p2)); $\par$  
  u5=component(U,5);   if(u5==0, u5=Mod(0,p2)); $\par$  
  x1=component(u1,2)/p;   x2=component(u2,2)/p;   x3=component(u3,2)/p;  $\par
  \hfill $x4=component(u4,2)/p;   x5=component(u5,2)/p; $\par$  
  X1=Mod(x1,p);   X2=Mod(x2,p);   X3=Mod(x3,p);   X4=Mod(x4,p); X5=Mod(x5,p); $\par$ 
  E1=  Mod(X5*s1^4+X4*s1^3+X3*s1^2+X2*s1+X1,P); $\par$ 
  E2=  Mod(X5*s2^4+X4*s2^3+X3*s2^2+X2*s2+X1,P); $\par$ 
  E3=  Mod(X5*s3^4+X4*s3^3+X3*s3^2 +X2*s3+X1,P); $\par$ 
  E4=  Mod(X5*s4^4+X4*s4^3 +X3*s4^2+X2*s4+X1,P); $\par$ 
  E5=  Mod(X5*s5^4+X4*s5^3 +X3*s5^2+X2*s5+X1,P); $\par$ 
  F1=component(E1,2);   F2=component(E2,2);   F3=component(E3,2);   $\par$ 
 \hfill F4=component(E4,2);   F5=component(E5,2); $\par$
  GG=[F1, F2, F3, F4, F5]; $\par$ 
 M=matrix(5, 5, i, j, component ( component (GG, i), j) );   r=matrank(M); $\par$ 
  if(r<5, N5=N5+1 ))) ;  print("p = ",p,"     ",N0,"   ","N5 = ",N5, "    ",(N5+0.0)/N0); $\par$ 
  res=Mod(p,5);   if(res==Mod(1,5), print(5./p-10./p^2 +10./p^3 -5./p^4 + 1./p^5 ) ); $\par$
  if(res^2==Mod(-1,5), print(1./p+1./p^4-1./p^5) ); $\par$
  if(res==Mod(-1,5), print(1./p+2./p^2-2./p^3 - 1./p^4 + 1./p^5) )$$\}$
 \normalsize

\medskip
Valeurs num\'eriques obtenues~:

\smallskip
$p = 7\ \,$,    $N_0= 499977$,   $N5 = 71650 $,    $\frac{N_5}{N_0} = 0.14330$,
probabilit\'e $0.143214$.

$p = 19$,    $N_0= 500000$,   $N5 = 29033$,    $\frac{N_5}{N_0} = 0.05806$,
probabilit\'e $0.057880$.

$p = 31$,    $N_0= 500000$,   $N5 = 75737$,    $\frac{N_5}{N_0} = 0.15147$,
probabilit\'e $0.151214$.

\medskip
En modifiant la fin du programme comme suit~:
\footnotesize

\medskip
  $F1=component(E1,2); F2=component(E2,2); F3=component(E3,2); $\par$ 
  \hfill F4=component(E4,2); F5=component(E5,2); $\par$
  H= F1+ F2+F3+ F4+ F5; $\par$ 
  R1 = F1+ 2*F2+4*F3+ 8*F4+ 16*F5; R2 = F1+ 4*F2+16*F3+ 2*F4+ 8*F5; $\par$ 
  R3 = F1+ 8*F2+2*F3+ 16*F4+ 4*F5; R4 = F1+ 16*F2+8*F3+ 4*F4+ 2*F5; $\par$ 
  M=matrix(5, 5, i, j, component ( component (GG, i), j) ); $\par$ 
  if(H!=0 \& R1==0 \& R2==0 \& R3!=0 \& R4!=0, N1=N1+1 ))) ; $\par$ 
  print("p = ",p,"     ",N0,"   ","N1 = ",N1, "    ",(N1+0.0)/N0); print(1./p^2) $
\normalsize

\medskip
on teste la fr\'equence de nullit\'e modulo $p$ des $\theta$-r\'egulateurs relatifs 
\`a deux caract\`eres $p$-adiques ($p=31$ totalement d\'ecompos\'e), et deux seulement parmi les cinq, \`a savoir
par exemple pour $\theta_1$ et $\theta_2$ d\'efinis par $\theta_1(\sigma^{\!-1}) \equiv 2$, 
$\theta_2(\sigma^{\!-1}) \equiv 4 \pmod p$~:

\medskip
\centerline{$R_1 = \alpha + 2\alpha^\sigma + 4\alpha^{\sigma^2} + 8\alpha^{\sigma^3} +16 \alpha^{\sigma^4}$ et 
$R_2 = \alpha + 4\alpha^{\sigma} + 16\alpha^{\sigma^2} +2 \alpha^{\sigma^3} + 8\alpha^{\sigma^4}$.}

\medskip
Pour $N_0 = 1000000$, $N_1 = 943$, on a $\frac{N_1}{N_0} = 0.000943$, et
la probabilit\'e $0.001040$, ce qui montre l'ind\'ependance des r\'egulateurs relatifs 
\`a des caract\`eres $p$-adiques d'un m\^eme caract\`ere rationnel.
Les autres combinaisons donnent des r\'esultats semblables.

\medskip
\subsubsection {Cas $G$ di\'edral d'ordre 6\ \rm (trois caract\`eres rationnels et $p$-adiques)}\label{subD6}
Dans ce cas on a $h = f =1$ pour tous les caract\`eres. On doit \'egalement tester les probabilit\'es relevant de la valeur de 
$\delta$.

\footnotesize
\medskip
  $\{$$Q=x^6+9*x^4-4*x^3+27*x^2+36*x+31;   p=17;   p2=p^2; $\par$ 
  e1= x ; $\par$   
  e2= 11/180*x^5 + 1/180*x^4 + 11/18*x^3 - 1/45*x^2 + 403/180*x + 419/180 ; $\par$  
  e3= 13/180*x^5 - 7/180*x^4 + 13/18*x^3 - 38/45*x^2 + 509/180*x + 127/180 ; $\par$  
  e4= -4/45*x^5 + 1/45*x^4 - 8/9*x^3 + 26/45*x^2 - 137/45*x - 91/45 ; $\par$   
  e5= -1/60*x^5 - 1/60*x^4 - 1/6*x^3 - 4/15*x^2 - 73/60*x - 79/60 ; $\par$  
  e6= -1/36*x^5 + 1/36*x^4 - 5/18*x^3 + 5/9*x^2 - 65/36*x + 11/36 ; $\par$  
  N0=0; N1=0;  N6=0; $\par$
  for(i=1, 50000, 
  a=random(p2);   b=random(p2);    c=random(p2);   aa=random(p2);   $\par$ 
\hfill bb=random(p2);   cc=random(p2); $\par$  
  Eta= Mod(a*e1^5+b*e1^4+c*e1^3+aa*e1^2+bb*e1+cc,Q); $\par$  
  N=norm(Eta);  if (Mod(N,p)!=0, N0=N0+1; $\par$  
  u=Mod(a,p2);   v=Mod(b,p2);   w=Mod(c,p2);  uu=Mod(aa,p2);   $\par$ 
\hfill  vv=Mod(bb,p2);   ww=Mod(cc,p2); $\par$   
  P=x^6+9*x^4-4*x^3+27*x^2+36*x+Mod(31,p); $\par$  
  P2=x^6+9*x^4-4*x^3+27*x^2+36*x+Mod(31,p2); $\par$ 
  y=Mod(u*x^5+v*x^4+w*x^3+uu*x^2+vv*x+ww ,P2); $\par$  
  z=y^{(p-1)}  ;   t=z;   for(i=1, 5, z=z^p*t); $\par$   
  U=component(z-1,2); $\par$   
  u1=component(U,1);   if(u1==0, u1=Mod(0,p2)) ; $\par$  
  u2=component(U,2);   if(u2==0, u2=Mod(0,p2)) ; $\par$  
  u3=component(U,3);   if(u3==0, u3=Mod(0,p2)) ; $\par$  
  u4=component(U,4);   if(u4==0, u4=Mod(0,p2)) ; $\par$  
  u5=component(U,5);   if(u5==0, u5=Mod(0,p2)) ; $\par$   
  u6=component(U,6);   if(u6==0, u6=Mod(0,p2)) ; $\par$  
  x1=component(u1,2)/p;   x2=component(u2,2)/p;   x3=component(u3,2)/p; $\par$
  x4=component(u4,2)/p; x5=component(u5,2)/p;   x6=component(u6,2)/p; $\par$   
  X1=Mod(x1,p);   X2=Mod(x2,p);   X3=Mod(x3,p); $\par$
  X4=Mod(x4,p);   X5=Mod(x5,p);   X6=Mod(x6,p); $\par$   
  E1= Mod(X1+X2*e1+X3*e1^2+X4*e1^3+X5*e1^4+X6*e1^5, P); $\par$  
  E2=Mod(X1+X2*e2+X3*e2^2+X4*e2^3+X5*e2^4+X6*e2^5, P); $\par$  
  E3=Mod(X1+X2*e3+X3*e3^2+X4*e3^3+X5*e3^4+X6*e3^5, P); $\par$  
  E4=Mod(X1+X2*e4+X3*e4^2+X4*e4^3+X5*e4^4+X6*e4^5, P); $\par$  
  E5=Mod(X1+X2*e5+X3*e5^2+X4*e5^3+X5*e5^4+X6*e5^5, P); $\par$  
  E6=Mod(X1+X2*e6+X3*e6^2+X4*e6^3+X5*e6^4+X6*e6^5, P); $\par$  
  F1=component(E1,2);   F2=component(E2,2);   F3=component(E3,2);$\par$
  F4=component(E4,2);   F5=component(E5,2);   F6=component(E6,2); $\par$   
  GG=[F1, F2, F3, F4, F5, F6]; $\par$   
   M=matrix(6, 6, i, j, component ( component (GG, i), j) );   r=matrank(M); $\par$ 
  if(r<6, N6=N6+1) )  ); $\par$ 
   print("p = ",p,"     ",N0,"    ","N6 = ",N6, "    ",(N6+0.0)/N0);  print(3./p - 3./p^2 + 1./p^3)$$\}$
\normalsize

\medskip
Les r\'esultats ne d\'ependent pas de classes de congruences de $p$ car on a $C_\chi = \Q$~:

\smallskip
$p = 13$,    $N_0= 49954$,   $N6 = 10794$,    $\frac{N_6}{N_0} = 0.21607$,
probabilit\'e $0.21347$.

$p = 17$,    $N_0= 49516$,   $\ \,N6 = 8337$,    $\frac{N_6}{N_0} = 0.16836$,
probabilit\'e $0.16629$.

$p = 29$,    $N_0= 49815$,   $\ \,N6 = 5056 $,    $\frac{N_6}{N_0} = 0.10149$,
probabilit\'e $0.09992$.

$p = 31$,    $N_0= 40982$,   $\ \,N6 = 3854$,    $\frac{N_6}{N_0} = 0.09404$,
probabilit\'e $0.09368$.

$p = 37$,    $N_0= 49998$,   $\ \,N6 = 3959$,    $\frac{N_6}{N_0} = 0.07918$,
probabilit\'e $0.07890$.

\medskip
On reprend ensuite le m\^eme programme en faisant les statistiques du cas $\delta = 2$, ce qui peut se tester en recherchant les cas o\`u les r\'egulateurs $\Delta_p^1(\eta)$ et $\Delta_p^{\chi_1}(\eta)$ sont non nuls modulo $p$ et la matrice des coefficients 
de rang 2 ($\Delta_p^\theta(\eta) \equiv 0 \pmod p$ pour $\theta = \chi_2$ et ${\mathcal L}^\theta$ de dimension 4)~; nombre de cas dans $N_6$.

\smallskip
On remplace la fin du programme par la s\'equence suivante~:

\footnotesize
\medskip
 $ GG=[F1, F2, F3, F4, F5, F6]; $\par$   
  DD=(F1+ F2+ F3)^2 - (F4+ F5+ F6)^2; $\par$  
  RP=E1^2+E2^2+E3^2 - E4^2-E5^2-E6^2-E1*E2-E2*E3 -E3*E1 $\par$  
    \hfill   +E4*E5+E5*E6+E6*E4; $\par$  
  if(RP==0 \& DD !=0, N1=N1+1; $\par$  
  M=matrix(6, 6, i, j, component ( component (GG, i), j) );   r=matrank(M);   $\par$  
   if(r==2, N6=N6+1) ) ) );  $\par$  
   print("p = ",p,"     ",N0,"    ","N6 = ",N6, "    ",(N6+0.0)/N0,"    ",1./p^4);  $\par$  
print("N1 = ",N1, "    ",(N1+0.0)/N0,"    ",1./p) $$\}$

\normalsize
\medskip
On obtient les r\'esultats suivants pour $p=13$~:

\smallskip
$p = 13$~;   $N_0 = 499541$~;  $N_6 = 18$~;   $\frac{N_6}{N_0}= 3.60 \times 10^{-5}$~;
 $\frac{1}{p^4}=  3.50 \times 10^{-5}$~;
$N_1 = 34925$~;    $\frac{N_1}{N_0}=0.06991$~;  $\frac{1}{p}= 0.07692$.

\subsection{Ind\'ependance locale des composantes sur une base}\label{sub16}
  Il reste \`a v\'erifier le caract\`ere de ``variables al\'eatoires ind\'ependantes'' de $A_1, \ldots, A_n$~; 
nous ne donnerons que deux exemples num\'eriques ($G = C_3$ et $G=D_6$), 
avec des programmes PARI que le lecteur peut reprendre et exp\'erimenter comme nous allons le pr\'eciser ci-apr\`es.

\subsubsection{Cas cubique}\label{sub17}
Soit $K$ le corps cubique cyclique d\'efini par le polyn\^ome
$x^3-11 x^2-14 x-1$, de conducteur $163$. On rappelle que les nombres premiers $p$ consid\'er\'es
sont assez grands.

\smallskip
 Il s'agit de v\'erifier que les variables $A, B, C$, qui d\'efinissent $\alpha  \equiv A x^2+ B x + C \pmod p$ 
sont ind\'ependantes.

\smallskip
Le programme ci-dessous consid\`ere des entiers al\'eatoires $\eta$ modulo~$p^2$, \'etrangers \`a $p$~; il est clair que, {\it alg\'ebriquement}, $\frac{-1}{p}{\rm log}_p(\eta)$ donne uniform\'ement tous les $\alpha$ possibles modulo $p$, mais m\^eme la restriction de la fonction {\it random} \`a un petit sous-domaine de $(\Z/p^2\Z)^3$ donne des r\'esultats analogues.

\smallskip
Ensuite on calcule par exemple
le nombre de couples $(A, B)$ (resp. $(B, C)$, $(C, A)$) ayant une valeur fix\'ee arbitrairement dans $\F_p^2$, puis le nombre de 
cas o\`u  $\Delta_p^\chi(\eta) \equiv 0 \pmod p$. On peut prendre $p \geq 11$ distinct de $163$.

\smallskip
On d\'esigne par $N_0$ le nombre d'entiers $\eta$ modulo $p^2$ \'etrangers \`a $p$ consid\'er\'es, par
$N_1$ le nombre de cas o\`u $\Delta_p^\chi(\eta) \equiv 0 \pmod p$ ($\chi$ rationnel $\ne 1$), 
par  $N_2$ le nombre de couples $(A, B)$ ayant la valeur
impos\'ee modulo $p$, et le programme calcule les proportions $\frac{N_1}{N_0}$, $\frac{N_2}{N_0}$, ainsi que $\frac{2}{p}$ ou $\frac{1}{p^2}$.
\footnotesize 

\medskip
$\{$$t=11; Q=x^3-t*x^2-(t+3)*x-1; p=11; p2=p^2;$\par
 $N0=0; N1=0; N2=0;$\par
 $ for(i=1, 500000, $\par
 $ aa=random(p2); bb=random(p2); cc=random(p2); $\par
 $e=Mod(aa*x^2+bb*x+cc, Q); N=norm(e);  if(Mod(N,p)!=0, N0=N0+1;$\par
 $a=Mod(aa,p2); b=Mod(bb,p2); c=Mod(cc,p2);  $\par
 $P=x^3-t*x^2-(t+3)*x-Mod(1,p); P2=x^3-t*x^2-(t+3)*x-Mod(1,p2) ; $\par
 $y=Mod(a*x^2+b*x+c,P2); u= y^{(p+1)} ; v=u^p*y ; z=v^{(p-1)}  -1; $\par
 $U=component(z,2);  $\par
 $u1=component(U,1); if(u1==0, u1=Mod(0,p2)); $\par
 $u2=component(U,2); if(u2==0, u2=Mod(0,p2)) ; $\par
 $u3=component(U,3); if(u3==0, u3=Mod(0,p2)); $\par
 $x1=component(u1,2)/p; x2=component(u2,2)/p; x3=component(u3,2)/p; $\par
 $X1=Mod(x1,p); X2=Mod(x2,p); X3=Mod(x3,p);  $\par
 $if(X3==Mod(4,p)\& X2==Mod(1,p), N2=N2+1); $\par
 $SX1=X1-2*X2-X3*(t-2);  $\par
 $SX2=-X3*(t^2+t+1)-X2*(t+1);  $\par
 $SX3=X3*t+X2; $\par
 $TX1=X3*(t^2+3*t+4)+X2*(t+2)+X1;  $\par
 $TX2=X3*(t^2+t+1)+X2*t; $\par
 $TX3=-X3*(t+1)-X2; $\par
 $F=Mod(X1+X2*x+X3*x^2, P);  $\par
 $SF=Mod(SX1+SX2*x+SX3*x^2, P); $\par
 $TF=Mod(TX1+TX2*x+TX3*x^2, P);$\par
 $RP= F^2+SF^2+TF^2-F*SF-SF*TF-TF*F;$\par
 $rp=component(RP,2); if(rp==0, N1=N1+1) ) )  ;$\par
 $print(p,"     ",N0,"   ",N1,"   ",N2, "    ",(N1+0.0)/N0,"    ",(N2+0.0)/N0,"   ",1./p2)$$\}$
\normalsize

\smallskip
Les r\'esultats sont inchang\'es quelles que soient les valeurs impos\'ees num\'eriquement aux couples de coefficients
(on commence dans le tableau ci-dessous par deux cas de degr\'e r\'esiduel $2$ dans $\Q(j)/\Q$ et ensuite par des cas totalement d\'ecompos\'es)~:
$$\begin{array}{lllllllll}
\ p  &\   N_0  &\   N_1 &\   N_2 &\ \ \frac{N_1}{N_0} & \ \ \frac{N_2}{N_0} &\ \  \frac{1}{p^2} & \\  \vspace{-0.4cm} \\
5 &  255562  & 10023  & 10155 &  0.039219& 0.039736&  0.04 &  \\
11 &  499624 & 4127 &  4191 &  0.00826& 0.008388 & 0.00826&  \\ 
\ p  &\   N_0  &\   N_1 &\   N_2 &\ \ \frac{N_1}{N_0} & \ \ \frac{N_2}{N_0} &\ \  \frac{1}{p^2} &\ \  \frac{2}{p}\\  \vspace{-0.4cm} \\
7 &  498553 &   132167  & 10275   &0.2651& 0.0206 & 0.0204 &0.286 \\
13 &  392751 &  57826  &  2401   & 0.1472 & 0.006113& 0.005917 & 0.154\\
19 &  499907  & 51293 & 1421  &  0.1025 & 0.00284 & 0.00277& 0.105
\end{array} $$

Toutes les proportions $\frac{N_2}{N_0}$ sont proches de $\frac{1}{p^2}$. 
Dans les cas $p\equiv 1 \pmod 3$  les proportions $\frac{N_1}{N_0}$ sont proches de $\frac{2}{p}$
(existence de deux caract\`eres $p$-adiques),
et proches de $\frac{1}{p^2}$ dans le cas $p\equiv 2 \pmod 3$.
Naturellement, si l'on impose seulement une valeur num\'erique (\`a $A, B$ ou $C$) on obtient 
$\frac{N_2}{N_0} \sim \frac{1}{p}$ et $\frac{1}{p^3}$ si l'on impose les trois valeurs. 
On a donc bien une ind\'ependance statistique des variables $A, B, C$.

\subsubsection{Cas de $D_6$}\label{sub18}
 Une \'etude analogue utilise le programme PARI~suivant, o\`u $nc$ indique le nombre de composantes impos\'ees
($1 \leq nc \leq 6$)~:
\footnotesize 

\medskip
 $\{$$Q=x^6+9*x^4-4*x^3+27*x^2+36*x+31; p=7; p2=p^2; N0=0; N1=0; $\par
  $e1= x ;$\par
  $e2=11/180*x^5 + 1/180*x^4 + 11/18*x^3 - 1/45*x^2 + 403/180*x + 419/180 ;  $\par
  $e3=13/180*x^5 - 7/180*x^4 + 13/18*x^3 - 38/45*x^2 + 509/180*x + 127/180 ;  $\par
  $e4=-4/45*x^5 + 1/45*x^4 - 8/9*x^3 + 26/45*x^2 - 137/45*x - 91/45 ;   $\par
  $e5=-1/60*x^5 - 1/60*x^4 - 1/6*x^3 - 4/15*x^2 - 73/60*x - 79/60 ;  $\par
  $e6=-1/36*x^5 + 1/36*x^4 - 5/18*x^3 + 5/9*x^2 - 65/36*x + 11/36 ;  $\par
  $for(i=1, 500000, $\par
  $a=random(p2); b=random(p2); c=random(p2);$\par
  $aa=random(p2); bb=random(p2); cc=random(p2);$\par
  $E= Mod(a*e1^5+b*e1^4+c*e1^3+aa*e1^2+bb*e1+cc,Q);  $\par
  $N=norm(E);  if (Mod(11*N, p)!=0, N0=N0+1;$\par
  $u=Mod(a,p2); v=Mod(b,p2); w=Mod(c,p2); $\par  
  $uu=Mod(aa,p2); vv=Mod(bb,p2); ww=Mod(cc,p2); $\par  
  $P=x^6+9*x^4-4*x^3+27*x^2+36*x+Mod(31,p);  $\par
  $P2=x^6+9*x^4-4*x^3+27*x^2+36*x+Mod(31,p2);$\par  
  $y=Mod(u*x^5+v*x^4+w*x^3+uu*x^2+vv*x+ww ,P2); $\par 
  $z=y^{(p-1)}  ; t=z; for(i=1, 5, z=z^p*t); U=component(z-1,2); $\par  
  $u1=component(U,1); if(u1==0, u1=Mod(0,p2));  $\par
  $u2=component(U,2); if(u2==0, u2=Mod(0,p2));  $\par
  $u3=component(U,3); if(u3==0, u3=Mod(0,p2));  $\par
  $u4=component(U,4); if(u4==0, u4=Mod(0,p2));  $\par
  $u5=component(U,5); if(u5==0, u5=Mod(0,p2));  $\par
  $u6=component(U,6); if(u6==0, u6=Mod(0,p2));  $\par
  $x1=component(u1,2)/p; x2=component(u2,2)/p; x3=component(u3,2)/p; $\par
  $x4=component(u4,2)/p; x5=component(u5,2)/p; x6=component(u6,2)/p;   $\par
  $X1=Mod(x1,p); X2=Mod(x2,p); X3=Mod(x3,p); $\par
  $X4=Mod(x4,p); X5=Mod(x5,p); X6=Mod(x6,p);   $\par
  $nc=3; if(X1==Mod(4,p) \& X2==Mod(4,p) \& X5==Mod(1,p) , N1=N1+1)  ) );$\par
  $print(p,"     ",N0,"    ",N1, "    ",(N1+0.0)/N0,"   ",1./p^{nc})$$\}$
\normalsize

\medskip
Pour $p=17$,  on obtient pour trois conditions sur les composantes de $\alpha$,
$N_0 =  494865$, $N_1 = 111$ et  $\frac{N_1}{N_0} = 0.0002243$, pour $\frac{1}{p^3} =0.0002035$.
Toutes les exp\'erimentations num\'eriques ont donn\'e les r\'esultats attendus.

\vspace{-0.2cm}
\subsection{Extra $p$-divisibilit\'es} \label{subex}
Examinons le cas des $p$-divisibilit\'es sup\'erieures ``non automatiques'', pour voir que leur probabilit\'e est 
aussi au plus en $\frac{1}{p^2}$.

Rappelons la d\'ecomposition du r\'egulateur de~$\eta$~:

\vspace{-0.5cm}
\begin{eqnarray*}
{\rm Reg}_p^G (\eta) &\!\!\!=\!\!\!&\prd_\theta  {\rm Reg}_p^\theta (\eta)^{\varphi(1)}\ \ {\rm et}\ \  \\
{\rm Reg}_p^\theta  (\eta) &\!\!\!=\!\!\!& \prd_{\varphi \div \theta} {\rm Reg}_p^\varphi(\eta) = 
\No_{\mathfrak p}\big (P^\varphi \big(\ldots, \hbox{$\frac{-1}{p}{\rm log}_p(\eta)$}, \ldots \big)\big) 
\end{eqnarray*}

\vspace{-0.2cm}
(cf. Remarques \ref{rema110}, \ref{rema11}).
 Lorsque l'on est dans le cas de $p$-divisibilit\'e minimale, on a par d\'efinition (cf. D\'efinition \ref{defidec}, \S\ref{sub10})
${\rm Reg}_p^\theta (\eta)  \sim p \ \ {\rm et}\ \  {\rm Reg}_p^G (\eta) \sim p^{\varphi(1)}$.

\smallskip
Si l'on suppose seulement que $p$ est totalement d\'ecompos\'e dans $C_\chi/\Q$ ($f=1$) et
qu'il existe un unique $\theta$ tel que ${\rm Reg}_p^\theta (\eta) \equiv \Delta_p^\theta (\eta) \equiv 0 \pmod p$ 
(avec $\delta=1$), on a de possibles extra $p$-divisibilit\'es ${\rm Reg}_p^\theta (\eta) \sim p^e$, $e\geq 2$ (donc
${\rm Reg}_p^G (\eta) \sim  p^{e\,\varphi(1)}$), dont on veut v\'erifier qu'elles sont de probabilit\'e 
en $\frac{O(1)}{p^2}$ (en r\'ealit\'e $\frac{1}{p^2} +\frac{1}{p^3} + \cdots$).
Le programme suivant (pour $G=D_6$, auquel cas tout $p$ assez grand convient) v\'erifie ce fait pour le r\'egulateur~:
\begin{eqnarray*}
 {\rm Reg}_p^{\varphi} (\eta) \!&\!\!=\!\!&\! \hbox{$\frac{1}{\sqrt {-3} }$} (E_1^2+E_2^2+E_3^2 - E_4^2-E_5^2-E_6^2  -E_1. E_2-E_2.E_3 -E_3.E_1 \\
&& \hspace{2cm} +E_4.E_5+E_5.E_6+E_6.E_4) \in \Z ,
\end{eqnarray*}

lorsque $\varphi=\theta = \chi_2$ est le caract\`ere de degr\'e 2, o\`u les $E_i$, $1\leq i \leq 6$, sont les conjugu\'es
d'un entier quelconque de $K$ (en effet, on peut supposer que $\frac{-1}{p}{\rm log}_p(\eta)$ est repr\'esent\'e par un 
entier arbitraire $E$ de $K$). La division par $\sqrt{-3}$ permet d'avoit un entier rationnel sans \'elever le r\'egulateur au carr\'e.

\smallskip
On peut changer \`a volont\'e le nombre premier $p \geq 5$ (ici $p=101$), et la borne $N$.

\medskip
\footnotesize
 $\{$$ Q=x^6+9*x^4-4*x^3+27*x^2+36*x+31; p=101; p2=p^2 ; $\par$
  N=10^6;  N0=0; $\par$
  e1= Mod(x,Q) ;    $\par$
  e2= Mod(11/180*x^5 + 1/180*x^4 + 11/18*x^3 - 1/45*x^2 + 403/180*x + 419/180,Q) ;   $\par$
  e3= Mod(13/180*x^5 - 7/180*x^4 + 13/18*x^3 - 38/45*x^2 + 509/180*x + 127/180,Q) ;   $\par$
  e4= Mod(-4/45*x^5 + 1/45*x^4 - 8/9*x^3 + 26/45*x^2 - 137/45*x - 91/45,Q) ; $\par$
  e5= Mod(-1/60*x^5 - 1/60*x^4 - 1/6*x^3 - 4/15*x^2 - 73/60*x - 79/60,Q) ;   $\par$
  e6= Mod(-1/36*x^5 + 1/36*x^4 - 5/18*x^3 + 5/9*x^2 - 65/36*x + 11/36,Q) ; $\par$
  rm= (e1+e2+e3)/3; "(rm=sqrt(-3))";  $\par$
  for(i=1,N, $\par$
  a=random(400)-200; b=random(400)-200; c=random(400)-200;   $\par$
  aa=random(400)-200; bb=random(400)-200; cc=random(400)-200;   $\par$
  E1= (a*e1^5+b*e1^4+c*e1^3+aa*e1^2+bb*e1+cc);   $\par$
  E2= (a*e2^5+b*e2^4+c*e2^3+aa*e2^2+bb*e2+cc);   $\par$
  E3= (a*e3^5+b*e3^4+c*e3^3+aa*e3^2+bb*e3+cc);   $\par$
  E4= (a*e4^5+b*e4^4+c*e4^3+aa*e4^2+bb*e4+cc);   $\par$
  E5= (a*e5^5+b*e5^4+c*e5^3+aa*e5^2+bb*e5+cc);   $\par$
  E6= (a*e6^5+b*e6^4+c*e6^3+aa*e6^2+bb*e6+cc);   $\par$
  Pphi=component((E1^2+E2^2+E3^2 - E4^2-E5^2-E6^2-E1*E2-E2*E3 -E3*E1   $\par$
   \hfill         +E4*E5+E5*E6+E6*E4) / rm, 2); $\par$
  Reg=component(Pphi, 1); $\par$
   if(Mod(Reg, p2)==0, N0=N0+1) ) ;   print (N,"   ",(N0+0.0)/N,"   ",1.0/p2) $$\}$

\normalsize
\medskip
Pour $p=101$ et $10^6$ essais via {\it random}, on obtient une densit\'e de cas $e\geq 2$ \'egale
\`a $0.000101$ pour une probabilit\'e th\'eorique en $0.0000980$.

Pour $p=149$, on obtient $4.60\times 10^{-5}$ pour une  probabilit\'e en $4.50 \times 10^{-5}$.

\smallskip
Le cas des caract\`eres de degr\'e 1 n'offre aucune difficult\'e (sous la condition $f=1$) et nous ferons l'hypoth\`ese
heuristique qu'il en va de m\^eme pour tout groupe et tout caract\`ere dans le cas de $p$-divisibilit\'e minimale.
On peut donc consid\'erer que pour tout $p$ assez grand il n'y a pas d'extra $p$-divisibilit\'es pour ${\rm Reg}_p^G (\eta)$.

\vspace{-0.4cm}
\section{Consid\'erations th\'eoriques directes sur quelques cas particuliers}
Cette Section 5 est redondante par rapport \`a l'\'etude g\'en\'erale des Sections 2 et 3, 
mais elle permet une approche \'el\'ementaire et une \'etude num\'erique approfondie. 

\vspace{-0.2cm}
\subsection{Etude directe utilisant des corps cubiques cycliques}\label{sub19}
Pour intro\-duire le cas Ab\'elien g\'en\'eral (cf.\,\S\ref{sub27}),
nous d\'etaillons le cas de corps cubiques cycliques $K$ 
d\'efinis par les polyn\^omes irr\'eductibles (``simplest cubic fields'' de Shanks \cite{Sh})
$P_t~:= X^3 - t\,X^2 - (t+3)\, X -1, \ \, t \in \Z$,
dont les racines (conjugu\'ees d'une unit\'e $\varepsilon$) sont~:
$x,\ \   x^\sigma = - 1-x^{-1},\ \  x^{\sigma^2} = - (1+x)^{-1}$.

\smallskip
Le discriminant de $P_t$ est $D _t= (t^2+3\,t+9)^2$.
En outre on utilisera $\{x^2, x,  1 \}$ comme $\Z_{(p)}$-base de $Z_{K,(p)}$ et on prendra
les nombres $\eta \in K^\times$ sous la forme $\eta = a\,x^2 + b\,x + c$, $a, b, c \in \Z_{(p)}$.

\smallskip
On peut v\'erifier facilement que sur cette base~:
$$\hbox{ $x^\sigma = x^2 - (t+1)\,x -2\ $ et 
$\ x^{\sigma^2} = -x^2 + t\,x +t+2$.} $$

Donc si $\eta = a\,x^2 + b\,x + c$, les conjugu\'es de $\eta$ sont donn\'es par~:
$$\begin{array}{lll}
&\hspace{-0.42cm} \eta^\sigma &\!\!\!\!\!=\  (at+b)\,x^2 - (a(t^2+t+1) + b(t+1))\,x -a(t-2)-2b+c ,
 \\
&\hspace{-0.42cm}\eta^{\sigma^2} &\!\!\!\!\!=\  (-a(t +1)-b)\,x^2 + (a(t^2+t+1) + bt)\,x
 + a(t^2+3t+4) + b(t+2) + c. 
\end{array} $$

 Soit $F$ le $\Z[G]$-module engendr\'e par $\eta$.
Le calcul de $\eta^{p^{n_p}-1} = 1 + p\,\alpha_p(\eta)$ conduit \`a $\alpha_p(\eta) \equiv
A\,x^2 + B\,x + C \pmod p$, $A, B, C \in \Z_{(p)}$. 

\smallskip
 Nous allons supposer $F$ de $\Z$-rang 3 et
nous int\'eresser au caract\`ere rationnel $\chi \ne 1$.

\smallskip
Posons $\alpha = \alpha_p(\eta)$, $\alpha' = \alpha^\sigma$,  $\alpha'' = \alpha^{\sigma^2}$~;
on rappelle que pour $\chi \ne 1$, le $\chi$-r\'egulateur local est~:
\vspace{-0.1cm}
$$\Delta^\chi_p(\eta)=\N_{\Q(j)/\Q} (\alpha +j^{-1}\alpha'  +j^{-2} \alpha'')
= \alpha^2 +\alpha'{}^2 + \alpha''{}^2 -\alpha\alpha' - \alpha'\alpha''  -  \alpha''\alpha  . $$

\vspace{-0.2cm}
\subsubsection{Exemples num\'eriques}\label{sub21}
Donnons d'abord des exemples num\'eriques (au moyen du programme PARI \'ecrit au\,\S\ref{sub26})~:

\medskip
(i) Cas $p$ inerte dans $\Q(j)$. Pour $t=41$, on a $P_t = X^3 - 41\,X^2 - 44\, X -1$, et son discriminant 
est $D=1813=7^2.37 $.

\smallskip
Prenons $\eta = 3\,x^2 -2\,x + 6$ dont la norme est $353731 = 7^2.7219$. 

\smallskip
On trouve, pour $\Delta^\chi_p(\eta)\equiv 0 \pmod p$, la solution $p=5$,  inerte dans $\Q(j)$. 
Les donn\'ees num\'eriques sont $\alpha \equiv   \alpha' \equiv   \alpha'' \equiv 3 \pmod{p}$
(provenant de la relation $\eta^{5^3-1} \equiv 1 + 5\,.\,3 \pmod {25}$).
L'ensemble des conjug\'es de $\alpha$ est astreint \`a v\'erifier deux relations (conjugu\'ees par $\sigma$) ind\'ependantes modulo~$p$~:
\begin{eqnarray*}
\alpha' - \alpha  &\equiv & 0\pmod p, \\
\alpha'' - \alpha' &\equiv & 0\pmod p, 
\end{eqnarray*}
selon le principe d\'efinissant une probabilit\'e en $\frac{O(1)}{p^2}$.

\medskip
(ii) Cas $p$ d\'ecompos\'e dans $\Q(j)$.
Pour $t=17$, le discriminant est $D=349$. Prenons encore $\eta = 3\,x^2 -2\,x + 6$ dont la norme est $66739$. On obtient une unique solution $p< 10^9$, \`a $\Delta^\chi_p(\eta)\equiv 0 \pmod p$, qui est  $p = 1309963 $ (totalement d\'ecompos\'e dans $\Q(j)$). 
Les donn\'ees num\'eriques sont~:
\begin{eqnarray*}
\alpha &\equiv&  627681 +  294668\,x +  143675\,x^2 \pmod{p}, \\
\alpha' &\equiv & 503146 +  366345\,x +  117217\,x^2 \pmod{p}, \\
\alpha'' &\equiv& 632127 +  648950\,x +  1049071\,x^2\!\! \pmod{p}.
\end{eqnarray*}

On v\'erifie que les 3 mineurs~:
$$
\left \vert\begin{matrix} 
294668 & 143675  \\
366345 & 117217
\end{matrix}\right \vert , \ \ \  \ \ \ 
\left  \vert\begin{matrix} 
 648950 & 1049071 \\
294668 & 143675  
\end{matrix}\right \vert , \ \ \  \ \ \ 
\left  \vert\begin{matrix} 
366345 & 117217  \\
648950 &1049071
\end{matrix}\right \vert , 
$$
sont nuls modulo $p$ et que le rang du syst\`eme est \'egal \`a 2 (une  relation non triviale).
Pour $r = 160549$, d'ordre 3 modulo $p$, et $\theta \div \chi$ d\'efini par $\theta(\sigma) = j \equiv r \pmod {\mathfrak p}$,
on obtient l'unique relation~:
$$\alpha + r^{-1} \alpha' + r^{-2}\alpha'' \equiv 0 \pmod p $$
donnant la nullit\'e modulo $p$ de 
$\Delta^\theta_p(\eta) = \No_ {\mathfrak p}
(\alpha + j^{-1} \alpha' + j^{-2}\alpha'' ). $

\subsubsection{Analyse th\'eorique des exemples pr\'ec\'edents}\label{sub22}
La nullit\'e de $\Delta^\chi_p(\eta)$ modulo $p$ s'\'ecrit~:
$$\Delta^\chi_p(\eta) = \No_{\Q(j)/\Q} \big( \alpha - \alpha'' - j (\alpha'' - \alpha' )\big) \equiv 0 \pmod p. $$

Posons $\alpha \equiv A\,x^2 + B\,x + C \pmod p$.
Toute relation entre $\alpha$, $\alpha'$, $\alpha''$ (modulo~$p$) 
peut se traduire exclusivement sur les coefficients  $A$, $B$, $C$~: en effet,
par conjugaison de $\alpha = A\,x^2 + B\,x + C$, on obtient 
(en posant $x' = x^\sigma$ et $x'' = x^{\sigma^2}$)~:
$$\alpha' = A\,x'{}^2 + B\,x' + C =: A'\,x^2 + B'\,x + C' , $$ 
$$\ \,\alpha'' = A\,x''{}^2 + B\,x'' + C =: A''\,x^2 + B''\,x + C'' , $$ 

ce qui conduit au r\'esultat via les formules explicites $x' = x^2 - (t+1)\,x -2$
et $x'' = -x^2 + t x + t+2$,  qui donnent $A', B', C'$, $A'', B'', C''$ par des
transformations lin\'eaires ind\'ependantes de $p$, \`a savoir~:
$$\begin{array}{lllll}
&\hspace{-0.62cm} A'  = At\!+\!B, &\hspace{-0.2cm}  B' = \!-\!A(t^2\!+\!t\!+\!1) \!-\!B(t\!+\!1),  &\hspace{-0.2cm} C' = \!-\!A(t\!-\!2) \!-\!2B \!+\! C \\
&\hspace{-0.62cm} A'' \! = \!-\!A(t\!+\!1)\!-\! B, &\hspace{-0.3cm}  B''\! = A(t^2\!+\!t\!+\!1) \!+\! Bt ,  &\hspace{-0.32cm} C''\! = A(t^2 \!+\! 3t \!+\!4) \!+\!B(t\!+\!2) \!+\! C. 
\end{array}$$

On en d\'eduit que l'ensemble des relations modulo $p$ entre $A$, $B$, $C$ 
conduit \`a la probabilit\'e cherch\'ee (cf.\,\S\S\ref{sub14}, \ref{HP}).

Consid\'erons le sch\'ema suivant et supposons  $\Delta^\chi_p(\eta) \equiv 0 \pmod p$~:
\unitlength=0.5cm
$$\vbox{\hbox{\hspace{-2cm}  \begin{picture}(11.5,5.4)
\put(4.0,4.50){\line(1,0){2.5}}
\put(4.0,0.50){\line(1,0){2.5}}
\put(3.50,1.1){\line(0,1){3.0}}
\put(7.50,1.1){\line(0,1){3.0}}

\put(6.85,0.4){$\Q(j)\ \ {\mathfrak p}$}
\put(6.85,4.4){$K(j)\ \ {\mathfrak P}$}
\put(3.3,4.4){$K$}
\put(-0.2,4.4){$\alpha,\,\alpha', \alpha'' \!\in $}
\put(3.3,0.40){$\Q$}

\put(8.0,2.4){$\langle \sigma\rangle$}
\put(5.0,5.0){$\langle s\rangle $}
\end{picture}   }} $$
\unitlength=1.0cm

Puisque $\Delta^\chi_p(\eta) = \No_{\Q(j) /\Q} (\alpha+j^{-1} \alpha' + j^{-2} \alpha'') =
\No_{\Q(j) /\Q} (\alpha - \alpha - j (\alpha'- \alpha'') ) \equiv 0 \pmod p$, 
n\'ecessairement $\alpha - \alpha' - j (\alpha'- \alpha'') \equiv 0 
\pmod {\mathfrak P}$ dans $K(j)$, pour au moins un id\'eal premier $ {\mathfrak P}$ au-dessus de $p$ dans $K(j)$.
Soit ${\mathfrak p}$ l'id\'eal premier de $\Q(j)$ au-dessous de ${\mathfrak P}$.

\smallskip
On obtient, par conjugaison par $\sigma$ (cf. Corollaire \ref{coro3},\,\S\ref{calp})~:
$$\alpha' - \alpha'' - j (\alpha''- \alpha)  \equiv
j \,(\alpha - \alpha' - j (\alpha'- \alpha'')) \equiv 0 \!\!\!\!\pmod {\mathfrak P^{\sigma}}, $$ 
et de m\^eme avec ${\mathfrak P^{\sigma^2}}$~; 
par cons\'equent $\alpha - \alpha' - j (\alpha'- \alpha'') $  
appartient \`a tous les id\'eaux premiers au-dessus de ${\mathfrak p}$ dans $K(j)$, d'o\`u
$\alpha - \alpha' - j (\alpha'- \alpha'') \equiv 0 \pmod {\mathfrak p}$ (\'etendu \`a $K(j)$). 
Distinguons selon la d\'ecomposition de $p$ dans $\Q(j) /\Q$. 

\smallskip
(i) Cas $p$ inerte dans $\Q(j) /\Q$. L'unique id\'eal ${\mathfrak p}$ est l'\'etendu de $(p)$ dans $\Q(j)$.
 N\'ecessairement $\alpha - \alpha' - j (\alpha'- \alpha'')  \equiv 0 \pmod p$ dans $K(j)$~; 
comme $\{1, j\}$ est une $K$-base de $K(j)$ et que 
$\alpha - \alpha' - j (\alpha'- \alpha'') \in Z_K$,  il vient $\alpha \equiv \alpha'\equiv \alpha''  \pmod p$ 
(en un sens, $\alpha$ est rationnel modulo $p$), ce qui conduit \`a~:
$$A \equiv 0 \pmod p, \ \ B \equiv 0 \pmod p$$
(cas de l'exemple (i) du\,\S\ref{sub21}). D'o\`u une probabilit\'e en $\Frac{O(1)}{p^2}$. Le coefficient $C$ n'est li\'e par aucune condition.

\smallskip
(ii) Cas $p$ d\'ecompos\'e dans $\Q(j) /\Q$. Soient ${\mathfrak p}$ et ${\mathfrak p}^s$ les deux id\'eaux 
premiers de $\Q(j)$ au-dessus de $p$. D'apr\`es ce qui pr\'ec\`ede, on a $\alpha - \alpha' - j (\alpha'- \alpha'') 
\equiv 0 \pmod {\mathfrak p}$ (\'etendu \`a $K(j)$).
Soit $r$ un rationnel tel que $j \equiv r \pmod {\mathfrak p}$ ; alors 
$\alpha - \alpha' - j (\alpha'- \alpha'')  \equiv \alpha - \alpha' - r (\alpha'- \alpha'') 
 \equiv 0 \pmod {\mathfrak p}$ (\'etendu \`a $K(j)$), et comme 
$\alpha - \alpha' - r (\alpha'- \alpha'')  \in K$, cette congruence est vraie par conjugaison par $s$, d'o\`u
$\alpha - \alpha' - r (\alpha'- \alpha'')  \equiv 0 \pmod p$. 
Ceci conduit dans $K$ \`a~:
$$\Delta^\theta_p(\eta)  = \No_ {\mathfrak p}(\alpha + j^{-1} \alpha' + j^{-2}  \alpha'')
\equiv \alpha + r^{-1} \alpha' + r^{-2}  \alpha''  \equiv 0 \pmod {\mathfrak p}. $$

\vspace{-0.1cm}

On en d\'eduit (une seule relation et deux choix pour $r$) une probabilit\'e en $\Frac{2}{p}$. Comme dans le cas inerte, 
on peut voir comment ces relations se transmettent \`a $A, B, C$~; il vient, puisque $r^{-1} \equiv r^2 \pmod p$~:
\begin{eqnarray*}
A +  r^2  A' +r A'' &\equiv &0 \pmod p, \\
B +  r^2   B' +r B'' &\equiv& 0 \pmod p, \\
C +  r^2  C' +r C''&\equiv&0 \pmod p.
\end{eqnarray*}

\vspace{-0.1cm}

En utilisant les relations pr\'ec\'edentes~:
$$\begin{array}{lllll}
&\hspace{-0.62cm} A'  = At\!+\!B, &\hspace{-0.2cm}  B' = \!-\!A(t^2\!+\!t\!+\!1) \!-\!B(t\!+\!1),  &\hspace{-0.2cm} C' = \!-\!A(t\!-\!2) \!-\!2B \!+\! C \\
&\hspace{-0.62cm} A'' \! = \!-\!A(t\!+\!1)\!-\! B, &\hspace{-0.3cm}  B''\! = A(t^2\!+\!t\!+\!1) \!+\! Bt ,  &\hspace{-0.32cm} C''\! = A(t^2 \!+\! 3t \!+\!4) \!+\!B(t\!+\!2) \!+\! C ,
\end{array}$$
on obtient le syst\`eme suivant~:
\begin{eqnarray*}
A (t(2r^2+1)+r^2 +2 )+ B(2 r^2+1) &\equiv& 0 \pmod p \\
A(r-r^2) (t^2+t+1) + B (r t - r^2(t+1)+1)& \equiv &0 \pmod p \\
A(r (t^2+3t+4)- r^2(t-2) ) + B ((t+2)-2r^2 ) &\equiv&0 \pmod p.
\end{eqnarray*}

Ce syst\`eme en $A$ et $B$ ($C$ ne figure pas) est de rang 1 et est alors \'equivalent \`a la seule relation 
$B \equiv  - (t + \frac{r-1}{2 +1})\,A  \equiv  (r - t )\,A \pmod p$. 
 En r\'esum\'e on obtient $B - (r^2 - t )\,A \equiv  0 \ \pmod p$.
Ceci est illustr\'e par l'exemple num\'erique (ii) du\,\S\ref{sub21}. On a obtenu une relation
entre $A$ et $B$ fonction de $r$ (probabilit\'e en $\frac{2}{p}$).

\subsubsection{Cas o\`u $F$ est de $\Z$-rang 2}\label{sub23}
Lorsque $\eta$ est un \'el\'ements non trivial de norme $\pm 1$, le $\Z$-rang de $F$ est \'egal \`a 2 
et on a la nullit\'e triviale modulo $p$ de $\Delta^1_p(\eta) =\alpha+\alpha'+\alpha''$
(provenant de $\No_{K/\Q}(\eta) =\pm1$).
Cette relation peut \^etre utilis\'ee pour simplifier les calculs, mais elle ne doit induire aucune
modification de la probabilit\'e de nullit\'e modulo $p$ de $\Delta^\chi_p(\eta)$ pour $\chi \ne 1$.
Le programme donn\'e ci-dessous permet de s'assurer de ce fait.

\smallskip
On est donc conduit au $\chi$-d\'eterminant plus simple~:

\smallskip\smallskip
\centerline{ $\Delta^\chi_p(\eta) = \No_{\Q(j)/\Q} (\alpha + j^{-1}\alpha' + j^{-2}\alpha'') =
\No_{\Q(j)/\Q} (\alpha (1-j) - \alpha'(j- j^2) ), $ }

\smallskip\smallskip
qui, au facteur 3 pr\`es, s'\'ecrit $\Delta^\chi_p(\eta) = \No_{\Q(j)/\Q} (\alpha - j\, \alpha')$.

\smallskip
Donnons un seul exemple dans le cas inerte, particuli\`erement rare~: pour 
$t=33$ ($D=1197 = 9. 7. 19$) et $p = 17$ on trouve que $\alpha$ et ses conjugu\'es sont nuls modulo~$p$
(en effet les relations $\alpha \equiv \alpha' \equiv \alpha'' \pmod p$ du cas inerte (i) vu pour le $\Z$-rang 3
deviennent ici $\alpha \equiv 0 \pmod p$  puisque $\alpha+\alpha'+\alpha'' \equiv 0 \pmod p$).

\smallskip
\subsubsection{Programme PARI g\'en\'eral}\label{sub26}
 Ce programme, valable pour les $\Z$-rangs 2 et 3, peut s'utiliser en modifiant les valeurs de $t$, $aa, bb, cc$~;
il recherche les solutions $p< 10^8$ et donne les conjugu\'es de $\alpha$~:
\footnotesize 

\medskip
$\{$$t=13; Q=x^3-t*x^2-(t+3)*x-1; D=t^2+3*t+9; $\par
$aa=0; bb=1; cc=0; e=Mod(aa*x^2+bb*x+cc, Q); N=norm(e); $\par
$print("t = ",t," \ \   ","Discriminant = ",D," \ \    ",aa,"   \ \   ",bb,"  \ \    ",cc,"    \ \     ", N);$\par
$for (k= 2, 5*10^7, p=2*k+1; if (isprime(p)==1 \& Mod(D*N,p)!=0, p2=p^2;$\par
$a=Mod(aa,p2); b=Mod(bb,p2); c=Mod(cc,p2); $\par
$P=x^3-t*x^2-(t+3)*x-Mod(1,p); P2=x^3-t*x^2-(t+3)*x-Mod(1,p2) ;$\par
$y=Mod(a*x^2+b*x+c,P2); u= y^{(p+1)}; v=u^p*y ; z=v^{(p-1)} -1; $\par
$U=component(z,2); $\par
$u1=component(U,1); if(u1==0, u1=Mod(0,p2));$\par
$u2=component(U,2); if(u2==0, u2=Mod(0,p2)) ;$\par
$u3=component(U,3); if(u3==0, u3=Mod(0,p2));$\par
$x1=component(u1,2)/p; x2=component(u2,2)/p; x3=component(u3,2)/p; $\par
$X1=Mod(x1,p); X2=Mod(x2,p); X3=Mod(x3,p); $\par
$SX1=X1-2*X2-X3*(t-2); $\par
$SX2=-X3*(t^2+t+1)-X2*(t+1); $\par
$SX3=X3*t+X2;$\par
$TX1=X3*(t^2+3*t+4)+X2*(t+2)+X1; $\par
$TX2=X3*(t^2+t+1)+X2*t;$\par
$TX3=-X3*(t+1)-X2;$\par
$F=Mod(X1+X2*x+X3*x^2, P); $\par
$SF=Mod(SX1+SX2*x+SX3*x^2, P); $\par
$TF=Mod(TX1+TX2*x+TX3*x^2, P);$\par
$RP=F^2+SF^2+TF^2-F*SF-SF*TF-TF*F; rp=component(RP,2); $\par
$if (rp==0,  print("p = ",p,"            ","p mod 3 =  ", component(Mod(p,3),2) ); $\par
$print("coefficients\,de\, alpha,\, alpha',\, alpha'' \,:\,  X1+X2*x+X3*x^2 :");$\par
$print("coefficients \,de  \, alpha " , \,X1," \ \   ", X2," \ \   ", X3); $\par
$print("coefficients \,de \,s.alpha  ",SX1," \ \   ", SX2,"  \ \  ", SX3); $\par
$print("coefficients\, de\, t.alpha  ",TX1," \ \   ", TX2,"  \ \  ", TX3)  ) ) )$$\}$
\normalsize

\vspace{-0.1cm}
\subsection{Etude directe du cas Ab\'elien  g\'en\'eral}\label{sub27}
On peut toujours se ramener au cas o\`u $G$ est cyclique d'ordre $n > 2$ (voir \S\,\ref{sub7} pour les cas $n \leq 2$).

\subsubsection{R\'esultat g\'en\'eral}
Soit  $\chi$ le caract\`ere rationnel de $G$ d'ordre $n$~; soit
$\zeta$ une racine primitive $n$-i\`eme de l'unit\'e. On a $C_\chi = \Q(\zeta)$ et le $\chi$-r\'egulateur local~:

\centerline {$\Delta^\chi_p(\eta) = \No_{\Q(\zeta)/\Q} \Big( \sm_{\nu\in G} \varphi (\nu) \alpha^{\nu^{\!-1}}\Big)$, }

 pour $\varphi \div \chi$ fix\'e.
La condition $\Delta^\chi_p(\eta) \equiv 0 \pmod p$ implique par des raisonnements connus 
(cf.  Corollaire \ref{coro3},\,\S\ref{calp})~:
$$\sm_{\nu\in G} \varphi (\nu) \alpha^{\nu^{\!-1}} \equiv 0 \!\!\pmod {\mathfrak p}\ \, \hbox{(id\'eal premier de 
$\Q(\zeta)$ \'etendu \`a $K(\zeta)$)}. $$

Soient $L$ le corps de d\'ecomposition de $p$ dans $\Q(\zeta)/\Q$ et $D = {\rm Gal}(\Q(\zeta)/L)$~; on d\'esigne par 
$f$ le degr\'e r\'esiduel en $p$. Alors $\{\zeta, \ldots, \zeta^f\}$ est une $L$-base de~$\Q(\zeta)$
(et une $KL$-base de $K(\zeta)$). Posons 
$\varphi (\nu)  = \sum_{i=1}^{f} a_i(\nu)\, \zeta^i$,  $a_i(\nu) \in Z_{L,(p)}$ pour tout $\nu \in G$.
 Il vient $\sum_{\nu\in G} \varphi (\nu)\, \alpha^{\nu^{\!-1}} = \sum_{i=1}^{f} \zeta^i \sum_{\nu\in G} a_i(\nu)\, \alpha^{\nu^{\!-1}}
\equiv 0 \pmod {\mathfrak p}$.

\smallskip
Comme il y a inertie dans $K(\zeta)/KL$, il en r\'esulte $\sum_{\nu\in G} a_i(\nu)\,\alpha^{\nu^{\!-1}}\equiv 0 
\pmod {\mathfrak p}$, pour $i = 1,\ldots, f$, en voyant ici ${\mathfrak p}$ dans $L$ et \'etendu \`a $KL$.

\smallskip
Il existe alors des $p$-entiers rationnels $r_i(\nu)$ tels que $a_i(\nu) \equiv r_i(\nu) \pmod {\mathfrak p}$
dans $L$,  ce qui conduit \`a $\sum_{\nu\in G} r_i(\nu) \,\alpha^{\nu^{\!-1}}\equiv 0 \pmod {\mathfrak p}$, $i = 1, \ldots, f $~;
d'o\`u finalement $\sum_{\nu\in G}r_i(\nu)\, \alpha^{\nu^{\!-1}} \equiv 0 \pmod p, \ \  i = 1, \ldots, f$.

\smallskip
 La matrice  $\big (r_i(\nu) \big )_{\nu \in G,\,1\leq i \leq f}$ est de $\F_p$-rang $f$~; en effet,
soit $H = \{\nu_1, \ldots , \nu_f\}$ un sous-ensemble de $G$ tel que $\varphi (\nu_j) = \zeta^j$, $j= 1, \ldots, f$~; 
alors la sous-matrice $\big(a_i (\nu_j) \big)_{i,j}$ et donc la sous-matrice $\big(r_i (\nu_j) \big)_{i,j}$,
  est la matrice unit\'e $I_f$ d'ordre~$f$.

\medskip
Il en r\'esulte un syst\`eme de $f$ relations ind\'ependantes d\'efinissant 
${\mathcal L}^\theta$ pour un caract\`ere $p$-adique $\theta \div \chi$ 
(cf. D\'efinitions \ref{defi01}\,(ii),\,\S\ref{calp} et \ref{defi11}\,(ii),\,\S\ref{sub155})~:
$$\sm_{\nu\in G} r_i(\nu) \alpha^{\nu^{\!-1}} \equiv 0  \pmod p,   \ \ i= 1, \ldots, f ,$$

correspondant \`a une probabilit\'e en $\frac{O(1)}{p^f}$. Chacun des $h=[L : \Q]$ r\'egulateurs 
$\Delta^\theta_p(\eta) = \No_{\mathfrak p}(\sum_{\nu\in G}\varphi (\nu) \,\alpha^{\nu^{\!-1}})$
(cf. Remarque \ref{rema11} (i))
a une probabilit\'e de nullit\'e modulo $p$ en $\frac{O(1)}{p^f}$. On peut donc se limiter au ca  $f=1$
qui conduit \`a la probabilit\'e la plus grande (en $\frac{O(1)}{p}$ car $f=\delta=1$).

\subsubsection{Exemple du sous-corps r\'eel maximal de $\Q(\mu_{11})$}\label{sub28}
Les programmes PARI suivants traitent le cas du sous-corps r\'eel maximal $K$ de $\Q(\mu_{11})$. 

\medskip
a) {Recherche des solutions $p$ telles que $\Delta_p^\chi(\eta)\equiv 0 \pmod p$}.\label{sub281}
On peut changer \`a volont\'e les coefficients $aa, bb, cc, dd, ee$ sur la base $\{x^4, x^3, x^2,  x,  1\}$
des puissances de $x = \zeta_{11}+ \zeta_{11}^{-1}$.
Attention, la liste des conjugu\'es de $x$~:
$$s1=x,  s2=x^2-2,  s3=x^4 - 4*x^2 + 2,  s4= x^3-3*x,  s5=-x^4 - x^3 + 3*x^2 + 2*x - 1, $$

qui correspond \`a  l'ordre naturel $1$, $\sigma$, $\sigma^2$, $\sigma^3$, $\sigma^4$, est donn\'ee dans le d\'esordre par 
la proc\'edure {\it nfgaloisconj} de PARI.
\footnotesize 

\medskip
$\{$$Q=x^5+x^4-4*x^3-3*x^2+3*x+1; $\par
$s1=x; s2=x^2-2; s3=x^4 - 4*x^2 + 2; s4= x^3-3*x; s5=-x^4 - x^3 + 3*x^2 + 2*x - 1;$\par
 $aa=2; bb=1; cc=2; dd=0; ee=-3; Eta=Mod(aa*x^4+bb*x^3+cc*x^2+dd*x+ee,Q); $\par
 $N=norm(Eta); print("norme(eta) = ",N); print("   "); $\par
 $for (k=1, 5*10^6, p=2*k+1; if (isprime(p)==1 \& Mod(11*N,p)!=0, p2=p^2; $\par
 $a=Mod(aa,p2); b=Mod(bb,p2); c=Mod(cc,p2); d=Mod(dd,p2);  e=Mod(ee,p2);  $\par
$P=x^5+x^4-4*x^3-3*x^2+3*x+Mod(1,p); P2=x^5+x^4-4*x^3-3*x^2+3*x+Mod(1,p2); $\par
 $y=Mod(a*x^4+b*x^3+c*x^2+d*x+e,P2);  $\par
 $u=y^{(p+1)} ; v=u^p*y; w=v^p *y; s=w^p*y; t=s^{(p-1)} -1; $\par
 $U=component(t,2); $\par
 $u1=component(U,1); if(u1==0, u1=Mod(0,p2)); $\par
 $u2=component(U,2); if(u2==0, u2=Mod(0,p2)); $\par
 $u3=component(U,3); if(u3==0, u3=Mod(0,p2)); $\par
 $u4=component(U,4); if(u4==0, u4=Mod(0,p2)); $\par
 $u5=component(U,5); if(u5==0, u5=Mod(0,p2)); $\par
 $x1=component(u1,2)/p; x2=component(u2,2)/p; x3=component(u3,2)/p; $\par
\hfill $x4=component(u4,2)/p; x5=component(u5,2)/p; $\par
 $X1=Mod(x1,p); X2=Mod(x2,p); X3=Mod(x3,p); X4=Mod(x4,p); X5=Mod(x5,p);$\par
 $e1=  Mod(X5*s1^4+X4*s1^3+X3*s1^2+X2*s1+X1,P);$\par
 $e2=  Mod(X5*s2^4+X4*s2^3+X3*s2^2+X2*s2+X1,P);$\par
 $e3=  Mod(X5*s3^4+X4*s3^3+X3*s3^2 +X2*s3+X1,P);$\par
 $e4=  Mod(X5*s4^4+X4*s4^3 +X3*s4^2+X2*s4+X1,P);$\par
 $e5=  Mod(X5*s5^4+X4*s5^3 +X3*s5^2+X2*s5+X1,P);$\par
$A=e5*e4+e4*e3+e3*e2+e2*e1+e1*e5; $\par
$B=e5*e3+e4*e2+e3*e1+e2*e5+e1*e4; $\par
$C=e5^2+e4^2+e3^2+e2^2+e1^2;$\par
$RP=C*(C-A-B) - (A+B)^2 + 5*A*B; rp=component(RP,2); if(rp==0, print(p); $\par
 $print(component(e1,2)); print(component(e2,2)); print(component(e3,2)); $\par
 $print(component(e4,2)); print(component(e5,2)); g=Mod(p,5); print("g =  ", g)  )  )) $$\}$
\normalsize

\medskip
(i)  Pour $aa=-2; bb=1; cc=0; dd=0; ee=-3$,
on a $\No_{K/\Q}(\eta) = -60589$, et les solutions pour $p< 10^7$ sont $p=31, 101, 39451$,
 tous d\'ecompos\'es dans $\Q(\zeta_5)$.     
Consid\'erons les donn\'ees num\'eriques pour $p=31$~:
$$\begin{array}{llllllllllllll}
\alpha &\equiv&  25  x^4 & \!\!\!+\!\!\! &   10  x^3 & \!\!\!+\!\!\! &  7  x^2 & \!\!\!+\!\!\! &   21  x & \!\!\!+\!\!\! & 29 &\pmod p \\
\alpha^{\sigma}  &\equiv&  4  x^4 & \!\!\!+\!\!\! &  15  x^3 & \!\!\!+\!\!\! &  25  x^2 & \!\!\!+\!\!\! &  7 x & \!\!\!+\!\!\! &  16 &\pmod p\\
\alpha^{\sigma^2}  &\equiv&  26  x^4 & \!\!\!+\!\!\! &  20  x^3 & \!\!\!+\!\!\! &  26 x^2 & \!\!\!+\!\!\! &  18  x & \!\!\!+\!\!\! &  22 & \pmod p\\
 \alpha^{\sigma^3}  &\equiv&  17  x^4 & \!\!\!+\!\!\! & 6 x^3 & \!\!\!+\!\!\! & 21  x^2 & \!\!\!+\!\!\! & 24  x & \!\!\!+\!\!\! &  4 &\pmod p\\
 \alpha^{\sigma^4} &\equiv& 21  x^4 & \!\!\!+\!\!\! &  11  x^3 & \!\!\!+\!\!\! &  14  x^2 & \!\!\!+\!\!\! &  23  x & \!\!\!+\!\!\! & 19 &\pmod p
\end{array}$$

Pour $r = 4$, qui est tel que $\theta(\sigma) \equiv r \pmod p$, on a imm\'ediatement,
comme pr\'evu~:
$\Delta^\theta_p(\eta) =\alpha + r^{-1}\alpha^{\sigma} + r^{-2}\alpha^{\sigma^2} +r^{-3}\alpha^{\sigma^3} +  
r^{-4}\alpha^{\sigma^4} \equiv 0 \pmod p$
 identiquement sur la base $\{x^4, x^3, x^2, x, 1\}$.

\smallskip
(ii) Pour $aa=10; bb=-7; cc=0; dd=1; ee=-2$,
on a $\No_{K/\Q}(\eta) = 303623$, et on trouve l'unique solution $p=7$,
premier cas totalement inerte dans  $\Q(\zeta_5)$.  Le programme donne tous les conjugu\'es de $\alpha$ nuls modulo $p$
(d'o\`u en plus $\Delta^1_p(\eta) \equiv 0 \pmod p $).

\smallskip
Il est clair que le cas inerte dans  $\Q(\zeta_5)/\Q$ est tr\`es rare. D'ailleurs, $p$ est petit pour compenser
 une probabilit\'e en $\Frac{O(1)}{p^4}$.

\smallskip
(iii) Pour $aa=10; bb=-7; cc=-3; dd=1; ee=-2$,
on a $\No_{K/\Q}(\eta) = 1358171$, et on trouve la solution $p=79$
(deux caract\`eres $p$-adiques $\theta$ de degr\'e r\'esiduel  $f=2$~; $p$ d\'ecompos\'e dans $L = \Q(\sqrt 5)$).

\smallskip
La r\'esolvante $\alpha + \zeta_5 \alpha^\sigma + \zeta_5^2\alpha^{\sigma^2} + \zeta_5^3\alpha^{\sigma^3}
 + \zeta_5^4\alpha^{\sigma^4}$ (qui correspond \`a $\Delta^\varphi_p(\eta)$ pour $\varphi(\sigma) = \zeta_5^{-1}$) se d\'ecompose de la fa\c con suivante sur la base relative $\{1, \zeta_5\}$~:

\smallskip
On a la relation $\zeta_5^2 -\zeta_5 \frac{\sqrt 5 - 1}{2} + 1 =0$ qui d\'efinit le polyn\^ome irr\'eductible de
$\zeta_5$ sur $\Q(\sqrt 5)$. On obtient alors $\zeta_5^3 =  -\zeta_5 \frac{\sqrt 5 - 1}{2} + \frac{1-\sqrt 5}{2}$,
$\zeta_5^4 =  -\zeta_5 + \frac{\sqrt 5 - 1}{2}$ et le syst\`eme de relations dans $K(\zeta_5)$
exprimant $\Delta^\varphi_p(\eta) \equiv 0 \pmod  {\mathfrak p}$~:
\begin{eqnarray*}
\alpha - \alpha^{\sigma^2} + \Frac{\sqrt 5 - 1}{2} (\alpha^{\sigma^4}-\alpha^{\sigma^3}) 
&\equiv& 0 \pmod {\mathfrak p} \\
\alpha^{\sigma} - \alpha^{\sigma^4} + \Frac{\sqrt 5 - 1}{2} (\alpha^{\sigma^2}-\alpha^{\sigma^3}) 
&\equiv& 0 \pmod {\mathfrak p}.
\end{eqnarray*}

Ensuite, l'id\'eal ${\mathfrak p}$ est par exemple d\'efini par la congruence $\sqrt 5 \equiv 20 \pmod {\mathfrak p}$, 
d'o\`u $\frac{\sqrt 5 - 1}{2} \equiv 49 \pmod {\mathfrak p}$ ce qui d\'efinit les coefficients $r_i(\nu)$, $i = 1, 2$ et le couple
$(\theta, {\mathfrak p})$. 

\smallskip
On a donc obtenu deux relations lin\'eaires \`a coefficients rationnels ind\'ependantes~:
\begin{eqnarray*}
\alpha - \alpha^{\sigma^2} + 49\, (\alpha^{\sigma^4}-\alpha^{\sigma^3}) 
&\equiv& 0 \pmod p \\
\alpha^{\sigma^4}  - \alpha^{\sigma}  + 49\, (\alpha^{\sigma^3} - \alpha^{\sigma^2}) 
&\equiv& 0 \pmod p.
\end{eqnarray*}

Les donn\'ees num\'eriques pour $\alpha$ et ses conjugu\'es 
sont respectivement~:
$$\begin{array}{llllllllllllll}
\alpha &\equiv&  37  x^4 & \!\!\!+\!\!\! &   13  x^3 & \!\!\!+\!\!\! &   19  x^2 & \!\!\!+\!\!\! &   3  x & \!\!\!+\!\!\! &   10  &\pmod p \\
\alpha^{\sigma}  &\equiv&  75  x^4 & \!\!\!+\!\!\! &   24  x^3 & \!\!\!+\!\!\! &   45  x^2 & \!\!\!+\!\!\! &   73  x & \!\!\!+\!\!\! &   33  &\pmod p\\
\alpha^{\sigma^2}  &\equiv&  5  x^4 & \!\!\!+\!\!\! &   51  x^3 & \!\!\!+\!\!\! &   22  x^2 & \!\!\!+\!\!\! &   60  x & \!\!\!+\!\!\! &   1  &\pmod p\\
 \alpha^{\sigma^3}  &\equiv&  70  x^4 & \!\!\!+\!\!\! &   33  x^3 & \!\!\!+\!\!\! &   40  x^2 & \!\!\!+\!\!\! &   8  x & \!\!\!+\!\!\! &   77 & \pmod p\\
 \alpha^{\sigma^4} &\equiv& 50  x^4 & \!\!\!+\!\!\! &   37  x^3 & \!\!\!+\!\!\! &   32  x^2 & \!\!\!+\!\!\! &   14  x & \!\!\!+\!\!\! &    22 & \pmod p
\end{array}$$

qui v\'erifient le syst\`eme de deux congruences ci-dessus.

\smallskip
On a les deux relations ind\'ependantes, d\'efinissant ${\mathcal L}^\theta \simeq V_\theta$ de $\F_p$-dimension 2,
$\ 1 - {\sigma^2} + 49\, ({\sigma^4}-{\sigma^3})\ \ \& \ \ {\sigma^4} - {\sigma} + 49\, ({\sigma^3}-{\sigma^2})$,
la seconde \'etant conjugu\'ee par $\sigma^4$ de la premi\`ere.
D'o\`u  une probabilit\'e en $\frac{2}{p^2}$ (deux choix pour la congruence $\sqrt 5 \equiv 20 \pmod {\mathfrak p}$).

\medskip
b) {Calcul de la densit\'e des $\Delta_p^\theta (\eta) \equiv 0 \pmod p$}.\label{sub282}
Le programme ci-dessous reprend le cas pr\'ec\'edent du sous-corps r\'eel maximal de $\Q(\mu_{11})$ et s'int\'eresse aux diff\'erents
degr\'es r\'esiduels possibles pour v\'erifier que la probabilit\'e pour $\Delta_p^\theta (\eta) \equiv 0 \pmod p$
est bien en $\frac{O(1)}{p^f}$. 

On affiche les probabilit\'es th\'eoriques selon le cas ($f =1,2,4$).
\footnotesize 

\bigskip
  $\{$$Q=x^5+x^4-4*x^3-3*x^2+3*x+1;   p=19; p2=p^2;  N0=0; N1=0; $\par
 $ s1=x; s2=x^2-2; s3=x^4 - 4*x^2 + 2; s4= x^3-3*x; s5=-x^4 - x^3 + 3*x^2 + 2*x - 1; $\par
 $ for(i=1, 5*10^5,$\par
 $ aa=random(p2); bb=random(p2); cc=random(p2); dd=random(p2); ee=random(p2);$\par
  $Eta=Mod(aa*x^4+bb*x^3+cc*x^2+dd*x+ee,Q);  $\par
  $N=norm(Eta);  if(Mod(N,p)!=0, N0=N0+1; $\par
  $a=Mod(aa,p2); b=Mod(bb,p2); c=Mod(cc,p2); d=Mod(dd,p2); e=Mod(ee,p2); $\par
  $P=x^5+x^4-4*x^3-3*x^2+3*x+Mod(1,p); P2=x^5+x^4-4*x^3-3*x^2+3*x+Mod(1,p2);  $\par
  $y=Mod(a*x^4+b*x^3+c*x^2+d*x+e,P2);   $\par
  $u=y^{(p+1)}  ; v=u^p*y; w=v^p *y; s=w^p*y; t=s^{(p-1)}  -1;  $\par
  $U=component(t,2);  $\par
  $u1=component(U,1); if(u1==0, u1=Mod(0,p2));  $\par
  $u2=component(U,2); if(u2==0, u2=Mod(0,p2));  $\par
  $u3=component(U,3); if(u3==0, u3=Mod(0,p2));  $\par
  $u4=component(U,4); if(u4==0, u4=Mod(0,p2));  $\par
  $u5=component(U,5); if(u5==0, u5=Mod(0,p2));  $\par
  $x1=component(u1,2)/p; x2=component(u2,2)/p; x3=component(u3,2)/p; $\par
 \hfill  $ x4=component(u4,2)/p; x5=component(u5,2)/p;  $\par
  $X1=Mod(x1,p); X2=Mod(x2,p); X3=Mod(x3,p);  X4=Mod(x4,p); X5=Mod(x5,p); $\par
  $e1=  Mod(X5*s1^4+X4*s1^3+X3*s1^2+X2*s1+X1,P); $\par
  $e2=  Mod(X5*s2^4+X4*s2^3+X3*s2^2+X2*s2+X1,P); $\par
  $e3=  Mod(X5*s3^4+X4*s3^3+X3*s3^2 +X2*s3+X1,P); $\par
  $e4=  Mod(X5*s4^4+X4*s4^3 +X3*s4^2+X2*s4+X1,P); $\par
  $e5=  Mod(X5*s5^4+X4*s5^3 +X3*s5^2+X2*s5+X1,P); $\par
 $A=e5*e4+e4*e3+e3*e2+e2*e1+e1*e5;  $\par
 $B=e5*e3+e4*e2+e3*e1+e2*e5+e1*e4;  $\par
 $C=e5^2+e4^2+e3^2+e2^2+e1^2; $\par
 $RP=C*(C-A-B) - (A+B)^2 + 5*A*B; rp=component(RP,2);  if(rp==0, N1=N1+1) ) ) ; $\par
  $print(p,"     ",N0,"   ",N1,"    ",(N1+0.0)/N0) ; $\par
 $print("4/p-6/p^2+4/p^3 -1/p^4 = ",  4./p-6./p^2+4./p^3 -1./p^4  );$\par
 $print("2/p^2 -1/p^4 = ", 2./p^2 -1./p^4  );$\par
 $print("1/p^4 = ", 1./p^4  ) $$\}$
\normalsize

\medskip
Pour $p=31$, le degr\'e r\'esiduel est \'egal \`a 1 et on trouve les valeurs suivantes~:
$N_1 = 61505$, $\frac{N_1}{N_0} = 0.1230$ pour 
$\frac{4}{p} - \frac{6}{p^2} + \frac{4}{p^3} - \frac{1}{p^4}= 0.12292$. 

\smallskip
Pour $p=19$, le degr\'e r\'esiduel est \'egal \`a 2 et on trouve les valeurs suivantes~:
$N_1 = 2756$, $\frac{N_1}{N_0} = 0.005512$ pour $\frac{2}{p^2}-  \frac{1}{p^4}= 0.00553$.

\smallskip
Pour $p=13$, le degr\'e r\'esiduel est \'egal \`a 4 et on trouve les valeurs suivantes~:
$N_1 = 17$, $\frac{N_1}{N_0} = 3.40\times 10^{-5}$ pour $\frac{1}{p^4} = 3.50 \times 10^{-5}$.

\subsection{Etude directe du cas du groupe $D_6$} \label{sub29}
On d\'esigne par $k = \Q(\sqrt m)$ le sous-corps quadratique de $K$ et par $\chi_1$ et $\chi_2$
les deux caract\`eres rationnels (et $p$-adiques) irr\'eductibles non triviaux de $D_6$. On pose $\alpha =\alpha_p(\eta)$.

\subsubsection{Rappels}
On \'etudie les trois $\chi$-r\'egulateurs locaux $\Delta^\chi_p(\eta)$, chaque fois suppos\'es non trivialement nuls modulo $p$
(pas de $\chi$-relations dans $F$). 

\smallskip
On pose $\alpha'= \alpha^\sigma$,
$\alpha''= \alpha^{\sigma^2}$, $\beta = \alpha^\tau$, $\beta' =  \alpha^{\tau\sigma} = \alpha'{}^\tau$, 
$\beta'' =  \alpha^{\tau\sigma^2} = \alpha''{}^\tau$.

\medskip
(i) Cas de $\Delta^1_p(\eta)$~; on a donc $\No_{K/\Q} (\eta)= a \ne \pm1$, auquel cas, $\Delta_p^1 (\eta)$ est le $p$-quotient de 
Fermat de $a$.

\medskip
(ii) Cas $\Delta^{\chi_1}_p(\eta)$~; on a $\No_{K/k} (\eta) \in k^\times\,\Sauf \Q^\times$ et on suppose que~:
$$\Delta_p^{\chi_1}  (\eta)=  \alpha +\alpha' +\alpha'' -\beta -\beta' -\beta'' \equiv 0 \pmod p. $$

Si $A = \alpha +\alpha' +\alpha'' =: u+v \sqrt m$, alors $\Delta_p^{\chi_1}  (\eta)=A - A^\tau = 2v \sqrt m\equiv 0 \pmod p$~; 
on a donc la seule condition $v \equiv 0 \pmod p$, ce qui conduit \`a une probabilit\'e en $\frac{O(1)}{p}$, la trace modulo $p$
dans $K/k$ \'etant surjective.

\medskip
(iii) Cas $\Delta^{\chi_2}_p(\eta)$ (consid\'er\'e au facteur $\sqrt m$ pr\`es)~; on a  ${\rm dim}((F\otimes \Q)^{e_\chi}) = 4$
(cas d'un caract\`ere de degr\'e 2), ce qui conduit, pour $\varphi = \theta=\chi_2$, \`a la condition  
(cf. Exemple \ref{ex5},\,\S\ref{sub111})~:
\begin{eqnarray*}
\Delta_p^{\theta} (\eta)&=& \alpha^{2} +\alpha'{}^2 +\alpha''{}^{2} -\beta^{2} -\beta'{}^2 -\beta''{}^{2}\\
&&\hspace{2cm} - \alpha  \alpha' - \alpha' \alpha'' -\alpha''  \alpha + \beta  \beta'+ \beta' \beta''+\beta'' \beta \equiv 0 \pmod p.
\end{eqnarray*}

Le calcul des trois repr\'esentations ${\mathcal L}^{\theta'}\simeq \delta' V_{\theta'}$, $0\leq \delta' \leq \varphi'(1)$,
permet de savoir quels sont les $\Delta_p^{\theta'}(\eta)$ nuls modulo $p$, m\^eme si l'on peut \'ecarter les cas o\`u $\Delta_p^1(\eta)$ ou $\Delta_p^{\chi_1}(\eta)$ est nul modulo $p$ (sinon, la nullit\'e suppl\'ementaire de $\Delta_p^{\theta}(\eta)$ conduit \`a une probabilit\'e au plus en $\frac{O(1)}{p^2}$, et on peut exclure ce nombre premier). 

\smallskip
Commen\c cons par des exemples num\'eriques ($\theta$ d\'esigne le caract\`ere $p$-adique $\chi_2$).

\smallskip
\subsubsection{ Programme PARI pour $D_6$ et le calcul de ${\mathcal L}^{\theta}$ et $\Delta_p^{\theta} (\eta)$}
Ce programme calcule les conjugu\'es de $\alpha$ sur la base des puissances de $x = \sqrt[3] 2 + j$.
 Ceci permet de trouver les relations de $\F_p$-d\'ependance de ces conjugu\'es (attention, on doit calculer le noyau 
de la transpos\'ee de la matrice des coefficients sur la base $\{x^5, x^4, x^3, x^2, x, 1\}$). 
Le r\'esultat de ``matker(MM)'' donne les vecteurs 
$(c_1, c_2, c_3, c_4, c_5, c_6)$ tels que~:

\smallskip
\centerline{$c_1 \alpha + c_2 \alpha^{\sigma} + c_3\alpha^{\sigma^2}+ c_4\alpha^\tau + c_5\alpha^{\tau\sigma}+ c_6\alpha^{\tau\sigma^2} \equiv 0 \pmod p. $}
\footnotesize 

\medskip
 $\{$$   Q=x^6+9*x^4-4*x^3+27*x^2+36*x+31; D=poldisc(Q); $\par
 $e1= x ;  $\par
 $e2=11/180*x^5 + 1/180*x^4 + 11/18*x^3 - 1/45*x^2 + 403/180*x + 419/180 ; $\par
 $e3=13/180*x^5 - 7/180*x^4 + 13/18*x^3 - 38/45*x^2 + 509/180*x + 127/180 ; $\par
 $e4=-4/45*x^5 + 1/45*x^4 - 8/9*x^3 + 26/45*x^2 - 137/45*x - 91/45 ;  $\par
 $e5=-1/60*x^5 - 1/60*x^4 - 1/6*x^3 - 4/15*x^2 - 73/60*x - 79/60 ; $\par
 $e6=-1/36*x^5 + 1/36*x^4 - 5/18*x^3 + 5/9*x^2 - 65/36*x + 11/36 ; $\par
 $a=1; b=-3 ; c=0 ; aa=-7; bb=1 ; cc=-1 ;  $\par
 $Eta= Mod(a*e1^5+b*e1^4+c*e1^3+aa*e1^2+bb*e1+cc,Q); $\par
 $N=norm(Eta); print("norme(eta) = ",N); print("   "); $\par
 $for (k= 3,5* 10^6, p=2*k+1;  if (isprime(p)==1  \& Mod(D*N,p)!=0, p2=p^2; $\par
 $u=Mod(a,p2); v=Mod(b,p2); w=Mod(c,p2);  $\par
 $uu=Mod(aa,p2); vv=Mod(bb,p2); ww=Mod(cc,p2);  $\par
 $P=x^6+9*x^4-4*x^3+27*x^2+36*x+Mod(31,p); $\par
 $P2=x^6+9*x^4-4*x^3+27*x^2+36*x+Mod(31,p2); $\par
 $y=Mod(u*x^5+v*x^4+w*x^3+uu*x^2+vv*x+ww ,P2); $\par
 $z=y^{(p-1)} ; t=z; for(i=1, 5, z=z^p*t); $\par
 $U=component(z-1,2);   $\par
 $u1=component(U,1); if(u1==0, u1=Mod(0,p2)); $\par
 $u2=component(U,2); if(u2==0, u2=Mod(0,p2)); $\par
 $u3=component(U,3); if(u3==0, u3=Mod(0,p2)); $\par
 $u4=component(U,4); if(u4==0, u4=Mod(0,p2)); $\par
 $u5=component(U,5); if(u5==0, u5=Mod(0,p2)); $\par
 $u6=component(U,6); if(u6==0, u6=Mod(0,p2)); $\par
 $x1=component(u1,2)/p; x2=component(u2,2)/p; x3=component(u3,2)/p;  $\par
 $x4=component(u4,2)/p; x5=component(u5,2)/p; x6=component(u6,2)/p;  $\par
 $X1=Mod(x1,p); X2=Mod(x2,p); X3=Mod(x3,p);  $\par
 $X4=Mod(x4,p); X5=Mod(x5,p); X6=Mod(x6,p);  $\par
 $E1= Mod(X1+X2*e1+X3*e1^2+X4*e1^3+X5*e1^4+X6*e1^5, P); $\par
 $E2=Mod(X1+X2*e2+X3*e2^2+X4*e2^3+X5*e2^4+X6*e2^5, P); $\par
 $E3=Mod(X1+X2*e3+X3*e3^2+X4*e3^3+X5*e3^4+X6*e3^5, P); $\par
 $E4=Mod(X1+X2*e4+X3*e4^2+X4*e4^3+X5*e4^4+X6*e4^5, P); $\par
 $E5=Mod(X1+X2*e5+X3*e5^2+X4*e5^3+X5*e5^4+X6*e5^5, P); $\par
 $E6=Mod(X1+X2*e6+X3*e6^2+X4*e6^3+X5*e6^4+X6*e6^5, P); $\par
 $ RP=E1^2+E2^2+E3^2 - E4^2-E5^2-E6^2-E1*E2-E2*E3 -E3*E1 $\par
  \hfill $   +E4*E5+E5*E6+E6*E4;  $\par
 $rp=component(RP,2);  if(rp==0, print(p)  ; $\par
$F1=component(E1,2); F2=component(E2,2); F3=component(E3,2); $\par
$F4=component(E4,2); F5=component(E5,2); F6=component(E6,2); $\par
 $print(F1); print(F2);  print(F3); print(F4);  print(F5); print(F6);$\par
$GG=[F1, F2, F3, F4, F5, F6]; $\par
$ M=matrix(6, 6, i, j, component ( component (GG, i), j) ); print(M); $\par
 $MM=mattranspose(M) ; K=matker(MM) ;   print(matker(MM))  )) ) $$\}$
\normalsize

\smallskip
\subsubsection{Cas $\eta = x^5-3x^4-7x^2+x -1$}
On a $\No(\eta) = 12393229477$ et on trouve les solutions 
$p = 7, 13, 69677, 387161$, pour~$p< 10^7$.

\medskip
a)  Pour $p=7$, on a les donn\'ees num\'eriques suivantes~:
$$\begin{array}{lllllllllllllll}
  \alpha &\equiv&     &0x^5& \!\!\!+\!\! \!  &2x^4 & \!\!\!+\!\! \!  &1 x^3 & \!\!\!+\!\! \!  &1 x^2 & \!\!\!+\!\! \!  & 5x & \!\!\!+\!\! \!  & 0 & \pmod p \\
 \alpha^\sigma &\equiv&       & 1 x^5 & \!\!\!+\!\! \!  &1 x^4 & \!\!\!+\!\! \!  & 6x^3 & \!\!\!+\!\! \!  & 3x^2 & \!\!\!+\!\! \!  & 5x & \!\!\!+\!\! \!  & 2  &\pmod p\\
 \alpha^{\sigma^2} &\equiv&                &   0x^5 & \!\!\!+\!\! \!  &   2x^4 & \!\!\!+\!\! \!  & 3x^3 & \!\!\!+\!\! \!  &0x^2& \!\!\!+\!\! \!  &  4x & \!\!\!+\!\! \!  &0 &\pmod p \\
   \alpha^\tau &\equiv&              & 0x^5 & \!\!\!+\!\! \!  &  5x^4 & \!\!\!+\!\! \!  & 6x^3 & \!\!\!+\!\! \!  & 6x^2 & \!\!\!+\!\! \!  & 2x & \!\!\!+\!\! \!  & 6 &\pmod p \\
  \alpha^{\tau\sigma}  &\equiv&      & 0x^5 & \!\!\!+\!\! \!  &  5x^4 & \!\!\!+\!\! \!  & 4x^3 & \!\!\!+\!\! \!  &  0x^2& \!\!\!+\!\! \!  & 3x & \!\!\!+\!\! \!  & 6  &\pmod p\\
 \alpha^{\tau\sigma^2} &\equiv&         & 6x^5 & \!\!\!+\!\! \!  & 6x^4 & \!\!\!+\!\! \!  &1 x^3 & \!\!\!+\!\! \!  & 4x^2 & \!\!\!+\!\! \!  & 2x & \!\!\!+\!\! \!  & 4 &\pmod p,
\end{array}$$
qui conduisent imm\'ediatement aux deux $\F_p$-relations lin\'eaires ind\'ependantes~:

\medskip
$\alpha -\alpha^{\sigma}  +  \alpha^\tau- \alpha^{\tau\sigma^2}\equiv 0 \pmod p\ \ \ \&\ \ \ 
\alpha  - \alpha^{\sigma^2} + \alpha^\tau - \alpha^{\tau\sigma}  \equiv 0 \pmod p$,
et leur rel\`evements $\eta_1^{1 - \sigma+ \tau -\tau \sigma^2 } \equiv 1 \pmod {p^2}$, 
$\eta_1^{1 - \sigma^2 + \tau - \tau\sigma } \equiv 1 \pmod {p^2}$. 

\medskip
Pour la $\theta$-relation $U= 1-\sigma+\tau-\tau \sigma^2$ on obtient $\sigma^2U = -U-\sigma U$,
$\tau U = -\sigma^2 U$, $\tau \sigma U = - \sigma U$, $\tau \sigma^2 U = - U$, et $U$ engendre
un espace de dimension 2 (${\mathcal L}^\theta \simeq V_\theta$).

\smallskip
b) Pour $p=13$, on a les donn\'ees num\'eriques suivantes~:
$$\begin{array}{lllllllllllllll}
  \alpha &\equiv&                                 &9x^5&  \!\!\!+\!\! \!     &0x^4 & \!\!\!+\!\! \!  &1 x^3 & \!\!\!+\!\! \!  &2 x^2 & \!\!\!+\!\! \!  & 10x & \!\!\!+\!\! \!  & 0 & \pmod p \\
 \alpha^\sigma &\equiv&                 & 3 x^5 & \!\!\!+\!\! \!  &8 x^4 & \!\!\!+\!\! \!  & 4x^3 & \!\!\!+\!\! \!  & 7x^2 & \!\!\!+\!\! \!  & 9x & \!\!\!+\!\! \!  & 6  &\pmod p\\
 \alpha^{\sigma^2} &\equiv&          & 11x^5 & \!\!\!+\!\! \!  & 9x^4 & \!\!\!+\!\! \!  & 4x^3 & \!\!\!+\!\! \!  &4x^2& \!\!\!+\!\! \!  & 7x & \!\!\!+\!\! \!  &0 &\pmod p \\
 \alpha^\tau &\equiv&                     & 0x^5 & \!\!\!+\!\! \!  & 2x^4 & \!\!\!+\!\! \!   & 11x^3 & \!\!\!+\!\! \!  & 4x^2 & \!\!\!+\!\! \!  & 7x & \!\!\!+\!\! \!  & 0 &\pmod p \\
 \alpha^{\tau\sigma}  &\equiv&      & 5x^5 & \!\!\!+\!\! \!  &1x^4 & \!\!\!+\!\! \!  & 11x^3 & \!\!\!+\!\! \!  & 7x^2& \!\!\!+\!\! \!  & 9x & \!\!\!+\!\! \!  & 6  &\pmod p\\
 \alpha^{\tau\sigma^2} &\equiv&    & 11x^5 & \!\!\!+\!\! \!  & 6x^4 & \!\!\!+\!\! \!  &8 x^3 & \!\!\!+\!\! \!  & 2x^2 & \!\!\!+\!\! \!  & 10x & \!\!\!+\!\! \!  & 0 &\pmod p ,
\end{array}$$

qui conduisent ici \`a~:

\smallskip
$\alpha - \alpha^{\sigma^2}+ \alpha^{\tau} - \alpha^{\tau\sigma^2} \equiv 0\ \pmod p\ \ \ \  \ \& \ \  \ \ \
 \alpha^{\sigma^2} - \alpha^{\sigma} + \alpha^{\tau\sigma}  - \alpha^{\tau} \equiv 0 \ \pmod p\ $
 et leur rel\`evements $\eta_1^{1-  \sigma^2 + \tau - \tau\sigma^2  } \equiv 1 \pmod {p^2}$, 
$\eta_1^{ \sigma^2 -\sigma + \tau\sigma - \tau  } \equiv 1 \pmod {p^2}$. 

\smallskip
 Pour la $\theta$-relation
$U=1-  \sigma^2 + \tau - \tau\sigma^2$, on a $\sigma^2 U = -U - \sigma U$, $\tau U = U$, $\tau \sigma U = \sigma^2 U$, $\tau \sigma^2 = \sigma U$  (${\mathcal L}^\theta \simeq V_\theta$).  

\smallskip
c) Dans le cas de $p=69677$, on trouve (modulo $p$)~:
$$\begin{array}{lllllllllllllll}
  \alpha &\!\!\!\equiv&     59144 x^5 &\!\!\!+\!\!\!&   25387 x^4 &\!\!\!+\!\!\!&   5395 x^3 &\!\!\!+\!\!\!&   69261 x^2 &\!\!\!+\!\!\!&   5560 x &\!\!\!+\!\!\!&   65182  \\
 \alpha^\sigma &\!\!\!\equiv&   33305 x^5 &\!\!\!+\!\!\!&   7449 x^4 &\!\!\!+\!\!\!&   67108 x^3 &\!\!\!+\!\!\!&     3673 x^2 &\!\!\!+\!\!\!&   9338 x &\!\!\!+\!\!\!&   28342 \\
 \alpha^{\sigma^2} &\!\!\!\equiv&   18122 x^5 &\!\!\!+\!\!\!&   61456 x^4 &\!\!\!+\!\!\!&   57729 x^3 &\!\!\!+\!\!\!&   9640 x^2 &\!\!\!+\!\!\!&   34177 x &\!\!\!+\!\!\!&   61498 \\
   \alpha^\tau &\!\!\!\equiv&  1290 x^5 &\!\!\!+\!\!\!&   27633 x^4 &\!\!\!+\!\!\!&   41529 x^3 &\!\!\!+\!\!\!&   16366 x^2 &\!\!\!+\!\!\!&   3270 x &\!\!\!+\!\!\!&   35656  \\
  \alpha^{\tau\sigma}  &\!\!\!\equiv&   47848 x^5 &\!\!\!+\!\!\!&   11928 x^4 &\!\!\!+\!\!\!&   47652 x^3 &\!\!\!+\!\!\!&   55209 x^2 &\!\!\!+\!\!\!&   6965 x &\!\!\!+\!\!\!&   8976 \\
 \alpha^{\tau\sigma^2} &\!\!\!\equiv& 49322 x^5 &\!\!\!+\!\!\!&   5501 x^4 &\!\!\!+\!\!\!&   59295 x^3 &\!\!\!+\!\!\!&   54882 x^2 &\!\!\!+\!\!\!&   10367 x &\!\!\!+\!\!\!&   61315\, ,
\end{array}$$

et les combinaisons $(c_1, c_2, c_3, c_4, c_5, c_6)$~:
$$(53404, 39540, 46410, 69676, 1, 0)\ \   \&  \ \  (23267, 16273, 30137, 69676, 0, 1) . $$

\subsubsection{Cas $p=7$, $\eta= x^5-x^4-7x^2+x-1$}
On obtient quatre relations $\F_p$-lin\'eaires
 ind\'ependantes dont $\alpha^\tau - \alpha\equiv 0 \pmod p$, 
$\alpha+\alpha^\sigma+\alpha^{\sigma^2}\equiv 0 \pmod p$, et leurs conjugu\'ees~:
$$\begin{array}{lllllllllllllll}
 \alpha &\!\!\!\equiv&6  x^5 &   &      &       \!\!\!+\!\!\!&  4  x^3 &\!\!\!+\!\!\!&  5  x^2 &\!\!\!+\!\!\!&   &  & 3  \\
 \alpha^\sigma &\!\!\!\equiv&  4  x^5 &\!\!\!+\!\!\!&  4  x^4 &\!\!\!+\!\!\!&  1  x^3 &\!\!\!+\!\!\!&  3  x^2 &\!\!\!+\!\!\!&  5  x &\!\!\!+\!\!\!&  1 \\ 
\alpha^{\sigma^2} &\!\!\!\equiv& 4  x^5 &\!\!\!+\!\!\!&  3  x^4 &\!\!\!+\!\!\!&  2  x^3 &\!\!\!+\!\!\!&  6  x^2 &\!\!\!+\!\!\!&  2  x &\!\!\!
 +\!\!\! &    3   \\
 \alpha^\tau &\!\!\!\equiv&  6  x^5 &\!\!\!+\!\!\!&  &  & 4  x^3 &\!\!\!+\!\!\!&  5  x^2 &\!\!\!+\!\!\!&   &  & 3  \\
   \alpha^{\tau\sigma}  &\!\!\!\equiv&4  x^5 &\!\!\!+\!\!\!&  3  x^4 &\!\!\!+\!\!\!&  2  x^3 &\!\!\!+\!\!\!&  6  x^2 &\!\!\!+\!\!\!&  2  x &\!\!\!+\!\!\!&   3  \\
\alpha^{\tau\sigma^2} &\!\!\!\equiv& 4  x^5 &\!\!\!+\!\!\!&  4  x^4 &\!\!\!+\!\!\!&  1  x^3 &\!\!\!+\!\!\!&  3  x^2 &\!\!\!+\!\!\!&  5  x &\!\!\!+\!\!\!&  1  \ .
\end{array} $$

Donc les trois r\'egulateurs  sont nuls modulo $p$. Mais pour $\theta$, ${\mathcal L}^\theta$ est engendr\'e par
$U= e_\theta (1-\tau)$ et par $\sigma U$~; on a $\sigma^2 U = -U - \sigma U$, $\tau U = -U$, $\tau\sigma U = -\sigma^2 U$,
$\tau\sigma^2 U = -\sigma U$ (${\mathcal L}^\theta \simeq V_\theta$).

\subsubsection{Cas $p=61$, $\eta= x^5-2x^4+4x^3-3x^2+x-1$}\label{p=61} 
On obtient~:
$$\begin{array}{lllllllllllllll}
 \alpha &\!\!\!\equiv&   8  x^5 &\!\!\!+\!\!\!&  32  x^4 &\!\!\!+\!\!\!&  26  x^3 &\!\!\!+\!\!\!&  31  x^2 &\!\!\!+\!\!\!&  55  x &\!\!\!+\!\!\!&  18  \\
 \alpha^\sigma &\!\!\!\equiv&  15  x^5 &\!\!\!+\!\!\!&  44  x^4 &\!\!\!+\!\!\!&  11  x^3 &\!\!\!+\!\!\!&  21  x^2 &\!\!\!+\!\!\!&  34  x &\!\!\!+\!\!\!&  25   \\
\alpha^{\sigma^2} &\!\!\!\equiv&  38  x^5 &\!\!\!+\!\!\!&  46  x^4 &\!\!\!+\!\!\!&  24  x^3 &\!\!\!+\!\!\!&  9  x^2 &\!\!\!+\!\!\!&  33  x &\!\!\!+\!\!\!&  44   \\
  \alpha^\tau &\!\!\!\equiv&  5  x^5 &\!\!\!+\!\!\!&  58  x^4 &\!\!\!+\!\!\!&  43  x^3 &\!\!\!+\!\!\!&  17  x^2 &\!\!\!+\!\!\!&  47  x &\!\!\!+\!\!\!&  16   \\
  \alpha^{\tau\sigma}  &\!\!\!\equiv&  31  x^5 &\!\!\!+\!\!\!&  7  x^4 &\!\!\!+\!\!\!&  22  x^3 &\!\!\!+\!\!\!&  55  x^2 &\!\!\!+\!\!\!&  36  x &\!\!\!+\!\!\!&  56   \\
 \alpha^{\tau\sigma^2} &\!\!\!\equiv& 25  x^5 &\!\!\!+\!\!\!&  57  x^4 &\!\!\!+\!\!\!&  57  x^3 &\!\!\!+\!\!\!&  50  x^2 &\!\!\!+\!\!\!&  39  x &\!\!\!+\!\!\!&  15\ .
\end{array} $$

D'o\`u trois relations $\F_p$-lin\'eaires ind\'ependantes, conjugu\'ees (par $\sigma$ et $\sigma^2$) de la relation 
$19\alpha+56\alpha^\sigma+46\alpha^{\sigma^2}+\alpha^\tau\equiv 0 \pmod p$~;
la somme des trois relations donne $\alpha+\alpha^\sigma+\alpha^{\sigma^2}- \alpha^\tau-\alpha^{\tau\sigma}-\alpha^{\tau\sigma^2} \equiv 0 \pmod p$ (nullit\'e de $\Delta_p^{\chi_1} (\eta)$). 

\smallskip
Pour $\theta=\chi_2$ de degr\'e 2, le $\theta$-r\'egulateur $\Delta_p^\theta (\eta)$ est nul modulo $p$ 
(voir  Exemple \ref{ex18},\,\S\ref{sub10} pour les d\'etails num\'eriques et de fait la nullit\'e triviale de $\Delta_p^\theta (\eta)$
car dans cet exemple on a par hasard $\eta^{1+\sigma+\sigma^2-\tau-\tau\sigma-\tau\sigma^2}=1$).

\subsubsection{Cas $p=7$, $\eta=3 x^5-20x^4+15 x^3+16x^2+9x+21$}
On obtient~:
$$\alpha \equiv \alpha^{\sigma}  \equiv \alpha^{\sigma^2} \equiv 6  x^5 +  2  x^4 +  4  x^3 +  3  x^2 +  6 \pmod p,$$
$$\alpha^\tau\equiv \alpha^{\tau\sigma} \equiv\alpha^{\tau\sigma^2}\equiv  x^5 +  5  x^4 +  3  x^3 +  4  x^2 +  6 
\pmod p. $$

Ce cas o\`u le $G$-module ${\mathcal L}^\theta$ est de $\F_p$-dimension 4 
(${\mathcal L}^\theta \simeq 2 \,V_\theta \simeq e_{\theta}\,\F_p[G]$) est tr\`es rare, comme on l'a vu au \S\ref{subD6}
(probabilit\'e en $\frac{O(1)}{p^4}$),
 car on doit prendre $\eta$ de telle sorte que $F$ soit de $\Z$-rang 6 et qu'aucun des $\Delta^\chi_p(\eta)$, $\chi = 1, \chi_1$,
ne soit nul modulo $p$, ce qui est le cas ici.

\smallskip
Il est facile de voir que cela correspond au cas o\`u $\eta$ est proche (modulo $p^2$) d'un \'el\'ement non trivial du sous-corps quadratique de $K$. 

\smallskip
 Il se trouve que $p=13$ est aussi solution et donne les conjugu\'es suivants~:
$$\begin{array}{lllllllllllllll}
 \alpha &\!\!\!\equiv&  2 x^5 &\!\!\!+\!\!\!&  3  x^4 &\!\!\!+\!\!\!& 7  x^3 &\!\!\!+\!\!\!&  8  x^2 &\!\!\!+\!\!\!& 3  x &\!\!\!+\!\!\!&  12  \\
\alpha^{\sigma} &\!\!\!\equiv&  2  x^5 &\!\!\!+\!\!\!& 5  x^4 &\!\!\!+\!\!\!&    x^3 &\!\!\!+\!\!\!&  3  x^2 &\!\!\!+\!\!\!&  8 x &\!\!\!+\!\!\!&  4   \\
 \alpha^{\sigma^2} &\!\!\!\equiv& 7  x^5 &\!\!\!+\!\!\!&  12  x^4 &\!\!\!+\!\!\!&  11  x^3 &\!\!\!+\!\!\!&  2  x^2 &\!\!\!+\!\!\!&  2  x &\!\!\!+\!\!\!& 10   \\
  \alpha^\tau &\!\!\!\equiv& 12  x^5 &\!\!\!+\!\!\!&  7 x^4 &\!\!\!+\!\!\!&  3 x^3 &\!\!\!+\!\!\!&  8  x^2 &\!\!\!+\!\!\!&  3  x &\!\!\!+\!\!\!&  12  \\
  \alpha^{\tau\sigma}  &\!\!\!\equiv&  4  x^5 &\!\!\!+\!\!\!&  3  x^4 &\!\!\!+\!\!\!&  7 x^3 &\!\!\!+\!\!\!&  2  x^2 &\!\!\!+\!\!\!&  2  x &\!\!\!+\!\!\!&  10   \\
 \alpha^{\tau\sigma^2} &\!\!\!\equiv& 12  x^5 &\!\!\!+\!\!\!&  9 x^4 &\!\!\!+\!\!\!& 10  x^3 &\!\!\!+\!\!\!&  3 x^2 &\!\!\!+\!\!\!&  8 x &\!\!\!+\!\!\!&  4\ ,
\end{array} $$

d'o\`u les trois $\F_p$-relations lin\'eaires ind\'ependantes~:
$$4\alpha+5\alpha^\sigma+5\alpha^{\sigma^2}+\alpha^\tau\! \equiv
5\alpha+5\alpha^\sigma+4\alpha^{\sigma^2}+\alpha^{\tau\sigma} \!\equiv
5\alpha+4\alpha^\sigma+5\alpha^{\sigma^2}+\alpha^{\tau\sigma^2}\! \equiv 0\!\!\!\pmod p, $$
donnant la nullit\'e modulo $13$ de  $\Delta_p^{1}(\eta)$ et  $\Delta_p^{\theta}(\eta)$
(relation associ\'ee d\'efinissant ${\mathcal L}^\theta$, $U=1 - {\sigma^2} - \tau + {\tau\sigma}$), 
et non  celle de $\Delta_p^{\chi_1}(\eta)$~; $F$ est bien de $\Z$-rang 6.

\begin{rema}\label{rema512} Les  \'el\'ements
$U_1 = 1-\sigma^2 + \tau -\tau\sigma, \ U_2 = 1-\sigma - \tau +\tau\sigma$
 de $e_\theta\,\F_p[G]$ (cf. Remarque \ref{rema511},\,\S\ref{sub5}) forment une base, et permettent d'expri\-mer 
toutes les $\theta$-relations 
trouv\'ees par combinaisons $\F_p[G]$-lin\'eaires (des coefficients de la forme $a+b\sigma$ \'etant suffisants en raison 
des relations de conjugaison de $U_1$,~$U_2$)~; ceci s'applique \`a tous les exemples num\'eriques pr\'ec\'edents.

\smallskip
Le $\F_p$-espace obtenu est de dimension $\delta \varphi(1)$, $1\leq \delta \leq 2$,
et la probabilit\'e associ\'ee est en $\frac{O(1)}{p^{\delta^2} }$ puisqu'ici $f=1$ (cf.\,\S\ref{HP}).
\end{rema}

\vspace{-0.7cm}
\section{Interpr\'etation locale concernant $\eta$}\label{sub32}
 L'id\'ee est de voir dans quelle mesure $\eta$ est  ``en partie'' une puissance 
$p$-i\`eme locale en $p$ lorsque $\Delta^{\theta}_p(\eta) \equiv 0 \pmod p$ 
pour un caract\`ere $p$-adique $\theta$, tout particu\-li\`erement dans le cas $f = \delta = 1$ et pour une
$p$-divisibilit\'e minimale $p^{\varphi(1)}$, d'apr\`es la D\'efinition \ref{defidec}, \S\ref{sub10}. 
Commen\c cons par un cas particulier bien connu~:

\vspace{-0.15cm}
\subsection{Cas d'un caract\`ere de degr\'e 1, d'ordre $n \div p-1$}\label{sub33}
Reprenons le r\'esultat du Corollaire~\ref{coro25},\,\S\ref{sub15} avec $\eta^{U_\theta} \in \prod_{v \div p} K_v^{\times p}$, 
o\`u $U_\theta = \sum_{k=0}^{n-1} r^{k}\,s^{-k}$, $r$ d'ordre $n$ modulo $p$, $s$ g\'en\'erateur de $G$. 
Ici $\Delta_p^\theta(\eta) = \No_{\mathfrak p}\big( \sum_{k=0}^{n-1} \varphi(s^k)\, \alpha^{s^{-k}}\big)
\equiv  \sum_{k=0}^{n-1} r^{k}\,\alpha^{s^{-k}} \pmod {\mathfrak p}$.
L'\'ecriture $\eta^{U_\theta} \in \prod_{v \div p} K_v^{\times p}$ 
conduit \`a $\eta^{{U_\theta}(r-s)} = \lambda^{p(r-s)}$, $\lambda\in \prod_{v \div p} K_v^{\times}$~; 
comme dans $\Z_{(p)}[G]$ on a ${U_\theta}\,(r-s) = c\,p$, $c\in \Z_{(p)}^\times$\,\footnote{\,On peut toujours choisir
$r$ modulo $p$ de telle sorte que $r^n - 1 \not\equiv 0 \pmod {p^2}$.},  il vient $\eta =  \Lambda^{r-s}$,
$\Lambda \in \prod_{v \div p} K_v^{\times}$ (pas de $p$-torsion dans $\prod_{v \div p} K_v^{\times}$ pour $p$ assez grand).

\smallskip
Le caract\`ere $\varphi \div \theta$ est tel que $\varphi (s) = \xi$,  racine primitive 
$n$-i\`eme de l'unit\'e  dans~$\C$, et il d\'efinit ${\mathfrak p} = (p, r-\xi)$ associ\'e \`a $\theta$~;
on consid\`ere la $\varphi$-composante $(K^\times \otimes C_\chi)^{e_\varphi}$~; alors $\xi = \varphi(s)$ op\`ere par 
$y^\xi = y^s$, pour tout $y\in (K^\times \otimes C_\chi)^{e_\varphi}$. 
On peut alors 
(en vue des aspects conjecturaux de la Section 7) dire par abus de langage que 
$\eta = \Lambda^{r-\xi}$ est une ``puissance ${\mathfrak p}$-i\`eme'' dans $\prod_{v \div p} K_v^{\times}$. 

Dans le cas $\chi$ d'ordre 1 ou 2, voir les\,\S\ref{sub366},\,\S\ref{sub37} qui reprennent cet aspect.

\vspace{-0.1cm}
\subsection{Cas d'un caract\`ere de degr\'e quelconque}\label{sub34} Soit $G$ un groupe fini et soit 
$\theta$ l'un de ses caract\`eres $p$-adiques irr\'eductibles tel que $\Delta_p^{\theta}(\eta) \equiv 0 \pmod p$
 (i.e.,  ${\mathcal L}^\theta \ne \{0\}$).
Le rel\`evement $\eta_1^{U_\theta} \equiv 1 \pmod {p^2}$ a lieu avec une $\theta$-relation
non triviale ${U_\theta} := \sm_{\nu \in G} u (\nu)\,\nu^{\!-1} \in {\mathcal L}^\theta$. 
On a $\eta_1^{U_\theta} = (1+p\,\gamma)^p \in\prd_{v \div p} K_v^{\times p}$, et 
 la relation $\eta = \eta^{p^{n_p}} \eta_1^{-1}$ conduit aussi  \`a $\eta^{U_\theta} \in \prd_{v \div p} K_v^{\times p}$,
mais l'interpr\'etation de $\eta$ comme ``puissance ${\mathfrak p}$-i\`eme'' dans $\prd_{v \div p} K_v^{\times}$ n'est
plus possible lorsque $\varphi \div \theta$ est de degr\'e $\geq 2$. 

Dans la section suivante on consid\`erera cette \'ecriture
comme une propri\'et\'e de ``puissance $p$-i\`eme locale partielle en $p$'', selon la d\'efinition suivante 
qui fait le lien avec le Lemme \ref{nul} sur les questions de nullit\'es triviales~:

\begin{defi}\label{defi28}  Soit $\eta \in K^\times$. On suppose que le $\Z[G]$-module $F$ engendr\'e par $\eta$
est de $\Z$-rang $n = \vert G\vert$.  Soit $p$ un nombre premier fix\'e assez grand et soit~:

\smallskip
\centerline{$F_0 :=\Big \{\eta_0 \in F, \  \eta_0 \in \prd_{v \div p} K_v^{\times p} \Big \} . $}

On dira que $\eta$ est puissance $p$-i\`eme locale partielle en $p$ si ${\rm dim}_{\F_p}(F/F_0) < n$.

\smallskip
Cette d\'efinition est \'equivalente \`a l'existence d'un caract\`ere $p$-adique irr\'eductible $\theta$ tel que 
${\rm dim}_{\F_p} (F/F_0)^{e_\theta} < f\,\varphi(1)^2$, donc de la forme $t f\,\varphi(1)$, $0\leq t \leq \varphi(1)-1$,
o\`u $f$ est le degr\'e r\'esiduel de $\theta$ (cf. D\'efinition \ref{defi01} (ii),\,\S\ref{calp}).
\end{defi}

\begin{rema}\label{rema260} 
 Comme par hypoth\`ese $F$ est de $\Z$-rang $n$ et 
sans $p$-torsion  (pour $p$ assez grand), on a $F/F^p \simeq \F_p[G]$ (en parti\-culier, ${\rm dim}_{\F_p}(F/F^p)^{e_\theta}  = f\,\varphi(1)^2$ pour tout $\theta$).

\smallskip
(i) Il en r\'esulte que la condi\-tion ${\rm dim}_{\F_p} (F/F_0) <n$ indique qu'il existe $\theta$ tel que la composante $(F/F_0)^{e_\theta}$
n'est pas maximale, donc qu'il existe $U_\theta \in e_\theta \Z_{(p)}[G]$ non triviale telle que
$\eta^{U_\theta}$ est  puissance $p$-i\`eme locale en $p$, non puissance $p$-i\`eme globale dans $K^\times$
(car $F\cap K^{\times p} = F^p$)~;  
ceci \'equivaut \`a ${\rm dim}_{\F_p} ({\mathcal L}^\theta) = \delta f\,\varphi(1)$, $\delta \geq 1$.

\smallskip
(ii) De fait on a $F/F_0 \simeq \F_p[G]/{\mathcal L}$, d'o\`u $(F/F_0)^{e_\theta} \simeq e_\theta\F_p[G]/{\mathcal L}^\theta$, ce
qui donne la relation $t = \varphi(1) - \delta$. Dire que $\eta  \in \prod_{v \div p} K_v^{\times p}$, c'est dire que $F_0 = F$, donc que ${\mathcal L} = \F_p[G]$ de probabilit\'e $\frac{O(1)}{p^n}$, cas qui peut \^etre \'ecart\'e pour~$n>1$.
\end{rema}

\begin{rema}\label{rema27} 
 En r\'esum\'e, l'obstruction fondamentale pour l'\'eventuelle finitude de l'ensemble des $p$ tels que 
$\Delta_p^G(\eta) = \prod_\theta \Delta_p^\theta(\eta)^{\varphi(1)} \equiv 0 \pmod p$ provient des cas de $p$-divisibilit\'e 
minimale $p^{\varphi(1)}$, \`a savoir (cf. D\'efinition \ref{defidec},\,\S\ref{sub10})~:

\smallskip
(i) il existe un unique caract\`ere $p$-adique irr\'eductible $\theta$ tel que $\Delta_p^\theta(\eta) \equiv 0 \pmod p$,

\smallskip
(ii) $p$ est totalement d\'ecompos\'e dans le corps des valeurs $C_\chi$ du caract\`ere irr\'eductible $\varphi \div \theta$
(i.e., $f=1$) et ${\rm Reg}_p^G(\eta) \sim p^{\varphi(1)}$  (cf. Remarque \ref{rema110}, \S\,\ref{gene1}),

\smallskip
(iii)  la repr\'esentation ${\mathcal L}^\theta \simeq \delta V_\theta$ des $\theta$-relations (cf. D\'efinition \ref{defi11} (iii),\,\S\ref{sub155}) est irr\'eductible, ce qui est \'equivalent \`a ${\rm dim}_{\F_p} ({\mathcal L}^\theta) = \varphi(1)$ (i.e., $\delta = 1$, 
cf. Th\'eor\`eme \ref{theo24}, \S\ref{sub15}).
\end{rema}

\section{Conjectures et sp\'eculations $p$-adiques} 

\subsection{G\'en\'eralit\'es}\label{prem}
Nous proposons diverses conjectures en partant de la plus ``\'evidente'' (Conjecture~\ref{conj291},\,\S\ref{sub366}) 
en relation avec la conjecture $ABC$, pour aller vers des g\'en\'erali\-sations aux corps de nombres  (Conjectures \ref{conj31},\,\S\ref{sub38} et \ref{conj341},\,\S\ref{sub39}).
L'obstruction essentielle \'etant donn\'ee par les cas de probabilit\'es en $\frac{O(1)}{p}$ pour les cas de $p$-divisibilit\'e minimale $p^{\varphi(1)}$ de ${\rm Reg}_p^G(\eta)$ (cf. D\'efi\-nition \ref{defidec},\,\S\ref{sub10}, et Remarque ci-dessus), nous proposons\,\S\ref{subalt} une approche pouvant modifier ce point de vue (\'etude limit\'ee au cas du quotient de Fermat).

\smallskip
Dans le cadre de la th\'eorie probabiliste des nombres (cf. \S\ref{sub11}), il convient de pr\'eciser des questions de densit\'es
au niveau des hypoth\`eses des \'enonc\'es. 
En th\'eorie alg\'ebrique des nombres on parle par exemple de {\it ``presque tout nombre premier $p$''} pour
signifier {\it ``pour tout nombre premier $p$ sauf pour un ensemble $\Sigma$ fini''}. Or on peut remplacer cette propri\'et\'e par la suivante, plus faible~:

\medskip
{\ \ \ \it pour tout nombre premier $p$ sauf pour un ensemble $\Sigma$ de premiers $p$ tels que~: }
$$\big \vert \{ p\in \,\Sigma,\  p \leq x  \}\big \vert  = o \Big(\Frac{x}{{\rm log}(x)} \Big),$$
qui permet un ensemble infini d'exceptions de densit\'e nulle (cf. \cite{T}, Chap.\,III.3.1).
On peut \'egalement envisager l'hypoth\`ese  plus faible ``{\it pour un ensemble infini de~$p$}''
(non n\'ecessairement de densit\'e 1).
Les \'enonc\'es de la Section 7 sont donn\'es sous la forme forte (alg\'ebrique), mais on pourra consid\'erer les formes faibles
pour le choix de $\Sigma$ (infini de densit\'e nulle), ou le point de vue ``{\it pour une infinit\'e de~$p$}''.

\subsection{Retour sur le cas des caract\`eres d'ordre 1 ou 2}\label{sub35} 
Nous revenons sur des cas particuliers d\'ej\`a \'evoqu\'es (cf.\,\S\ref{sub7}).

\subsubsection{Cas d'un rationnel}\label{sub366}
On consid\`ere $K=\Q$ et un rationnel $a \in \Q^\times$, $a \ne \pm 1$. Si $p$ est un nombre premier impair \'etranger \`a $a$, on a le r\'esultat \'el\'ementaire suivant qui est un cas particulier du Corollaire \ref{coro25},\,\S\ref{sub15} (pour $\theta=1$ et 
$U_\theta = 1$)~: 

\begin{lemm}\label{lemm29}
 Le $p$-quotient de Fermat de $a$ est nul modulo $p$ si et seulement si $a \in \Q_p^{\times p}$.
\end{lemm}

\begin{proof}   Si $a^{p-1} = 1+ p^2\,b$, $b$ $p$-entier, alors $a = a^p(1+ p^2\,b)^{-1} \in \Q_p^{\times p}$
(car $p\ne 2$). Si $a = u^p$, $u \in \Q_p^\times$, $a^{p-1} = (u^{p-1} )^p \equiv 1 \pmod {p^2}$.
\end{proof}

Une question cruciale, pour $a \ne \pm 1$, est l'existence ou non d'une infinit\'e de nombres premiers $p$ tels que
$a^{p-1} \equiv 1 \pmod {p^2}$. 

\smallskip
On sait, d'apr\`es un r\'esultat de Silverman \cite{Si}, que sous la conjecture $ABC$ l'ensemble des nombres premiers $p$ 
tels que $a^{p-1} \not\equiv 1 \pmod {p^2}$ est infini.\,\footnote{\,Silverman  prouve que pour tout $a \geq 2$, l'ensemble de ces nombres premiers $p\leq x$ est de cardinal $\geq c \,{\rm log}(x)$. 
Ce r\'esultat a \'et\'e \'etendu par Graves et Murty dans \cite{GM} aux $p \equiv 1 \pmod k$, pour tout
$k\geq 2$ fix\'e, auquel cas, l'ensemble de ces $p \leq x$ est de cardinal $\geq c \, \frac{ {\rm log} (x) }{ {\rm log} ( {\rm log} (x))}$, 
toujours sous la conjecture $ABC$. }
L'\'etude statistique montre que ce r\'esultat est une forme tr\`es faible de la r\'ealit\'e. 

\smallskip
Autrement dit, on a la propri\'et\'e conjecturale tr\`es raisonnable suivante qui devient 
un th\'eor\`eme si l'on admet la conjecture $ABC$ (cf.\,\S\ref{prem} pour les diff\'erentes acceptions
de {\it ``pour presque tout $p$''} \`a partir de {\it ``pour tout $p$ sauf un nombre fini''})~:

\begin{conj}\label{conj291}
 Soit $a \in \Q^\times$; si $a \in \Q_p^{\times p}$ pour presque tout $p$ alors $a = \pm 1$. 
\end{conj}

On peut consid\'erer cet \'ennonc\'e comme un principe local--global tr\`es particulier par rapport \`a ceux qui existent
en th\'eorie du corps de classes (alors purement alg\'ebriques comme le ``principe de Hasse'' pour les puissances rappel\'e
Proposition~\ref{prop32},\,\S\ref{sub38}). On pourra le qualifier de {\it principe local--global diophantien}.

\smallskip
Ces principes alg\'ebriques restent valables lorsqu'on remplace
 ``presque tout $p$'' par un ensemble de $p$ infini (ou fini) convenable restreint, mais non effectif en g\'en\'eral.  
Citons par exemple une forme tr\`es simple du th\'eor\`eme de Schmidt--Chevalley (un cas de principe de Hasse) 
qui s'\'enonce pour tous les corps de nombres et qui est \`a la base des propri\'et\'es id\'eliques du corps de classes 
(cf. \cite{Gr1}, III.4.3)~:

\begin{prop}\label{prop30}
Soit $A$ un sous-$\Z$-module de type fini de $\Q^\times$. Soit $e$ un entier fix\'e~;
alors il existe une infinit\'e d'ensembles finis $T$, de nombres premiers \'etrangers \`a $A$,
qui ont la propri\'et\'e suivante~: 
$$\{x \in A, \ \,x \equiv 1 \!\!\pmod p, \ \forall p\in T\} \subseteq  \Q^{\times e} . $$
\end{prop}

On peut  imaginer qu'un principe  local--global  diophantien existe dans le cas des $p$-quotients 
de Fermat de $a\in \Q^\times$, auquel cas on peut conjecturer que si l'on dispose d'un ensemble infini (m\^eme de
densit\'e nulle) de nombres premiers~$p$ tels que $a^{p-1} \equiv 1 \pmod {p^2}$, 
ceci est suffisant pour assurer que $a$ est \'egal \`a $\pm 1$.

\subsubsection{Cas de l'unit\'e d'un corps quadratique $K=\Q(\sqrt m)$}\label{sub37}  Si $\eta =x+  y\sqrt m$ est non rationnel 
de norme $a$, on a  $\eta^{p^2 - 1} = 1 + p\,\alpha$, $\alpha = u + v\sqrt m$, d'o\`u 
$\Delta^\chi_p (\eta)\equiv 2v\sqrt m \pmod p$ pour $\chi \ne 1$. 
On  a $\Delta^\chi_p(\eta)\equiv 0 \!\pmod p$ si et seulement si $v \equiv 0 \!\pmod p$. 

\smallskip
Supposons $m>0$ et que $\eta$ est une unit\'e $\varepsilon$ non triviale de $\Q(\sqrt m)$ (nullit\'e triviale modulo $p$ de 
$\Delta^{1}_p(\varepsilon)$)~; on a $u  \equiv 0 \pmod p$ et 
encore $\Delta^\chi_p (\varepsilon)\equiv 2v\sqrt m \pmod p$~; la nullit\'e modulo $p$ de $\Delta^\chi_p (\varepsilon)$ 
implique $\alpha\equiv 0 \pmod p$, et $\varepsilon$ est puissance $p$-i\`eme locale.  
Au plan conjectural, on est ramen\'e \`a la situation pr\'ec\'edente d'un rationnel. 
Autrement dit, il suffirait que l'on d\'emontre (via une forme ad\'equate de la conjec\-ture $ABC$) que la relation
$\varepsilon^{p^2-1} \not\equiv 1  \pmod {p^2}$ a lieu pour une infinit\'e de $p$, pour pouvoir \'enoncer l'analogue
pour $\varepsilon$ de la Conjecture ci-dessus.

On peut envisager que ce processus est valable pour le cas g\'en\'eral o\`u $\eta \in K^\times$ serait ``puissance
$p$-i\`eme locale partielle en $p$''  (cf. D\'efinition \ref{defi28},\,\S\ref{sub34}) pour une infinit\'e de~$p$.

\smallskip
Ceci suppose d'abord la r\'esolution du cas des puissances $p$-i\`emes locales en $p$ au sens commun, condition \'evidemment 
plus forte que celle de puissance $p$-i\`eme locale partielle en $p$~; c'est l'objet du point suivant.

\subsubsection{Conjecture g\'en\'erale sur les puissances $p$-i\`emes locales}\label{sub38} Le cas rationnel 
(Conjecture \ref{conj291},\,\S\ref{sub366}) a montr\'e la vraissemblance du type d'\'enonc\'e suivant qui cor\-respond \`a 
 l'\'ecriture $\eta^{p^{n_p} - 1} - 1 = p\, \alpha$ dans le cas (statistiquement tr\`es exceptionnel) o\`u $\alpha \equiv 0 \pmod p$
qui correspond \`a la condition ${\mathcal L} = \F_p[G]$ ($\delta =\varphi(1)$  pour tout caract\`ere $p$-adique $\theta$ de $G$) 
et une probabilit\'e en $\prod_\theta\frac{O(1)}{p^{f\varphi(1)^2}} = \frac{O(1)}{p^n}$, $n=\vert G\vert$, 
puisque $\sum_\theta f\,\varphi(1)^2 = \sum_\varphi \varphi(1)^2 = n$ (cf. Remarque \ref{rema260},\,\S\ref{sub34} et\,\S\ref{HP}).

\smallskip
Ceci pourrait \^etre une g\'en\'eralisation du th\'eor\`eme de Silverman (cas rationnel), qui utiliserait ici la conjecture $ABC$ pour les corps de nombres, mais la conjecture peut \^etre pos\'ee ind\'ependamment (se reporter aussi  au\,\S\ref{prem} pour les diff\'erentes acceptions de l'hypoth\`ese {\it ``pour presque tout nombre premier $p$''})~:

\begin{conj}\label{conj31}  Soit $K$ un corps de nombres et soit $\eta \in K^\times$. Si pour presque tout nombre premier $p$
on a $\eta \in  \prd_{v \div p} K_v^{\times p}$, alors $\eta$ est une racine de l'unit\'e de $K$.
\end{conj}

Cet \'enonc\'e est \`a comparer au ``principe de Hasse'' pour les puissances, beaucoup plus fort, et qui est le suivant (cf. \cite{Gr1}, II.6.3.3). Soit $\Pl_K$ (resp. $\Pl_p$) l'ensemble des places de $K$ (resp. des $p$-places de $K$)~:

\begin{prop}\label{prop32} Soit $\eta \in K^{\times}$. Soit $p$ un nombre premier fix\'e et soit $\Sigma$ un 
ensemble fini de places de $K$. 
Si $\eta$ est puissance $p$-i\`eme locale pour toute place $v \in \Pl_K\,\Sauf\,\Sigma$, alors $\eta \in K^{\times p}$.

Il existe une infinit\'e d'ensembles finis $T$ (non effectifs) de places de $K$  tels que si $\eta$ est puissance $p$-i\`eme locale 
pour toute place $v \in T$, alors $\eta \in K^{\times p}$.\,\footnote{Les \'enonc\'es classiques supposent toujours que $\Sigma$ est fini (simple \'elimination de places pathologiques) afin de pouvoir utiliser les th\'eor\`emes de densit\'e (Chebotarev) qui s'expriment en termes de progressions particuli\`eres de densit\'es canoniques~; or pour \^etre certain que de telles suites (en nombre fini), n\'ecessaires \`a la preuve, rencontrent bien le compl\'ementaire de $\Sigma$, celui-ci doit \^etre ``presque tout'', d\`es lors que s'il lui manquait une famille infinie (inconnue), il se pourrait que ``par hasard'' 
elle contienne les Frobenius dont on a besoin. On voit bien la distance qu'il peut y avoir entre un raisonnement alg\'ebrique g\'en\'eral et un raisonnement sur des hypoth\`eses nettement moins fortes portant par exemple sur des ensembles $\Sigma$ de densit\'e 
nulle~(cf.~\S\ref{prem}).}
\end{prop}

La diff\'erence op\`ere en deux temps pour la Conjecture \ref{conj31} ci-dessus~: partant de $p$ et de l'ensemble $\Pl_p$, 
on commence par dire que $\eta$ est puissance $p$-i\`eme locale pour tout $v \in \Pl_p$
 (i.e., on prend l'ensemble infini $\Sigma = \Pl_K\,\Sauf \Pl_p$, ce qui 
n'entra\^ine pas $\eta \in K^{\times p}$ comme on le v\'erifie facilement~; ou encore on peut dire qu'on essaye de prendre 
$T = \Pl_p$),  mais ensuite dans la Conjecture \ref{conj31} on suppose que cette propri\'et\'e locale (de type ``Hasse faible'') 
est vraie pour presque tout $p$, auquel cas $\eta$ serait conjecturalement dans $K^{\times p}$ pour presque 
tout $p$, donc une racine de~l'unit\'e.

\smallskip
La conjecture ``ultime'' qui fait le lien avec la th\'eorie des $\Delta^\theta_p(\eta)$ est la Conjecture \ref{conj34}
de la section suivante,\,\S\ref{sub39}. 

\smallskip
Auparavant, examinons le cas g\'en\'eral des unit\'es qui conforte l'analyse pr\'ec\'edente.

\subsubsection{Cas particulier du groupe des unit\'es -- Spiegelungssatz}\label{grunit}
Le cas o\`u $\eta$ est une unit\'e d'un corps de nombres $K$ quelconque met en jeu l'\'enonc\'e sp\'ecifique suivant 
(cf. \cite{Gr1}, II.6.3.8)~:

\begin{prop}\label{prop33} Soit $\eta$ une unit\'e de $K$. Soit $p$ un nombre premier fix\'e et soit $S$ un 
ensemble fini de places de $K$ tel que le $p$-groupe des classes de $K' := K(\mu_p)$ soit engendr\'e par les 
$p$-classes des id\'eaux premiers ${\mathfrak P_{v'}}$ de $K'$ pour les places $v'$ de $K'$ au-dessus de celles de $S$.

\smallskip
Si $\eta \in K_v^{\times p}$ pour toute place $v\in S \cup \Pl_p$, alors $\eta \in K^{\times p}$.
\end{prop}

Par exemple, pour $p=3$, $\eta = 1+ \sqrt 2$  (unit\'e fondamentale), on a  $\eta^8 \equiv 1+3\sqrt 2 \pmod 9$, et comme le $3$-groupe de classes de $\Q(\sqrt {-6})$ (donc celui de $K'$) est trivial ($S=\ev$), $\eta^m$ est localement un cube en 3 si et seulement si $3 \div m$. Autrement dit, le $3$-quotient de Fermat de $\eta$ ne pouvait pas \^etre nul modulo 3.

\smallskip
La Conjecture \ref{conj31},\,\S\ref{sub38} porte seulement sur $\Pl_p$ au lieu de $S \cup \Pl_p$ pour $S$ fini bien choisi (insuffisant pour avoir une puissance $p$-i\`eme globale), mais on suppose dans la conjecture que cette hypoth\`ese faible est vraie pour presque tout~$p$. 

\smallskip
Les deux syst\`emes d'hypoth\`eses co\"incident si  le $p$-groupe des classes du corps $K'$
est trivial ($S = \ev$), mais on peut \^etre plus pr\'ecis \`a ce sujet.

\medskip
Faisons quelques rappels classiques (voir par exemple \cite{Gr1}, I.6.3.1 et II.1.6.3) en supposant pour fixer les id\'ees
que $\eta$ est une unit\'e de Minkowski du corps totalement r\'eel $K$~; on peut toujours faire en sorte que $\eta$ soit non puissance $\ell$-i\`eme globale pour tout premier $\ell$.  Comme d'habitude les premiers $p$ consid\'er\'es sont assez grands et en particulier non ramifi\'es dans $K/\Q$.

\smallskip
S'il existe une $\theta$-relation ${U_\theta} \not \equiv 0 \pmod {p\Z_{(p)}[G]}$ pour laquelle
$\eta^{U_\theta} \in  \prod_{v \div p} K_v^{\times p}$ (cf.\,\S\ref{sub34}) alors l'extension $K'(\sqrt[p] {\eta^{U_\theta}})/K'$ 
est non ramifi\'ee et $p$-d\'ecompos\'ee ce qui conduit, par la th\'eorie du corps de classes, \`a une information sur
le $p$-groupe des classes $\Cl_{K'}$ de la fa\c con suivante~: soit $\Cl_{K'}^{\Pl'_p}$ le quotient de
$\Cl_{K'}$ par le $p$-sous-groupe des classes des id\'eaux premiers ${\mathfrak P}' \div p$ dans $K'$ et soit
$\theta^* := \omega \theta^{-1}$, o\`u $\omega$ est le caract\`ere de Teichm\"uller $p$-adique d\'efini, \`a partir d'une racine 
primitive $p$-i\`eme de l'unit\'e $\zeta_p$, par $\zeta_p^s = \zeta_p^{\omega(s)}$ pour tout $s \in {\rm Gal}(K'/K)$. 
Alors c'est la $\theta^*$-composante de $\Cl_{K'}^{\Pl'_p}$ qui est non triviale.

\smallskip
Revenons \`a l'\'ecriture $\eta_1 := \eta^{p^{n_p}-1} = 1 + p\,\alpha$, $\alpha \in Z_K$, et consid\'erons l'extension 
$K'(\sqrt[p] {\eta})= K'(\sqrt[p] {\eta_1})$ de $K'$,
qui est $p$-ramifi\'ee (i.e., non ramifi\'ee en dehors de $p$)~; pour chaque ${\mathfrak P}' \div p$ dans $K'$
le ${\mathfrak P}'$-conducteur est ${\mathfrak P}'{}^{p+1-r}$ o\`u $r$ est le plus grand entier tel que la congruence
$\eta_1 \equiv x'{}^p \pmod {{\mathfrak P}'{}^r}$, $x' \in K'{}^\times$, ait une solution $x'$. 
Comme $p$ est totalement ramifi\'e dans $K'/K$ (indice de ramification $p-1$ pour $p$ assez grand), $(p)$ s'\'ecrit $\prod_{{\mathfrak P}' \div  p} {\mathfrak P}'{}^{p-1}$ dans $K'$. 

Soit ${\mathfrak P} \div p$ dans $K$. Si $\alpha \in {\mathfrak P}$, alors $\alpha \in {\mathfrak P}'{}^{p-1}$ dans $K'$
et on v\'erifie facilement que, localement, $\eta_1$ est puissance $p$-i\`eme en ${\mathfrak P}'$ (non ramification
de $K'(\sqrt[p] {\eta})/K'$ en ${\mathfrak P}'$ et ${\mathfrak P}'$-d\'ecomposition). Si $\alpha \notin {\mathfrak P}$, 
on v\'erifie que la congruence $1+p\,\alpha \equiv  x'{}^p \pmod {{\mathfrak P}'{}^r}$ implique n\'ecessairement $r=p-1$, 
d'o\`u un ${\mathfrak P}'$-conducteur \'egal \`a ${\mathfrak P}'{}^2$ (ramification).

\smallskip
Appliquant ceci \`a $\eta^{U_\theta}$, il en r\'esulte que toute $\theta$-relation non triviale sur les 
conjugu\'es de $\alpha$, c'est-\`a-dire toute nullit\'e modulo $p$ d'un $\theta$-r\'egulateur $\Delta_p^\theta (\eta)$ 
($\theta \ne 1$), est \'equivalente \`a l'existence d'une $\theta^*$-extension non ramifi\'ee $p$-d\'ecompos\'ee de $K'$, de degr\'e une puissance non triviale $p$, contenue dans $K'(\sqrt[p] {F})/K'$, o\`u $F$ (ind\'ependant de $p$) est le $G$-module engendr\'e 
par $\eta$.
Une telle situation pour une infinit\'e de $p$ peut para\^itre excessive. 

\smallskip
En dehors du cas des unit\'es on a une situation un peu diff\'erente~:
prenons pour $K=\Q$ l'exemple de $a\in \Q^\times$, $a \ne \pm 1$. Alors la Proposition \ref{prop33},\,\S\ref{grunit} n'est plus valable car elle ne s'applique que si 
l'id\'eal $(a)$ est puissance $p$-i\`eme d'id\'eal dans $K$, mais si $a^{p-1} \equiv 1 \pmod {p^2}$ l'extension $\Q'(\sqrt[p] a)/\Q'$ 
est non ramifi\'ee en $p$ (et $p$-d\'ecompos\'ee) mais ramifi\'ee en les places de $\Q'$ divisant~$a$~; si $T$ est l'ensemble des diviseurs premiers de $a$, on doit alors remplacer le $p$-corps de classes de Hilbert $H'$ de $\Q'$ par sa g\'en\'eralisation $H'{}^{T'}$ comme $p$-extension Ab\'elienne maximale non ramifi\'ee en dehors des places de l'ensenble $T'$ des id\'eaux premiers de $\Q'$ au-dessus de $T$ ($T$ ne d\'epend pas de $p$).

\smallskip
Cette $p$-extension $H'{}^{T'}/\Q'$ est finie (c'est essentiellement un $p$-corps de rayon $K'_{\mathfrak m'}$, ${\mathfrak m'}$
construit sur ${T'}$) et joue donc un r\^ole analogue \`a celui de $H'$. Bien que l'aspect conjectural paraisse ici diff\'erent, 
l'infinitude des $p$ donnant un $p$-quotient de Fermat nul modulo $p$, peut se voir dans ce cadre.

\subsection{Conjectures sur les $\theta$-r\'egulateurs locaux $\Delta^\theta_p(\eta)$}\label{sub39}
 Les r\'esultats de la Section~3 (Th\'eor\`eme \ref{theo24} et Corollaire \ref{coro25},\,\S\ref{sub15}  g\'en\'eralisant en un sens le Lem\-me \ref{lemm29},\,\S\ref{sub366}) invitent \`a proposer les conjectures suivantes plus fortes que celles du\,\S\ref{sub35}, 
m\^eme si aucun cas particulier de preuve ne semble connu. 

\smallskip
Pour \'eviter le cas des $\chi$-r\'egulateurs trivialement nuls (cf. Remarques \ref{rema51}, \ref{rema511} et Lemme \ref{nul}, \S\ref{sub5}), on suppose le $G$-module engendr\'e par $\eta$ de $\Z$-rang $n$.

\smallskip
La premi\`ere conjecture se justifie en raison de l'\'ecriture de type ``quotient de Fermat''~: $\eta^{p^{n_p} - 1} - 1 = p\, \alpha$,
lorsque $\alpha$ satisfait \`a des relations de la forme $U_\theta \,. \,\alpha \equiv 0 \pmod p$
avec $U_\theta \not\equiv 0 \pmod {p\Z_p[G]}$ (i.e., ${\mathcal L}^\theta \simeq \delta V_\theta$, $\delta \geq 1$), 
pour au moins un $\theta$, condition beaucoup plus faible que $U =1$ qui correspond au 
``$\alpha \equiv 0 \pmod p$'' \`a la base de la Conjecture  \ref{conj31},\,\S\ref{sub38}, c'est-\`a-dire \`a 
${\mathcal L} \simeq \F_p[G]$ de probabilit\'e $\frac{O(1)}{p^n}$).

\begin{conj}\label{conj341} Soit $K/\Q$ une extension Galoisienne de degr\'e $n$, de groupe de Galois $G$.
Soit $\eta \in K^\times$ tel que le $\Z[G]$-module engendr\'e par $\eta$ soit de $\Z$-rang~$n$.
Alors pour tout $p$ assez grand, $\eta$ n'est pas une puissance $p$-i\`eme locale partielle en~$p$
(cf. D\'efinition \ref{defi28},\,\S\ref{sub34}).
\end{conj}

L'\'enonc\'e suivant est \'equivalent au pr\'ec\'edent. 
On rappelle (cf. D\'efinition \ref{defi10}, Remarque \ref{rema110},\,\S\ref{gene1}) que pour tout caract\`ere $p$-adique 
$\theta$, ${\rm Reg}_p^\theta(\eta) \equiv \Delta_p^\theta(\eta) \pmod p$
et que ${\rm Reg}_p^G(\eta) = p^{-n}\, {\rm det} \big ({\rm log}_p(\eta^{\tau\sigma})\big )_{\sigma, \tau \in G}$ (le r\'egulateur $p$-adique normalis\'e de $\eta$, cf. D\'efinition \ref{defimodif} (i),\,\S\ref{sub1}) se factorise en ${\rm Reg}_p^G(\eta) = \prod_\theta {\rm Reg}_p^\theta(\eta)^{\varphi(1)}$.

\begin{conj}\label{conj34} Soit $K/\Q$ une extension Galoisienne de degr\'e $n$, de groupe de Galois $G$.
Soit $\eta \in K^\times$ tel que le $\Z[G]$-module engendr\'e par $\eta$ soit de $\Z$-rang~$n$, et 
soit ${\rm Reg}_p^G(\eta)$ le  r\'egulateur $p$-adique normalis\'e de $\eta$.

Alors pour tout premier $p$ assez grand ${\rm Reg}_p^G(\eta)$ est une unit\'e $p$-adique.
\end{conj}

\begin{rema}\label{rema35}
(i) Les Conjectures \ref{conj341}, \ref{conj34}, peuvent s'\'enoncer sous une hypoth\`ese plus faible qui consiste \`a remplacer ``{\it presque tout $p$}''  au sens alg\'ebrique par ``{\it presque tout $p$}''  au sens  du\,\S\ref{prem}.

\smallskip
 (ii) La Conjecture \ref{conj34} implique celle de Leopoldt--Jaulent \cite{J}, au moins pour presque tout $p$, mais il est pr\'ef\'erable d'admettre cette derni\`ere et de dire que la Conjecture \ref{conj34} en est une version plus forte (voir la Remarque \ref{global}, (a) et (b)).

\smallskip
(iii) La version tr\`es raisonnable des conjectures pr\'ec\'edentes est que les $p$ assez grands tels que $\eta$ est puissance $p$-i\`eme locale partielle en $p$, sont les cas de $p$-divisibilit\'e minimale $p^{\varphi(1)}$ (cf. Remarques \ref{rema260}, \ref{rema27},
 \S\,\ref{sub34}).
\end{rema}

\vspace{-0.4cm}
\subsection{Conjectures en $p$-ramification Ab\'elienne} \label{sub360}
Soit $H^{p\rm ra}_{\!(p)}$ la $p$-extension Ab\'elienne $p$-ramifi\'ee (i.e., non rami\-fi\'ee en dehors de $p$)  maximale
d'un corps de nombres Galoisien $K$ (r\'eel). Soit $\wh K_{\!(p)}$ le compos\'e des $\Z_p$-extensions de $K$ et soit 
${\mathcal T}_p = {\rm Gal} \big (H^{p\rm ra}_{\!(p)}/\wh K_{\!(p)}\big )$.  Dans ce cas, pour tout $p$ assez grand, 
$\vert {\mathcal T}_p\vert $ a m\^eme valuation $p$-adique que le r\'egulateur normalis\'e du corps $K$,
$p^{1-n}\,{\rm Reg}_p(K)  \sim \prd_{\theta \ne 1}  {\rm Reg}_p^\theta(\varepsilon)^{\varphi(1)}$, o\`u $\varepsilon$ est une unit\'e de Minkowski fix\'ee (voir \cite{Gr1},\,\S\,III.2.6.5,  \cite{Gr2},\,\S\,3, pour des compl\'ements sur ces questions).

\smallskip
Alors la Conjecture \ref{conj34} est reformul\'ee par la conjecture suivante (au moins pour les corps r\'eels, toute partie ${\mathcal T}_p^-$ \'eventuelle \'etant triviale d\`es que $p$ est assez grand)~:

\begin{conj}\label{conj345}
L'invariant $\prd_p {\mathcal T}_p$ est fini pour tout corps de nombres.
\end{conj}

 La version raisonnable nettement plus faible consiste \`a dire que pour tout $p$ assez grand, $\vert {\mathcal T}_p\vert $
 est soit trivial soit \'egal \`a $p^{\varphi(1)}$ (\`a une unit\'e $p$-adique pr\`es),
en raison d'un unique facteur ${\rm Reg}_p^\theta(\varepsilon)$ divisible par $p$ avec $\theta \ne 1$ de degr\'e r\'esiduel $f=1$
et ${\mathcal L}^\theta \simeq V_\theta$ (i.e., $p$-divisibilit\'e minimale).
En effet, tout autre cas r\'esulte ou bien de la nullit\'e modulo $p$ d'au moins deux $\Delta^\theta_p(\varepsilon)$ pour des  caract\`eres $p$-adiques distincts (probabilit\'e au plus en $\frac{O(1)}{p^2}$), 
ou bien de la nullit\'e modulo $p$ d'un unique $\Delta^\theta_p(\varepsilon)$ avec $f \delta^2 \geq 2$ 
(probabilit\'e en $\frac{O(1)}{p^{f \delta^2}}$, cf.\,\S\ref{HP}),
le cas $p^{e} \div {\rm Reg}_p^\theta(\varepsilon)$, $e \geq 2$, lorsque $f=\delta=1$, \'etant
de probabilit\'e au plus en $\frac{O(1)}{p^2}$ (cf. \S\,\ref{subex}).

\smallskip
Dans le cas Ab\'elien, ceci impliquerait que les seuls $p$ assez grands pour lesquels $\vert {\mathcal T}_p\vert \equiv 0 \pmod p$ 
sont ceux pour lesquels un unique ${\rm Reg}_p^\theta(\varepsilon)$ est exactement divi\-sible par $p$ 
pour $p \equiv 1 \pmod d$ o\`u $d$ est l'ordre de $\varphi$.

\smallskip
De fait, dans le cadre des ``Th\'eor\`emes principaux'', on aurait (par exemple pour l'invariant ${\mathcal T}_p$ pr\'ec\'edent)
la conjecture raisonnable suivante~:

\begin{conj}\label{conj3451} Soit $K/\Q$ une extension Galoisienne r\'eelle finie de groupe de Galois $G$. 
Alors, pour tout $p$ assez grand il existe au plus un caract\`ere $p$-adique irr\'e\-ductible $\theta \ne 1$ de $G$ de degr\'e r\'esiduel $f=1$ pour lequel ${\mathcal L}^\theta \simeq V_\theta$ (i.e., $\delta=1$), tel que
${\mathcal T}_p \simeq e_\theta\, \F_p[G]$ (module $G$-cyclique d'ordre $p^{\varphi(1)}$, $\varphi \div \theta$).
\end{conj}

\begin{rema}
 On en d\'eduirait aussi des propri\'et\'es analogues sur le r\'esidu de la fonction z\^eta $p$-adique 
(cf. \cite{Coa} (Appendix) et  \cite{Se2}). Dans  \cite{Hat}, est utilis\'e le fait que lorsque la $p$-valuation de 
$\zeta_K(2-p)$ est n\'egative, alors elle vaut~$-1$ (cf. \cite{Se2}, Th\'eor\`eme 6)~; dans le cadre de la Conjecture 
\ref{conj345} ci-dessus, on aurait alors conjecturalement
$\Frac{\zeta_K(2-p)}{\zeta_\Q(2-p)} \sim \vert {\mathcal T}_p\vert \sim 1$ pour tout $p$ assez grand, et
sur un plan cohomologique, on aurait ${\rm H}^2({\mathcal G}_p, \Z_p) \simeq {\mathcal T}_p^*=1$ 
pour tout $p$ assez grand, o\`u ${\mathcal G}_p$ est le groupe de Galois de la pro-$p$-extension $p$-ramifi\'ee maximale de $K$. Plus faiblement, on aurait
 $\Frac{\zeta_K(2-p)}{\zeta_\Q(2-p)} \sim  p^{\varphi(1)}$ et
${\rm H}^2({\mathcal G}_p, \Z_p) \sim  p^{\varphi(1)}$ (pour les cas de $p$-divisibilit\'e minimale).
\end{rema}

\subsection{Analyse probabiliste alternative pour $q_p(a) \equiv 0 \pmod p$}\label{subalt}
Soit  $a \in \Q^\times$, $a\ne \pm 1$ fix\'e. Soit $p$ un nombre premier \'etranger \`a $a$ (de toutes fa\c cons $p$
est assez grand selon le point de vue adopt\'e). 
Soit $m = o_p(a)\,\big \vert \, p-1$ l'ordre de $a$ modulo $p$ et soit $\xi$ une racine primitive 
$m$-i\`eme de l'unit\'e dans $\C$~; alors on peut \'ecrire $a^m-1 = \prod_{j=1}^m (a- \xi^j) \equiv 0 \pmod p$. 
Comme $m$ est l'ordre de $a$ modulo $p$, c'est le facteur de $a^m-1$ d\'efini par~:
$$\Phi_m(a) = \prd_{t \in (\Z/m\Z)^\times} (a- \xi^t)$$

\vspace{-0.1cm}
qui est dans $p\,\Z_{(p)}$,  o\`u $\Phi_m$ est le $m$-i\`eme polyn\^ome cyclotomique.
De fa\c con pr\'ecise on a la relation $q_p(a) := \Frac{a^{p-1} - 1}{p}  = 
\Frac{\Phi_m(a)}{p} \times \prd_{\mathop{d\vert p-1}\limits_{d \ne m}} \Phi_d(a)$,
$\prd_{\mathop {d\vert p-1}\limits_{d \ne m}} \Phi_d(a) \not\equiv 0 \pmod p$. 

\vspace{-0.1cm}
Si $a = \frac{u}{v}$, $\Phi_m(a) = v^{-\phi(m)} \prod_{t} (u- v\,\xi^t)$ et  le num\'erateur 
$\Phi_m (u,v):= \prod_{t} (u- v\,\xi^t)$ est seul concern\'e. On pose~:
$$\wt\Phi_m(a) := \Frac{\Phi_m(u,v)}{{\rm p.g.c.d.}\,(\Phi_m(u,v),m )}$$

pour \'eliminer les facteurs premiers $p'$ ramifi\'es dans $\Q(\xi)/\Q$ qui n'interviennent pas 
car $m' := o_{p'}(a)$ est un diviseur strict de $m$ (e.g. $p=19$, $a=8$, $m=6$, 
$\Phi_6(8)=3\times 19$ qui introduit $p'=3$ avec $o_{p'}(8)=2$ et $\Phi_2(8)=9$ comme attendu).

On peut donc aussi d\'efinir un quotient de Fermat normalis\'e
$\wt q_p(a) := \frac{\wt \Phi_{o_p(a)}(a)}{p}$~; 
dans tous les cas on remplace $q_p(a)$ et $\wt q_p(a)$ par leurs repr\'esentants dans $[0, p[$.

\smallskip
Puisque dans le corps $\Q(\xi)$, $p$ est totalement d\'ecompos\'e, il existe un id\'eal premier
${\mathfrak p}\div p$ de $\Q(\xi)$ tel que $(u-v\,\xi)\,\Z[\xi] = {\mathfrak p}^{n_p}\,{\mathfrak a}$, $n_p \geq 1$ et 
${\mathfrak p} \notdiv {\mathfrak a}$, o\`u ${\mathfrak a}$ est un id\'eal entier de $\Q(\xi)$~; de plus, ${\mathfrak a}$
est \'etranger \`a $p$ car les congruences $u-v\,\xi \equiv 0 \pmod  {\mathfrak p}$ et $u-v\,\xi \equiv 0 \pmod 
{\mathfrak p' \div p}$, ${\mathfrak p'} \ne {\mathfrak p}$, conduisent, par la conjugaison non triviale $\xi\mapsto \xi^t$
qui transforme ${\mathfrak p'}$ en ${\mathfrak p}$, \`a 
$u-v\,\xi \equiv 0 \pmod  {\mathfrak p}$ et $u-v\,\xi^t \equiv 0 \pmod {\mathfrak p}$, d'o\`u $v \,(\xi^t - \xi)\equiv 0 \pmod {\mathfrak p}$ (absurde).

\smallskip
Si  ${\mathfrak l}$ est un id\'eal premier divisant ${\mathfrak a}$, la congruence $\xi \equiv a \pmod {\mathfrak l}$
montre que ${\mathfrak l}$ est totalement d\'ecompos\'e dans $\Q(\xi)/\Q$ (sauf les \'eventuels $ {\mathfrak l} \div \ell \div m$, ramifi\'es que l'on a \'ecart\'es). Autrement dit, si l'on pose~:

\smallskip
\centerline{ $\wt\Phi_m(a) = \prd_{k=1}^e \ell_k^{n_k}$, $\ \,\ell_1 < \ell_2< \ldots < \ell_e$,\ \ $n_k \geq 1, \ \ m = o_p(a)$,  } 

\smallskip
 tous les premiers $\ell_k$ sont congrus \`a 1 modulo $m$ et $p$ est l'un d'eux~; en outre on a $\ell \div \wt\Phi_m(a)$
si et seulement si $o_\ell(a) = m$. Il en r\'esulte aussi que pour un tel $\ell =\ell_k$ (en posant $\ell -1 = h\,m$), $\ell $ est totalemment d\'ecompos\'e dans l'extension Galoisienne $\Q(\mu_{\ell-1})(\sqrt[h] a)/\Q$ puisque $a$ est localement de la forme 
$b^h$ modulo~$\ell$ ($\ell$~ne divise pas $a$ et n'est pas ramifi\'e dans cette extension).
Ces questions d'ordres modulo $p$ sont li\'ees \`a des techniques issues de la conjecture d'Artin
sur les racines primitives et de la d\'emonstration de Hooley (voir \cite{Mo} pour un expos\'e exhaustif).

\smallskip
On voit alors que la condition $q_{\ell_k}(a) = 0$ est \'equivalente \`a $n_k \geq 2$, et pour un entier $m$ fix\'e
les seuls nombres premiers $p$ v\'erifiant \`a la fois $o_p(a) = m$ et $q_p(a)=0$ sont les 
nombres premiers ${\ell_k}$ tels que $n_k \geq 2$. 

\smallskip
Ainsi, la recherche des quotients de Fermat nuls est de nature 
multiplicative, a priori diff\'erente de  celle des quotients de Fermat $1, 2, \ldots , p-1$. En effet, 
le cas $q_{\ell_k}(a)=0$ se lit sur l'exposant $n_k$ tandis qu'autrement, $\wt q_{\ell_k}(a) \equiv \prd_{j\ne k} \ell_j^{n_j} \pmod{\ell_k}$ d\'epend des autres diviseurs premiers selon une congruence al\'eatoire.

 Il serait utile de conna\^itre d\'ej\`a la probabilit\'e pour les entiers $\wt\Phi_m(a)$
(qui ne sont pas ``quelconques'' car $\ell_k \equiv 1 \pmod m$ \'elimine les petits carr\'es) d'avoir un facteur carr\'e non trivial, celle  d'un entier pris au hasard \'etant assez grande
(\'egale \`a $1-\frac{6}{\pi^2} \sim 0.392$)~; nous renvoyons aux techniques de \cite{T} qui peuvent le permettre.

Les cas o\`u $\wt\Phi_m(a)$ est divisible par le carr\'e d'un nombre premier $\ell$ sont en effet rarissimes~; 
par exemple, pour $a=14$ et $p=29$, on a $m=o_p(a)=28$ avec 
$$\wt\Phi_m(a)=  29^2 \times 3361 \times176597, $$
 ce qui donne donc $q_{29}(14)=0$. 

\smallskip
Consid\'erons l'exemple num\'erique suivant~:  $a=12$ et $m=35$~; on a $\Phi_{35}(x) \ =\ $

$x^{24} - x^{23} + x^{19} - x^{18} + x^{17} - x^{16} + x^{14} - x^{13} + x^{12} - x^{11} + x^{10} - x^8 + x^7 - x^6 + x^5 - x + 1$,

qui conduit  \`a~:

\smallskip
$\wt\Phi_{35}(12)\! =\! 72872404828019704577129461 \! =\!  71\!  \times\!   491\!  \times\!  806821\!  \times \! 6089651\!  \times \! 425455031$.

\smallskip
Les premiers $p = 71,  491, 806821, 6089651, 425455031$ jouent des r\^oles sym\'etriques mais leurs quotients de Fermat sont tous non nuls.
Si l'on attribuait en g\'en\'eral la probabilit\'e $\frac{1}{\ell_k}$  d'avoir $n_k \geq 2$, il semble qu'on aboutirait \`a des contradictions num\'e\-riques vu la pr\'esence assez syst\'ematique de petits et grands nombres premiers et le fait que cette probabilit\'e d\'epend de l'ordre de grandeur de $\wt\Phi_m(a)$. 

\smallskip
 Le produit sur $m \geq 1$ des $\wt\Phi_m(a)$ a des propri\'et\'es de r\'egularit\'e int\'eressantes. 

\smallskip
Consid\'erons l'aspect num\'erique donn\'e par le programme suivant~: 

\footnotesize
\medskip
$\{$$a=14; A=1; for(m=1, 40, P=polcyclo(m); x=a; B=eval(P); $\par
$B=B/gcd(B,m); A=A*B); print(a); print(factor(A))$$\}$
\normalsize

\medskip
Pour $a=14$, on obtient
$\prod_{m=1}^{40} \wt\Phi_m(a)=$

\smallskip
\footnotesize
$3\!\times\! 5 \!\times\! 11 \!\times\! 13 \!\times\! 17 \!\times\! 19 \!\times\! 23 \!\times\! 29^2 \!\times\! 31 \!\times\! 37 \!\times\! 41 \!\times\! 43 \!\times\! 47 \!\times\! 61 \!\times\!  67 \!\times\! 71 \!\times\! 79 \!\times\! 101 \!\times\! 103 \!\times\! 113 \!\times\! 137 \!\times\! 157 \!\times\! 191 \!\times\! 193 \!\times\! 197 \!\times\! 211 \!\times\!  223 \!\times\! 397 \!\times\! 461 \!\times\! 547 \!\times\! 811 \!\times\! 911 \!\times\! 937 \!\times\! 1033 \!\times\! 1061 \!\times\! 2347 \!\times\! 2851\!\times\! 3361 \!\times\! 3761 \!\times\! 4027 \!\times\! 5393 \!\times\! 7307 \!\times\! 10627 \!\times\! 13109 \!\times\! 15511 \!\times\! 16097 \!\times\!  18973 \!\times\! 26981 \!\times\! 61001 \!\times\! 100621 \!\times\! 132049 \!\times\! 176597 \!\times\! 698521 \!\times\! 761437 \!\times\!  1154539 
\!\times\! 1383881 \!\times\! 1948981 \!\times\! 2249861 \!\times\! 7027567 \!\times\! 8108731 \!\times\! 14525237\!\times\! 51111761 \!\times\! 110256001 \!\times\! 1427145211 \!\times\! 1475750641 \!\times\! 2239000891 \!\times\!  11737870057 \!\times\! 25581350023 \!\times\! 29914249171 \!\times\! 141405986837 \!\times\!  299113818931 \!\times\! 56693904845761 \!\times\! 77312552100349 \!\times\! 758855846709601 \!\times\!  11284732320255809 \!\times\! 14837$

$638311110071 \!\times\! 22771730193675277 \!\times\!  396530555859061913 \!\times\! 459715689149916492091 \!\times\! 39211413$

$30646275580183\!\times\! 77720275181800334933851 \!\times\! 2984619585279628795345143571 \!\times\!  62233081779326$

$83558580086481 \!\times\!26063080998214179685167270877966651$.
\normalsize

\medskip
On constate une progression rapide mais r\'eguli\`ere des nombres premiers obtenus, qui doivent tous appara\^itre au fur et \`a mesure que la borne sur $m$ (ici $40$) augmente~; en effet, si $p$ est fix\'e, il appara\^it comme diviseur de $\wt\Phi_m(a)$ pour l'unique indice $m$ \'egal \`a $o_p(a)$. Seul $29$ est au carr\'e (le programme du\,\S\ref{sub7}.
donne pour $p< 10^9$ les deux solutions $29$ ($m= 28$) et $353$ ($m=352$)).

\begin{rema}\label{remademi} Signalons la propri\'et\'e d'uniformit\'e suivante qui concerne essentiellement la totalit\'e des 
quotients de Fermat non nuls, et non la valeur 0. Si l'on calcule la somme
$\Frac{1}{N}\sm_{j=1}^N  \Frac{q_{p_j}(a)}{p_j} , \ \  q_{p_j}(a) \in ] \,0, p_j \,[ $,
 pour la suite des nombres premiers \'etrangers \`a $a$ ($p_1=2$, $p_2=3$, \ldots , $p_N$),
 il y a une convergence vers $\frac{1}{2}$ tout \`a fait remarquable et ind\'ependante de $a$.
Ceci s'explique sans doute par le fait que la valeur moyenne de $q_{p}(a)$ pour $p$ fix\'e et $a$ variable est $\frac{p}{2}$.
Si l'on pose, pour $a$ fix\'e, $q_{p_j}(a) = u_j p_j$, $0 < u_j < 1$, la valeur moyenne de la variable al\'eatoire $u_j$ est $\frac{1}{2}$
et ces variables sont ind\'ependantes (cf. \cite{EM}, \cite{H-B}, \cite{He}, \cite{OS} pour d'autres r\'esultats et r\'ef\'erences
li\'es au quotient de Fermat).
\end{rema}

On peut v\'erifier la propri\'et\'e ci-dessus au moyen du programme suivant qui calcule aussi (dans $N_0$) le nombre de $q_{p_j}(a)$ 
\'egaux \`a une valeur donn\'ee (ici 0)~:

\smallskip
$\{$$for(i=1,100, e=random(1111); print("e = ",e); N=0; N0=0; S=0.0; $\par$
if(Mod(e,2)==Mod(1,2), f=Mod((e-1)/2,2); S=S+component(f,2)/2 ); $\par$
 for(n=1, 5*10^8, p=2*n+1;$\par$
 if(isprime(p)==1 \& Mod(e,p)!=Mod(0,p) , N=N+1; $\par$
 qp= Mod(e,p^2)^{(p-1)}-1; q=component(qp,2)/p; if(q==0, N0=N0+1); $\par$
 S=S+q/p )) ;  M=S/N; print(N,"   ",N0,"         ",M) )$$\}$

\smallskip
On trouve par exemple (dans l'ordre donn\'e par la fonction {\it random})~:

\smallskip
$e = 839$, $\ \ N=50847532$, $N_0 = 2$,  $M= 0.49999204844300$,

$e = 736$, $\ \ N=50847532$, $N_0 = 4$,  $M= 0.50002594430610$,
        
$e = 748$, $\ \ N=50847531$, $N_0 = 2$,  $M= 0.49998373833438$,

$e = 240$, $\ \ N=50847531$, $N_0 = 1$,  $M= 0.50001538843784$.

\begin{rema}\label{remaeps} 
On peut conjecturer que la probabilit\'e ${\rm Pr} \big (q_p(a) \! = \! 0)$ d\'epend d'une fonction de la forme
$ \Frac{1}{x^{1+\epsilon(x,y)}},  \hbox{\ \,  $\epsilon(x,y) > 0$, \ \  pour $x=p$, $y=a$. }$.

Si ${\rm Pr} \big (q_p(a)\!=\!0 \big) = \Frac{1}{p^{1+\epsilon(p,a)}}$ et 
${\rm Pr} \big (q_p(a)\!=\!q \big) = u$ pour tout $q \ne 0$, $u$ ind\'ependant de $q$, alors on obtient 
$(p-1)\,u +\Frac{1}{p^{1+\epsilon(p,a)}} = 1$, d'o\`u 
$u - \Frac{1}{p} = \Frac{1-p^{-\epsilon(p,a)} }{p\,(p-1)}$ et $u$
peu discernable de $\Frac{1}{p}$.
Il resterait \`a tester une fonction $\epsilon(x,y) >0$ telle que $\sm_p  \Frac{1}{p^{1+\epsilon(p,a)}} < \infty$.
Comme me l'a fait remarquer G. Tenenbaum, cette convergence a lieu d\`es que
$\epsilon(x,y) = \Frac{C \, {\rm log}({\rm log}({\rm log}(x))) }{{\rm log}(x)}, \ \ C>1.$
Il reste \`a tenter des statistiques en comparaison avec une probabilit\'e de ce type ($C$ pouvant d\'ependre de $a$)~;
on note que $u-\frac{1}{p} \sim \frac{1}{p^2}$ avec une telle fonction $\epsilon$.
\end{rema}

Ces observations sont tr\`es insuffisantes, mais la probabilit\'e $\frac{1}{p}$ para\^it tout aussi arbitraire car on est dans un cadre multiplicatif (comment factoriser $\wt\Phi_m(a)$) et non dans une structure additive de r\'esidus modulo $p$. 
Or si cette  probabilit\'e \'etait comme envisag\'e ci-dessus, les cons\'equences seraient importantes, y compris pour 
les corps de nombres pour lesquels les $\Delta_p^\theta(\eta)$ sont une g\'en\'eralisation tr\`es canonique des 
$q_p(a) = \Delta_p^1(a)$.

\vspace{-0.4cm}
\section{Conclusion} Nous avons  conscience du fait que les conjectures pr\'ec\'edentes sont particu\-li\`e\-rement t\'em\'eraires, 
mais nous avons essay\'e de donner un maximum de justifications, en particulier par le fait que lorsque les probabilit\'es sont 
au plus en $\Frac{O(1)}{p^2}$, les principes heuristiques de type Borel--Cantelli sugg\`erent un nombre fini de solu\-tions $p$
\`a $\Delta_p^G(\eta) \equiv 0 \pmod p$ et m\^eme aucune solution la plupart du temps puisque la somme des $\frac{1}{p^2}$ est tr\`es petite~:

\centerline{ $\,\sum_{p\geq 2} \frac{1}{p^2} = 0.45$, $\ \sum_{p\geq 100} \frac{1}{p^2} =0.0018$,    
et $\ \sum_{p\geq 10000} \frac{1}{p^2} = 9 \times 10^{-6}$. }

\medskip
De toutes fa\c cons il reste les cas de $p$-divisibilit\'e minimale $p^{\varphi(1)}$ pour ${\rm Reg}_p^G(\eta)$
(cf. D\'efi\-nition \ref{defidec},\,\S\ref{sub10}, ou Remarque \ref{rema27},\,\S\ref{sub34}, ainsi que \S\,\ref{subex}), c'est-\`a-dire lorsqu'un unique caract\`ere $p$-adique  $\theta$ est concern\'e par la congruence $\Delta^\theta_p(\eta) \equiv 0 \pmod p$, que $p$ est totalement d\'ecompos\'e dans le corps des valeurs $C_\chi$ de $\varphi \div \theta$, et que le $G$-module ${\mathcal L}^\theta$ est minimal. 
Rappelons que les $p$ totalement d\'ecompos\'es dans $C_\chi$ sont en nombre infini avec une densit\'e connue 
via le th\'eor\`eme de Dirichlet (tous les nombres premiers si $C_\chi = \Q$ comme dans le cas de $G=D_6$).

\smallskip
Dans ce cas, le ``nombre pr\'evisible'' de solutions $p< x$ est en $O(1) {\rm log}{\rm log}(x)$
sauf si le point de vue du\,\S\ref{subalt} g\'en\'eralis\'e aboutit \`a des constantes, m\^eme tr\`es grandes.

\smallskip
Mais s'il doit y avoir une certaine coh\'e\-rence des math\'ematiques, on peut alors croire que de telles conjectures (a priori inaccessibles \`a l'heure actuelle) sont l\'egitimes, comme la conjecture $ABC$ et ses nombreuses variantes. 
Voir par exemple \`a son sujet le texte de Waldschmidt \cite{W} pour l'importante liste d'applications et cons\'equences qui montrent bien que les aspects diophantiens de la th\'eorie des nombres reposent sur 
 un certain nombre de propri\'et\'es ``structurelles'' plus fortes que celles habituellement accessibles, mais h\'elas difficiles. 

\smallskip
Par exemple, on peut dire que la conjecture de Leopoldt--Jaulent sur la non nullit\'e des r\'egulateurs $p$-adiques semble 
une forme extr\^emement faible de la r\'ea\-lit\'e.

\smallskip
 De fait, ces conjectures suscitent de nombreuses propri\'et\'es, et il convient aussi d'envisager de vastes exp\'erimentations num\'eriques.

\end{document}